\documentclass[12pt]{article}
\usepackage[utf8]{inputenc}
\usepackage{amsmath}
\usepackage{amsfonts}
\usepackage{amssymb}
\usepackage{amsthm}

\usepackage{titlesec}
\titleformat{\subsection}[hang]{\normalfont\bfseries}{\thesubsection}{1em}{}

\titlespacing\section{0pt}{3.5ex plus 0.5ex minus .2ex}{0.3ex plus .2ex}
\titlespacing\subsection{0pt}{2.5ex plus 0.5ex minus .2ex}{0.3ex plus .2ex}
\titlespacing\subsubsection{0pt}{2.5ex plus 0.5ex minus .2ex}{0.3ex plus .2ex}

\usepackage{mathrsfs} %

\usepackage[shortlabels]{enumitem} 

\usepackage[nonewpage]{imakeidx}
 \makeindex[intoc,name=notation,title=Selected notation]
 \makeindex[intoc,name=terminology,title=Selected terminology]

\usepackage{fullpage}
\usepackage{setspace}
\usepackage[pdfusetitle, %
colorlinks,
linkcolor={red!50!black},
citecolor={blue!50!black},
urlcolor={blue!80!black},
bookmarksdepth=subsubsection,
unicode,]%
{hyperref}

\usepackage[alphabetic,lite,msc-links]{amsrefs} %
\usepackage{color}
\usepackage{ulem} %
\usepackage{comment}

\usepackage{marginnote}

\usepackage[all,cmtip]{xy} %

 \usepackage{fancyhdr}
\usepackage{calc}
\usepackage{cleveref}

\numberwithin{equation}{subsection}

\newcommand{\<}{\left\langle}
\renewcommand{\>}{\right\rangle}

\newcommand{\bC}{\mathbb{C}}

\newcommand{\bF}{\mathbb{F}}
\newcommand{\bG}{\mathbb{G}}
\newcommand{\bH}{\mathbb{H}}

\newcommand{\bN}{\mathbb{N}}

\newcommand{\bR}{\mathbb{R}}

\newcommand{\bZ}{\mathbb{Z}}

\newcommand{\cO}{\mathcal{O}}
\newcommand{\cP}{\mathcal{P}}

\newcommand{\fg}{\mathfrak{g}}
\newcommand{\fh}{\mathfrak{h}}

\newcommand{\ft}{\mathfrak{t}}

\newcommand{\sA}{\mathscr{A}}
\newcommand{\sB}{\mathscr{B}}

\newcommand{\sG}{\mathscr{G}}

\newcommand{\sS}{\mathscr{S}}
\newcommand{\sT}{\mathscr{T}}

\DeclareMathAlphabet{\mathpzc}{OT1}{pzc}{m}{it}

\newcommand{\isoarrow}{\stackrel{\sim}{\longrightarrow}}
\newcommand{\ra}{\rightarrow}

\newcommand{\wt}{\widetilde}

\providecommand{\abs}[1]{\left\lvert#1\right\rvert}

\DeclareMathOperator{\Hom}{Hom}			%
\DeclareMathOperator{\depth}{depth}		%
\DeclareMathOperator{\PGL}{PGL}			%
\DeclareMathOperator{\Id}{Id}		  	%
\DeclareMathOperator{\GL}{GL}		  	%

\DeclareMathOperator{\SL}{SL}		  	%
\DeclareMathOperator{\SO}{SO} 			%
\DeclareMathOperator{\SU}{SU} 			%
\DeclareMathOperator{\Gal}{Gal}			%
\DeclareMathOperator{\Aut}{Aut}			%
\DeclareMathOperator{\Out}{Out}			%
\DeclareMathOperator{\End}{End}			%
\DeclareMathOperator{\Ind}{Ind}			%
\DeclareMathOperator{\Sp}{Sp}		  	%
\DeclareMathOperator{\ad}{ad}	  		%
\DeclareMathOperator{\Ad}{Ad}	  		%
\DeclareMathOperator{\Lie}{Lie}  	  %
\DeclareMathOperator{\cind}{c-ind}  	  %
\DeclareMathOperator{\Res}{Res}  	  %
\DeclareMathOperator{\Sym}{Sym}  	  %
\DeclareMathOperator{\val}{val}  	  %

\setlist[enumerate]{label=(\alph*)}

\newcommand\smat[4]{\big[\begin{smallmatrix}#1&#2\\#3&#4\end{smallmatrix}\big]} %

\newcommand\defeq{:=}

\newcommand\mathdef\textsf
\newcommand\scr\mathscr
\newcommand\tn\textnormal
\renewcommand\cal{\mathcal}

\newcommand{\bbC}{\mathbb{C}}
\newcommand{\bbF}{\mathbb{F}}
\newcommand{\bbG}{\mathbb{G}}
\newcommand{\bbH}{\mathbb{H}}

\newcommand{\bbR}{\mathbb{R}}
\newcommand{\bbZ}{\mathbb{Z}}

\newcommand{\sfB}{\mathsf{B}}
\newcommand{\sfG}{\mathsf{G}}
\newcommand{\sfP}{\mathsf{P}}
\newcommand{\sfS}{\mathsf{S}}
\newcommand{\sfT}{\mathsf{T}}
\newcommand{\sfU}{\mathsf{U}}
\newcommand{\sfN}{\mathsf{N}}
\newcommand{\sfV}{\mathsf{V}}
\newcommand{\sfW}{\mathsf{W}}

\newcommand{\der}{\tn{der}}
\renewcommand{\sc}{\tn{sc}}
\newcommand{\sep}{\tn{sep}}
\newcommand{\wtG}{\widetilde{G}}
\newcommand{\wtK}{\widetilde{K}}

\newcommand\Alt{\operatorname{Alt}}
\DeclareMathOperator\fdeg{fdeg}

\DeclareMathOperator\Irr{Irr}

\DeclareMathOperator\Nm{Nm}

\DeclareMathOperator\vol{vol}

\renewcommand\O{\operatorname{O}}

\DeclareMathOperator\Spin{Spin}

\theoremstyle{definition}
\newtheorem{construction}[equation]{Construction}
	\crefname{construction}{construction}{constructions}
\newtheorem{definition}[equation]{Definition}
	\newlist{defenum}{enumerate}{1}
	\setlist[defenum]{label=(\alph*), ref=\theequation(\alph*)}
	\crefalias{defenumi}{definition} 
\newtheorem{notation}[equation]{Notation}
\newtheorem{example}[equation]{Example}
\newtheorem{remark}[equation]{Remark}

\newcommand\GEE{\hyperlink{GE2}{(GE2)}}

\theoremstyle{plain}

\newtheorem{corollary}[equation]{Corollary}
	\newlist{corollaryenum}{enumerate}{1}
	\setlist[corollaryenum]{label=(\alph*), ref=\theequation(\alph*)}
	\crefalias{corollaryenumi}{corollary} 
\newtheorem{fact}[equation]{Fact}
	\crefname{fact}{fact}{facts}
	\newlist{factenum}{enumerate}{1}
	\setlist[factenum]{label=(\alph*), ref=\theequation(\alph*)}
	\crefalias{factenumi}{fact} 
\newtheorem{lemma}[equation]{Lemma}
	\newlist{lemmaenum}{enumerate}{1}
	\setlist[lemmaenum]{label=(\alph*), ref=\theequation(\alph*)}
	\crefalias{lemmaenumi}{lemma} 
\newtheorem{proposition}[equation]{Proposition}

\newtheorem{theorem}[equation]{Theorem}
	\newlist{theoremenum}{enumerate}{1}
	\setlist[theoremenum]{label=(\alph*), ref=\theequation(\alph*)}
	\crefalias{theoremenumi}{theorem} 
\newtheorem{theoremx}{Theorem} %
	\crefname{theoremx}{theorem}{theorems}

\usepackage{booktabs,multicol,tikz-cd,xspace,microtype}

{\renewcommand{\descriptionlabel}%
[1]{\hspace{\labelsep}##1:}
\begin{description}
\setlength\itemsep{1em}}
{\end{description}}

\newcommand\AutZfix{\Aut_{Z}}
\newcommand\GLR[2]{\GL(#1,#2)}

\newcommand\PGLR[2]{\PGL(#1,#2)}

\newcommand\WeilRep{\omega}		%
\newcommand\minus{-}       		%
\newcommand\HeisLift[1]{0\times#1}	%
\newcommand\Kplus{K^+}  	%

\AtEndDocument{\bigskip{\footnotesize%
		\textsc{Universität Bonn,
			Mathematisches Institut,
			Endenicher Allee 60,
			53115 Bonn,
			Germany } \par  
		\textit{E-mail addresses}: \texttt{fintzen@math.uni-bonn.de} and \texttt{schwein@math.uni-bonn.de}
}}

\begin{document}
\author{Jessica Fintzen and David Schwein}
\title{Construction of tame supercuspidal representations \\ in arbitrary residue characteristic }
\date{}%
\maketitle

\begin{abstract}
\noindent
	Let $F$ be a nonarchimedean local field whose residue field has at least four elements. Let $G$ be a connected reductive group over $F$ that splits over a tamely ramified field extension of $F$. We provide a construction of supercuspidal representations of $G(F)$ via compact induction that contains, among others, all the supercuspidal representations constructed by Yu in 2001 (\cite{Yu}), but that also works in residual characteristic two. The input for our construction is described uniformly for all residual characteristics and is analogous to Yu's input except that we do not require our input to satisfy the second genericity condition \GEE\ that Yu imposes.
\end{abstract}

{
	\renewcommand{\thefootnote}{}  %
	\footnotetext{MSC2020: 22E50, 11F27} 
	\footnotetext{Keywords: representations of reductive groups over non-archimedean local fields, supercuspidal representations, $p$-adic groups, Heisenberg--Weil representations}
	\footnotetext{JF was partially supported by NSF Grant DMS-1802234 / DMS-2055230 and a Sloan Research Fellowship. Both authors were partially supported by a Royal Society University Research Fellowship of JF and the European Research Council (ERC) under the European Union's Horizon 2020 research and innovation programme (Grant agreement No.\ 950326).}
}

\clearpage
\tableofcontents
\clearpage

\hspace*{1cm}
\begin{quote}\renewcommand{\thefootnote}{\fnsymbol{footnote}}
\small{\itshape
\textsc{Remark.} It is precisely here that our assumption of $p$ odd makes its impact.
We are dealing with the representation theory
of a $2$-step nilpotent $p$-group. The extra complications in this theory
that arise when $p = 2$ could be handled, but at the expense of a long digression.}
\\\phantom{a}\hfill Roger Howe, 1977\footnote{\cite[p.~447]{Howe},
in a paper constructing supercuspidal representations of $\GL_n(F)$.}
\end{quote}

\section{Introduction}
The construction of supercuspidal representations of $p$-adic groups plays a central role in the representation theory of $p$-adic groups and beyond.
While for $\GL_n$ and its inner forms such a construction is known in full generality (\cite{BK,Secherre-Stevens}), for other families of reductive groups, including classical groups, the existing general constructions \cite{Adler,Yu,Stevens} assume that $p \neq 2$. In the present paper we provide a construction of supercuspidal representations of general tame $p$-adic groups for all $p$. Our construction generalizes Yu's construction (\cite{Yu}) by allowing $p=2$ and, in addition, relaxing a genericity condition imposed by Yu on the input for the construction. In particular, we recover as a special case the supercuspidal representations constructed by Yu, which are all supercuspidal representations if $p$ does not divide the order of the absolute Weyl group of $G$. 
Even in this already known setting, our proof contains new elements. In particular, we do not rely on Gérardin's delicate analysis of the Weil representation in \cite[Theorem~2.4(b)]{Gerardin}.

The reasons that previous authors required $p\neq2$ are subtle and depend on the setting.
Stevens' work for classical groups assumed $p\neq2$ because the Glauberman correspondence (see \cite[(2.1)~Theorem]{Stevens01}) does not apply to involutions of pro-$2$-groups. Yu's work for more general tame $p$-adic groups assumed $p\neq 2$ since he crucially relied on the theory of Heisenberg--Weil representations (see \cite[Section~10]{Yu}).
At the simplest level, this theory does not immediately extend to the case $p=2$ because a factor of $1/2$ appears in Yu's definition of the Heisenberg group over~$\bbF_p$.
But as we will see below, there are much more serious obstructions, which we address in this paper.

To describe our results in more detail, let $F$ be a nonarchimedean local field whose residue field has characteristic~$p$ and cardinality~$q$, and let $G$ be a connected reductive group over $F$
that splits over a tamely ramified extension of~$F$.
Let $\Upsilon = ((G_i)_{1\leq i\leq n+1},x,(r_i)_{1\leq i\leq n},\rho,(\phi_i)_{1\leq i\leq n})$
be an input for Yu's construction,
but where we allow $p=2$ and only require each $\phi_i$ to satisfy the first out of the two genericity conditions that Yu imposes in \cite{Yu}
(see \Cref{sec:input} for more details).

\begin{theoremx}[cf.~\Cref{thmseveralsc}] \label{thm69} 
If $q>2$, then the input $\Upsilon$ gives rise to a finite set of supercuspidal representations,
each of which is a compact induction $\cind_{\widetilde K}^{G(F)}(\sigma)$ of an irreducible representation $\sigma$ from an open, compact-mod-center subgroup $\wt K$ of $G(F)$.
\end{theoremx}
See \Cref{thmseveralsc} for a more precise statement\footnote{In fact, \Cref{thmseveralsc} assumes $q>3$, ultimately because of \Cref{thm27}.
	When $q=3$, although our analysis of the Heisenberg--Weil representation is insufficient to treat this case,
	one can combine \cite{Yu} and \cite{Fi-Yu-works} %
	to construct supercuspidal representations from our slightly more general input~$\Upsilon$.}.
For the reader familiar with Yu's construction, let us mention that the open, compact-mod-center subgroup $\wt K$ lies between the following two subgroups that are built out of Moy--Prasad filtration subgroups of the reductive groups appearing in the input~$\Upsilon$ (the notation is explained in \Cref{sec:notation}): 
\begin{align*}
\Kplus &\defeq G_1(F)_{x,r_1/2}\cdot G_2(F)_{x,r_2/2}\cdots
G_n(F)_{x,r_n/2}\cdot N_G(G_1,G_2,\dots,G_n, G_{n+1})(F)_{[x]} \, ,\\
K &\defeq G_1(F)_{x,r_1/2}\cdot G_2(F)_{x,r_2/2}\cdots
G_n(F)_{x,r_n/2}\cdot G_{n+1}(F)_{[x]} \, .
\end{align*}

The construction of the representation $\sigma$ proceeds in two steps. The first step produces from the given input $\Upsilon$ a unique representation, called $\rho \otimes \kappa^-$, of a subgroup $K^-$ of $K$ with $K/K^-$ being a finite abelian 2-group; see \eqref{thm67} for the precise definition of $K^-$ and \Cref{sec:overview} for an overview of the construction of the representation. The second step produces a representation $\sigma$ of $\wt K$ using Clifford theory, which allows the reader to make choices that we discuss and parameterize in \Cref{sec:choices}. If $G=\GL_n$, or if $G$ is a classical group and $p \neq 2$, or if we are in Yu's setting, i.e., if $p \neq 2$ and each $\phi_i$ satisfies the additional second genericity condition \GEE\ of \cite{Yu}, then $\wt K = K^-=K$ and no choices are required.

Let us give more details of the two steps.

The first step generalizes Yu's construction using the theory of Heisenberg--Weil representations. The challenge is that the theory of Heisenberg--Weil representations as used by Yu is not available if $p=2$.  

Already the Heisenberg $\bF_p$-group itself shows an exceptional feature for $p=2$. While for $p>2$ all Heisenberg $\bF_p$-groups of the same cardinality are isomorphic, for $p=2$ there are two isomorphisms classes of Heisenberg groups of cardinality $2^{1+2n}$ for any $n>1$.
Both classes of groups arise in the construction of supercuspidal representations (see \Cref{thm68}).
However, this quirk is not an obstacle for the construction of supercuspidal representations, as the theory of Heisenberg representations carries over to the setting of $p=2$. That theory has also been already used in the case of the the general linear group; see, for example, \cite[Section~II]{Waldspurger-GLN} and \cite[(7.2.4)~Proposition]{BK}.

On the other hand, the theory of Weil representations, which is crucial in the construction of supercuspidal representations, does not work for $p=2$ in the same way as for $p>2$.
A key difference consists in the structure of the group of automorphisms $\AutZfix(H)$ of $H$ that act trivially on the center of $H$,
where $H$ is a Heisenberg $\bbF_p$-group of order~$p^{1+2n}$.
When $p\neq2$, the group $\AutZfix(H)$ decomposes as a semi-direct product $\AutZfix(H)\simeq \bbF_p^{2n}\rtimes\Sp_{2n}(\bbF_p)$,
and the projective Weil representation of $\Sp_{2n}(\bbF_p)$ admits a linearization,
which Yu used to construct supercuspidal representations.
When $p=2$ such a decomposition as semi-direct product does not exist in general. Instead we have short exact sequence
$1 \to \bbF_2^{2n} \to \AutZfix(H) \to \O_{2n}(\bbF_2) \to 1$,
which is nonsplit if $n\geq3$.
Because $\AutZfix(H)$ contains the group $\bbF_2^{2n}$ of inner automorphisms,
its projective Weil representation does not linearize (see \Cref{thm52}) if $n>0$. 

So when $p=2$, there is no natural ambient group, like $\Sp_{2n}(\bF_p)$ when $p\neq2$,
whose Weil representation we can use to construct supercuspidal representations.
Instead, we prove in \Cref{thm14} that the projective Weil representation linearizes over certain large subgroups of the automorphism group $\AutZfix(H)$,
and we show that we can arrange for all the relevant groups that appear in the construction of supercuspidal representations to map to such linearizing subgroups (\Cref{thm16} and \Cref{thm12}).

A priori two possible linearizations of the (restriction of the) projective Weil representations differ by a character of the linearizing subgroup. Contrary to $\Sp_{2n}(\bF_p)$ with $p\neq2$, whose characters are trivial unless $(n,p)=(1,3)$, our linearizing subgroups do admit non-trivial characters in general. 
In order to obtain a unique representation $\kappa^-$ of $K^-$ using our theory of Heisenberg--Weil representations, we observe that if $p=2$, the Heisenberg representation is self-dual and thus carries an additional real or quaternionic structure. Requiring the Weil representation to preserve this structure pins down the Weil representation up to a character of order two, and pins down $\kappa^-$ uniquely if $q>2$ (\Cref{thm120}).

The second step in the construction of $\sigma$ involves Clifford theory (see \Cref{sec:construction} for details). When $p=2$, this step is needed in some cases to extend a Heisenberg--Weil representation to the full normalizer of a parahoric, i.e., to pass from $K^-$ to $K$, and in general, when $p$ is small but possibly odd, this step is needed when Yu's second genericity condition~\GEE\ fails to hold, i.e., to pass from $K$ to $\wt K$.

Condition \GEE\ is closely related to a natural question in the theory of reductive groups: Given a reductive group $\bG$ over a field~$k$ and a semisimple element $X\in\Lie^*(\bG)(k)$ of the dual Lie algebra, is the centralizer $Z_\bG(X)$ connected? Steinberg studied this question in detail in \cite{Steinberg-torsion}, and showed that the centralizer is always connected unless possibly the characteristic of~$k$ is small, what is called a \mathdef{torsion prime} for the dual root datum of~$\bG$.
(See \Cref{thm90} for a review of this notion.)
In our application, the group $\bG$ is a twisted Levi subgroup of~$G$
and the field $k$ is the residue field of~$F$.
Thus \GEE\ can fail, or equivalently, a certain centralizer can be disconnected, when $p$ is small, for example, when $G=\SL_n$ and $p\mid n$.
This disconnectedness is what creates the difference between $K$ and~$\widetilde K$. 

The Clifford theory featuring in the second step allows the reader to make a choice in the construction, and the supercuspidal representations in the finite set mentioned in \Cref{thm69} correspond to different choices. These additional choices can only be described after $\kappa^-$ is constructed and we therefore found it unnatural to record them as part of the input $\Upsilon$ (see \Cref{thm70}).  Although the quotient $\wt K/K$ (and hence also $\wt K/K^-$) is not always abelian (see \Cref{sec:ge2example}), we prove that $\wt K/K^-$ is a $p$-group (see \eqref{thm67} and \Cref{thm02b}).

Once the construction of $(\wt K, \sigma)$ is achieved, the proof that the resulting representation $\cind_{\widetilde K}^{G(F)}(\sigma)$ is irreducible and supercuspidal follows in rough terms the arguments of \cite{Yu} and \cite{Fi-Yu-works}, though the details require some new ideas.

First, we can no longer rely on  Gérardin's analysis of the Weil representation, in particular \cite[Theorem~2.4(b)]{Gerardin}, as he only covers the Weil representations of $\Sp_{2n}(\bF_p)$ that appear for $p \neq 2$. Consequently, our arguments reprove without using \cite[Theorem~2.4(b)]{Gerardin} that Yu's original construction yields supercuspidal representations.

Second, once we remove the condition that the characters in the input $\Upsilon$ satisfy Yu's second genericity condition \GEE, the proof of supercuspidality requires more complicated arguments.
When \GEE\ fails, the resulting disconnectedness requires us at certain points to replace the reductive groups $G_i$ of the input $\Upsilon$ by disconnected groups with identity component~$G_i$. The problem is that the character $\phi_{i-1}$ of~$G_i(F)$ in the input~$\Upsilon$ need not extend to this disconnected group, invalidating one key step in the old arguments for supercuspidality. We compensate by carefully passing certain results about intertwiners to simply-connected covers, referring the reader to the proof of \Cref{thm02} for details.

\subsection*{Structure of the paper}
After summarizing notation in \Cref{sec:notation},
the main body of the paper is \Cref{sec:heisweil},
which develops the theory of Heisenberg--Weil representations
with a view towards our applications to $p$-adic groups,
and \Cref{sec:construction}, which proves \Cref{thm69}.
In addition, for the convenience of the reader,
we have added to the end of the paper four appendices,
an index of selected notation, and an index of selected terminology.

\subsection*{Acknowledgments}
The authors thank Alexis Langlois-Rémillard and Kazuma Ohara for feedback on an earlier version of this paper.

\section{Notation and conventions} \label{sec:notation}
Let $F$ be a nonarchimedean local field of residue characteristic~$p$
with discrete valuation $\val\colon F \twoheadrightarrow \bbZ\cup\{\infty\}$
and ring of integers~$\cal O_F$.
We denote by $k_F$ the residue field of~$F$ and by $q=|k_F|$ the cardinality of $k_F$.
We fix a separable closure $F^\tn{sep}$ of~$F$ and take all finite separable field extensions of $F$ to lie inside $F^\tn{sep}$.
We denote by $\bar k_F$ the residue field of $F^\tn{sep}$, an algebraic closure of~$k_F$.

Given a group $A$, we denote by $Z(A)$ the center of~$A$ and by $\Irr(A)$ the set of isomorphism classes
of irreducible representations of~$A$.
We write $[a,b] \defeq aba^{-1}b^{-1}$ for the commutator of $a$ and $b$.
Given a subgroup $B$ of~$A$,
let $N_A(B)$ be the normalizer of $B$ in~$A$.
More generally, given subgroups $B_1,\dots,B_n$,
let $N_A(B_1,\dots,B_n) \defeq \bigcap_{i=1}^n N_A(B)$.
We write ${}^aB\defeq aBa^{-1}$ and
given a representation $\pi$ of~$B$,
we write ${}^a\pi$ for the representation $x\mapsto\pi(a^{-1}xa)$ of~${}^aB$.
If in addition $B$ is normal in~$A$, then we denote by
$N_A(\pi)$ the set of $a\in A$ such that ${}^a\pi\simeq \pi$
and by $\Irr(A,B,\pi)$ the set of $\sigma\in\Irr(A)$
such that $\sigma|_B$ contains $\pi$.
Given a set $X$ and an action of $A$ on~$X$,
we write $X^A$ for the set of elements of~$X$ fixed by~$A$,
and given in addition an element $x\in X$,
we write $Z_A(x)$ for the set of $a\in A$ such that $a(x) = x$.

Suppose $k$ is an arbitrary field.
We write $\Gal(\ell/k)$ for the Galois group of
a separable algebraic field extension $\ell/k$.
Given a linear algebraic $k$-group~$H$,
we write $H^\circ$ for the connected component of $H$ containing the identity
and $\pi_0(H)\defeq H/H^\circ$ for the component group of~$H$, a finite algebraic $k$-group.

All reductive groups in this paper are required to be connected unless explicitly stated otherwise.
Given a reductive $k$-group~$G$,
we write $G^\der$ for the derived subgroup of~$G$ and $G^\sc$ for the simply connected cover of $G^\der$. We denote the adjoint quotient of $G$ by $G^{\textnormal{ad}}$ and the image of $G^\sc(k)$ in $G(k)$ by $G(k)^\natural$.\index[notation]{_natural@${}^\natural$}
We denote the Lie algebra of~$G$ either by $\Lie(G)$
or by using lowercase Fraktur letters,
so that $\frak g$ is the Lie algebra of~$G$, for example,
and we denote the dual Lie algebra by $\Lie^*(G)$.
We write $\widehat{G}$ for the Langlands dual group of $G$.
A subgroup $H$ of~$G$ is a \mathdef{twisted Levi subgroup}\index[terminology]{twisted Levi subgroup}
if $H_\ell:=H \times_k \ell$ is a Levi subgroup of a parabolic subgroup of~$G$ for some finite, separable field extension $\ell/k$.
A twisted Levi subgroup $H$ of $G$ is \mathdef{elliptic} if $Z(H)/Z(G)$ is anisotropic.

Given a $k$-torus $T$, we denote by $X^*(T)$ the set of characters of
$T_{k^\tn{sep}} \defeq T\times_k {k^\tn{sep}}$
with action of the absolute Galois group $\Gal(k^\tn{sep}/k)$,
where $k^\tn{sep}$ is a separable closure of~$k$.
For an algebraic group $P$ containing $T$ we write $\Phi(P,T)$ for the set of non-zero weights of $T$ acting on the Lie algebra of $P$, equipped with the action of~$\Gal(k^\tn{sep}/k)$.
In particular, if $T$ is a maximal torus of $G$, then $\Phi(G,T)\subset X^*(T)$
is the absolute root system of $G$ with respect to $T$, equipped with Galois action.
When $T$ is a maximal torus of~$G$,
we write $H_\alpha\defeq d\alpha^\vee(1)\in\Lie(T)(k^\tn{sep})$ for $\alpha\in\Phi(G,T)$, 
and we write $W(G,T)\defeq N_G(T)/T$ for the Weyl group,
a finite algebraic $k$-group.
When $T$ is a maximal split torus of~$G$,
so that $\Phi(G,T)$ is the relative root system,
we denote by $U_\alpha$ the root group for
the set of positive-integer multiples of~$\alpha\in\Phi(G,T)$.
So if $2\alpha\in\Phi(G,T)$, then $U_\alpha$ is nonabelian.
Given a finite field extension $\ell/k$,
we denote by $\Nm_{\ell/k}\colon T(\ell)\to T(k)$ the corresponding norm map.

Let $\wt\bbR\defeq\bbR\cup\{r+ \mid r\in\bbR\}\cup\{\infty\}$ with its usual order,
as in \cite[Section~1.6]{Kaletha-Prasad-BTbook}.
We write $\sB(G,F)$ for the enlarged Bruhat--Tits building of~$G$ over~$F$.
For $r\in\wt\bbR$ and $x\in\sB(G,F)$,
we denote by $\frak g(F)_{x,r}$, $\frak g(F)^*_{x,r}$, $U_\alpha(F)_{x,r}$, and $G(F)_{x,r}$
the respective depth-$r$ Moy--Prasad subgroups at~$x$
of the $F$-points of the Lie algebra~$\frak g$,
its linear dual~$\frak g^*$, the root group~$U_\alpha$,
and the group $G$, where in the last case we assume $r\geq0$.
If $F$ is clear from the context, we might omit it from the notation, e.g., we write $\frak g_{x,r}$ instead of $\frak g(F)_{x,r}$ and $G_{x,r}$ instead of $G(F)_{x,r}$.
If $G^\tn{der}$ is anisotropic, for example, if $G=T$ is a torus,
then we may suppress $x$ from the notation and write
$\frak g_r$, $\frak g^*_r$, and $G(F)_r$.
Let $G(F)^\natural_{x,r}\defeq G(F)^\natural\cap G(F)_{x,r}$.
Given $x\in\sB(G,F)$, we write $[x]$ for the image of~$x$ in the reduced building of~$G$. If a group $H$ acts on the reduced building of $G$, then we denote by $H_{[x]}$ the stabilizer of~${[x]}$ in~$H(F)$.

If a twisted Levi subgroup $H$ splits over a tame extension~$E$,
then there is an admissible embedding of buildings $\sB(H,F)\to\sB(G,F)$ (\cite[Section~14.2]{Kaletha-Prasad-BTbook}).
In general, this embedding is only well-defined up to translation,
but all translations have the same image.
In this paper we will identify $\sB(H,F)$ with its image in $\sB(G,F)$
for some fixed choice of embedding,
and all constructions are independent of this choice.
If $H$ is elliptic and $x\in\sB(H,F)$,
then the images of $x$ in the reduced buildings of $H$ and~$G$
have the same stabilizer in~$H(F)$,
so that the notation $H(F)_{[x]}$ has the same meaning
whether $[x]$ is interpreted with respect to~$H$ or~$G$.

In this paper, the word ``representation'' with no additional modifiers
refers to a smooth complex representation.
However, we will sometimes work with ``$R$-representations'' or ``$R$-linear representations''
for $R\in\{\bbR,\bbC,\bbH\}$, referring to $R$-linear representations on $R$-modules.
See \cpageref{Rreps} for a discussion of these notions,
which can also be viewed as extra structure on an underlying complex representation.
We write $\cind$ for compact induction.
Given an irreducible representation~$\pi$ of~$G(F)$,
we denote by $\depth(\pi)$ the depth of~$\pi$.

Let $\sfV$ be a vector space over a field~$k$.
Given a quadratic form $Q$ on~$\sfV$,
we write $\O(\sfV,Q)$ for the usual orthogonal group,
the elements of $\GL(\sfV)$ stabilizing~$Q$.
We define the subgroup $\SO(\sfV,Q)$ of $\O(\sfV,Q)$
as the kernel of the determinant if $\tn{char}(k)\neq2$
or the kernel of the Dickson invariant if $\tn{char}(k)=2$,
so that $[\O(\sfV,Q):\SO(\sfV,Q)]=2$ if $\sfV\neq0$.
Similarly, given an alternating form $\omega$ on~$\sfV$,
we write $\Sp(\sfV,\omega)$ for the usual symplectic group,
the elements of $\GL(\sfV)$ stabilizing~$\omega$.
We will often drop $Q$ or $\omega$ from the notation
in $\O(\sfV,Q)$, $\SO(\sfV,Q)$ and $\Sp(\sfV,\omega)$
when their presence is clear from context.

\section{Heisenberg--Weil representations} \label{sec:heisweil}

Let $k$ be a field.
The reader is welcome to take $k=\bbF_p$
since this is the only case that will be needed in the construction of supercuspidal representations.
Nonetheless, we allow $k$ to be a general field, or sometimes, a finite field,
because it is no harder to state the results in that setting.
Recall that $p$ is a prime number, including possibly $p=2$,
and $q$ is a positive integer power of~$p$.

\subsection{Heisenberg groups} \label{sec:heisenberg}
In this subsection we explain how to extend the definition of
the Heisenberg group over~$\bbF_p$ for odd $p$ (cf.~Example~\ref{thm28}) to the case $p=2$.
The resulting group has extremely explicit models
(\Cref{thm30} and \Cref{thm28,thm29}),
but to find this group within a $p$-adic group,
we will also characterize it by intrinsic properties
(\Cref{thm39}).

Recall that the \mathdef{exponent}
\index[terminology]{exponent}
of a finite group is the least common multiple
of the orders of its elements.
The following class of finite $p$-groups %
already appeared in the work of
 Hall and Higman (\cite[Section~2.3]{hall_higman56}),
and has been extensively studied and used as a tool by finite group theorists;
see \cite[p.~183 and Chapter~5.5]{Gorenstein} for a textbook treatment. 

\begin{definition}%
A finite $p$-group $P$ is an \mathdef{extraspecial $p$-group}
\index[terminology]{extraspecial $p$-group}
if its center $Z(P)$ has order~$p$ and $P/Z(P)$ is abelian of exponent~$p$.
\end{definition}

Specializing the definition of extraspecial $p$-group very slightly
and allowing the degenerate case $\bbZ/p\bbZ$
yields the version of the Heisenberg group relevant
to the construction of supercuspidal representations.

\begin{definition} \label{thm39}
A \mathdef{Heisenberg $\bbF_p$-group} \index[terminology]{Heisenberg $\bbF_p$-group}
is a finite group $P$ whose center $Z(P)$ has order $p$ and for which $P/Z(P)$ is abelian of exponent at most $p$ and if $p\neq 2$ then also $P$ is of exponent at most $p$.	
\end{definition}

In other words, a Heisenberg $\bbF_p$-group is either $\bbZ/p\bbZ$
or an extraspecial $p$-group, which is in addition required
to have exponent at most~$p$ when $p\neq 2$.
The case $p=2$ requires special care:
In this case we cannot require $P$ to have exponent~$2$ because that would force~$P$ to be abelian,
and hence we would only obtain the group $P =Z(P) \simeq \bZ/2\bZ$.

The ``$\bbF_p$'' appearing in our terminology reflects the fact
that there is a general construction of the Heisenberg group over a field~$k$,
specializing to \Cref{thm39} when $k=\bbF_p$.
We recall this construction to help with computations and comparison with the literature,
though we will more often take the intrinsic viewpoint of \Cref{thm39}.

\begin{construction}[Heisenberg $k$-groups] \label{thm30}
		\addtocounter{equation}{-1}
	\begin{subequations} 
\index[terminology]{Heisenberg $k$-group}
Let $k$ be a field and $\sfV$ a finite-dimensional $k$-vector space.
Given a bilinear form $B\colon \sfV\otimes_k \sfV\to k$,
we can interpret $B$ as an element of $Z^2(\sfV,k)$
and define the resulting extension $\sfV^\sharp_B$
\index[notation]{VsharpB@$\sfV^\sharp_B$} of~$\sfV$ by~$k$.
In other words, $\sfV^\sharp_B$
is the group with underlying set $k\times \sfV$ and multiplication
\[
(a,v)\cdot(b,w) = (a+b+B(v,w),v+w).
\]
The group $\sfV^\sharp_B$ is a \mathdef{Heisenberg $k$-group}
if $Z(\sfV^\sharp_B) = k$, where we identify $k$ with $k \times \{0\}$ from now on, or equivalently,
if the associated alternating form below is nondegenerate:
\begin{equation}\label{thm37}\index[notation]{wB@$\omega_B$}
\omega_B(v,w) \defeq B(v,w) - B(w,v).
\end{equation}
\end{subequations}
\end{construction}

If $q=p^d$ with $d\geq2$, then a Heisenberg $\bbF_q$-group
is a $p$-group but not a Heisenberg $\bbF_p$-group because the center is too large.
Relatedly, our construction of supercuspidal representations will ultimately use Heisenberg $\bbF_p$-groups,
even though the residue field of~$F$ may be larger than~$\bbF_p$.
However, the two notions agree for $k=\bbF_p$.

\begin{lemma} \label{thm51}
Every Heisenberg $\bbF_p$-group is obtained from \Cref{thm30} with $k=\bbF_p$.
\end{lemma}

\begin{proof}
	Let $P$ be a Heisenberg $\bF_p$-group. If $P$ has order $p$, then $P \simeq \sfV^\sharp_B$ for $\sfV$ being a zero-dimensional $\bF_P$-vector space. Hence, we assume the order of $P$ is $p^{2n+1}$ with $n \geq 1$, which implies that $P$ is an extraspecial $p$-group. Therefore $P$ has the following explicit presentation (see, e.g. \cite[p.~160]{Winter}): $P$ has generators $x_1, x_2, \hdots, x_{2n}$ and relations

\hfill	\begin{tabular}{l}
	$x_ix_jx_i^{-1}x_j^{-1}=\begin{cases} z & \textnormal{ if } (i,j)=(2d-1,2d) \textnormal{ for } 1 \leq d \leq n, \\
		  z^{-1} & \textnormal{ if } (i, j)=(2d,2d-1) \textnormal{ for } 1 \leq d \leq n,\\ 
		  1 & \textnormal{ otherwise, } \end{cases}
		  $\\
	$zx_i=x_iz$ for $1 \leq i \leq 2n$, $z^p=1$, 
	and $x_i^p=1$ for $1 \leq i \leq 2n-2$, \\
	 if $p \neq 2$, then $x_{2n-1}^p=x_{2n}^p=1$, and \\
	 if $p=2$, then either (case a) $x_{2n-1}^2=x_{2n}^2=1$  or (case b) $x_{2n-1}^2=x_{2n}^2=z$. 
	 \end{tabular}\hfill\,

Let $\sfV=\bbF_p^{2n}$ with standard basis $\{e_i : 1\leq i\leq 2n\}$.
When $p\neq 2$, or $p=2$ and $P$ satisfies the relations in case a above, then
we define $B$ by
\[
B(e_i,e_j) = \begin{cases}
1 & \tn{if $(i, j)=(2d-1, 2d)$ for some $d$,} \\
0 & \tn{otherwise}.
\end{cases}
\]
When $p=2$ and $P$ satisfies the relations in case b above,
then define $B$ by setting
$B(e_{2d-1},e_{2d}) = 1$ for $1 \leq d < n$,
\[
B(e_{2n-1},e_{2n-1})
= B(e_{2n},e_{2n})
= B(e_{2n-1},e_{2n}) = 1,
\]
and $B(e_i,e_j) = 0$ otherwise.
Sending $x_i$ to $(0, e_i)$ and $z$ to $(1,0)$ defines a surjective group homomorphism $P \twoheadrightarrow V^\sharp_B$, which is an isomorphism since both groups have the same order.
\end{proof}

Although there are many possible bilinear forms $B$ for which $\omega_B$ is nondegenerate,
the classification of Heisenberg $k$-groups is rather simple.

\begin{lemma} \label{thm31}
In the setting of \Cref{thm30},
let $B,B'\colon \sfV\otimes_k \sfV\to k$ be two bilinear forms
whose associated alternating forms $\omega_B$ and $\omega_{B'}$ are nondegenerate.
\begin{lemmaenum}
\item \label{thm31a}
Suppose $\tn{char}(k)\neq 2$. Then $\sfV_B^\sharp\simeq \sfV^\sharp_{\omega_B} \simeq \sfV_{B'}^\sharp$.

\item \label{thm31b}
Suppose $\tn{char}(k)=2$. 
If the quadratic forms $B(v,v)$ and $B'(v,v)$ are $\GL(\sfV)$-conjugate,
then $\sfV_B^\sharp\simeq \sfV_{B'}^\sharp$.
The converse holds if $k=\bbF_2$.
\end{lemmaenum}
\end{lemma}

We refer the reader to \Cref{sec:forms} for a review of the definition
and properties of alternating and quadratic forms in characteristic~$2$.

\begin{proof}
Given a function $f\colon \sfV\to k$ and a linear automorphism $\sigma\in\GL(\sfV)$, the map 
\[
(a,v)\mapsto (a + f(v),\sigma v)
\]
defines an isomorphism $\sfV_B^\sharp\to \sfV_{B'}^\sharp$
as long as the following identity holds:
\[
f(v+w) - f(v) - f(w) = B'(\sigma v,\sigma w) - B(v,w),
\qquad v,w\in \sfV.
\]
If in addition $f\in\Sym^2(\sfV^*)$ is a quadratic form,
then the lefthand side of this expression
is a general symmetric bilinear form when $\tn{char}(k)\neq2$
and a general alternating form when $\tn{char}(k)=2$ (see \Cref{sec:forms}).

When $\tn{char}(k)\neq2$, since the form
\[
(2B-\omega_B)(v,w) = B(v,w) + B(w,v)
\]
is symmetric, $\sfV^\sharp_B\simeq \sfV^\sharp_{2B}\simeq \sfV^\sharp_{\omega_B}$.
But then $\sfV_B^\sharp\simeq \sfV^\sharp_{\omega_B} \simeq \sfV^\sharp_{\omega_{B'}} \simeq \sfV_{B'}^\sharp$
because any two nondegenerate alternating forms on~$\sfV$
are $\GL(\sfV)$-conjugate.

When $\tn{char}(k)=2$,
if $B'(v,v)$ is conjugated to $B(v,v)$ by $\sigma\in\GL(\sfV)$,
then the form $B'(\sigma v,\sigma w) - B(v,w)$ is alternating 
and thus $\sfV_B^\sharp \simeq \sfV_{B'}^\sharp$.
Note that
\(
(a,v)^2 = (B(v,v),0)
\) for any $a \in k$.
Hence,  
if $\tau\colon \sfV^\sharp_B\simeq \sfV^\sharp_{B'}$ is an isomorphism of abstract groups,
then the automorphism of~$\sfV$ induced by~$\tau$
takes the quadratic form~$B(v,v)$ to the quadratic form~$B'(v,v)$.
When $k=\bbF_2$, this induced automorphism of~$\sfV$ is automatically $k$-linear.
\end{proof}

Explicitly, \Cref{thm31} gives the following description of Heisenberg $\bbF_p$-groups.

\begin{example}[Heisenberg $\bbF_p$-groups, odd~$p$] \label{thm28}
Suppose $p\neq 2$.
By \Cref{thm31a},
for every $n\geq1$ there is a unique (up to isomorphism) Heisenberg $\bbF_p$-group
of order $p^{2n+1}$, constructed as follows.
Given a symplectic $\bbF_p$-vector space $(\sfV,\omega)$ of dimension~$2n$,
the group $\sfV_{\omega/2}^\sharp$ is the set-theoretic product $\bbF_p \times \sfV$
with multiplication
\[
(a,v)\cdot(b, w) = (a + b + \tfrac12\omega(v,w),v + w).
\]
\end{example}

\begin{example}[Heisenberg $\bbF_2$-groups] \label{thm29}
Let $(\sfV,Q)$ be a finite-dimensional quadratic space over~$\bbF_2$ of even dimension with $Q$ non-degenerate.
Up to isomorphism, there are two possible isomorphism classes of
$(\sfV,Q)$ when $\dim(\sfV)\geq2$:
the split space, isomorphic to $k^{2n}$ with $Q$ given by \eqref{thm32},
and the nonsplit space, isomorphic to $k^{2n-2}\oplus\ell$, where $\ell/k$ is a quadratic field extension, with $Q$ given by \eqref{thm17}.
By \Cref{thm31b}, there are two isomorphism classes of Heisenberg $\bbF_2$-groups
of every fixed order~$2^{2n+1}$:
one of ``positive type'' for the split form and one of ``negative type'' for the nonsplit form.
In the simplest nontrivial case, order~$8$,
the positive-type group is the dihedral group~$D_8$ of order 8
and the negative-type group is the quaternion group~$Q_8$.
\Cref{thm68} shows that both types of groups are needed
in the construction of supercuspidal representations.

In the finite group theory literature,
extraspecial $2$-groups are described as built up from $D_8$ and~$Q_8$ as follows.
Given Heisenberg $\bbF_2$-groups $P$ and~$Q$, identify $Z(P)$ and $Z(Q)$ with~$\bbF_2$
and define the \mathdef{central product} 
$P\circ Q$\index[notation]{_circ@${\underline{\phantom{P}}}\circ\underline{\phantom{P}}$}
as the quotient $(P\times Q)/Z$
where $Z$ is the kernel of the multiplication map $Z(P)\times Z(Q)\to\bbF_2$.
Then every Heisenberg $\bbF_2$-group $P$ of order $2^{2n+1}$
is isomorphic to a central product
\[
P_1\circ P_2 \circ \cdots \circ P_n
\]
where each $P_i$ is either $D_8$ or~$Q_8$.
Two such groups are isomorphic
if and only if the number of their quaternionic factors has the same parity,
since $D_8\circ D_8\simeq Q_8\circ Q_8$.
The positive-type group is $D_8\circ\cdots\circ D_8$
and the negative-type group is $Q_8\circ D_8\circ\cdots\circ D_8$.
\end{example}

\subsection{$R$-representations}
In this subsection we describe a certain structure of an ``$R$-representation''
carried by every self-dual irreducible representation of a finite group.
This structure plays a key role in linearizing projective Weil representations,
as we explain in more detail in \Cref{sec:weil}.

Let $A$ be a finite group.
Let $(\pi,V)$ be an irreducible complex representation of~$A$
and let $(\pi^*,V^*)$ be the dual representation.
Let $\bbH$ denote the ring of quaternions over~$\bbR$.

Suppose that $\pi$ is irreducible.
Following Serre (\cite[Section~13.2]{Serre77}),
there are the following three mutually exclusive possible situations, indexed by a ring $R\in\{\bbR,\bbC,\bbH\}$
that we call the \mathdef{Frobenius--Schur type} of~$\pi$.%
\index[terminology]{Frobenius--Schur type}%
\footnote{This terminology is nonstandard
but is inspired by the \mathdef{Frobenius--Schur indicator},
which equals $0$, $+1$, or~$-1$ for an irreducible representation
of Frobenius--Schur type $\bbC$, $\bbR$, or~$\bbH$,
respectively.}

\begin{enumerate}[(1)]
\item
$\pi$ is \mathdef{complex}: $\pi\not\simeq \pi^*$,
or equivalently, the character of $\pi$ is not real-valued.
\end{enumerate}
In the remaining two cases $\pi\simeq \pi^*$,
but there are two ways that this can happen,
depending on the sign of the form $V\otimes_\bbC V\to\bbC$
resulting from the isomorphism $\pi\simeq \pi^*$.
\begin{enumerate}[(1),resume]
\item
$\pi$ is \mathdef{real}: the form is symmetric,
or equivalently, there is a representation
defined over~$\bbR$ whose extension of scalars to~$\bbC$ is~$\pi$.

\item
$\pi$ is \mathdef{quaternionic}: 
the form is alternating, or equivalently,
there is a structure of a right $\bbH$-module on~$V$
for which the action of $G$ on~$V$ is $\bbH$-linear.
\end{enumerate}

We index the Frobenius--Schur type of $\pi$ by a ring $R$
because $\pi$ can be repackaged as an $R$-module, as follows.
Given $R\in\{\bbR,\bbC,\bbH\}$,
an \mathdef{$R$-representation}\index[terminology]{R-representation@$R$-representation}
\phantomsection \label{Rreps}
of~$A$ is a right $R$-module~$W$
together with an $R$-linear action of $A$ on~$W$,
or in other words, a homomorphism~$\rho$ from $A$
to the group $\GLR WR$\index[notation]{GzL@$\GLR WR$}
of $R$-linear automorphisms of~$W$.
Then to each $R$-representation $(\rho,W)$ of $A$
we associate as follows a complex representation~$(\pi,V)$ of~$A$:
\begin{enumerate}[(1)]
\item
If $R=\bbC$, then $\pi=\rho$.

\item
If $R=\bbR$, then $V = \bbC\otimes_\bbR W$
with $\pi$ the base change of~$\rho$.

\item
If $R=\bbH$, then $V=W$ with $\bbC$-module structure
pulled back along an $\bbR$-algebra embedding $\bbC\hookrightarrow\bbH$
and $\pi$ is the composition of $\rho$
and $\GLR WR\to\GL(V)$.

\end{enumerate}
If the complex representation $\pi$ is irreducible, then it has Frobenius--Schur type~$R$.

\begin{lemma} \label{LemmaRrep}
	Let $(\rho,W)$ be an irreducible $R$-representation of~$A$
	with associated complex representation $(\pi,V)$.
	Suppose $(\pi,V)$ is irreducible.
	Then the isomorphism class of $\rho$ as an $R$-representation
	is uniquely determined by the isomorphism class of $\pi$ as a complex representation.
\end{lemma}

\begin{proof}
If $R=\bbC$, then there is nothing to prove, so assume $R\in\{\bbR,\bbH\}$.
Then $\GLR WR\subseteq\GL(V,\bC)=\GL(V)$ and it suffices to show that
if two homomorphisms $\rho_1,\rho_2\colon A\to\GLR WR$
whose associated complex representation is irreducible become conjugate in~$\GL(V)$,
then they were already conjugate in $\GLR WR$.

For this, we use Galois descent.
There is a reductive $\bbR$-group $G$
such that $G(\bbR) = \GLR WR$ and $G(\bbC)\simeq \GL(V)$:
if $R=\bbR$, then $G$ is isomorphic to $\GL_{n,\bbR}$ where $n=\dim(V)$,
and if $R=\bbH$, then $G$ is the nonsplit inner form
of a general linear group.

Let $\tn{Transp}(\rho_1,\rho_2)$
be the elements of $\GL(V)$ that conjugate $\rho_1$ to~$\rho_2$.
To complete the proof,
we need to show that this set has a $\Gal(\bbC/\bbR)$-fixed point.
The centralizer $Z_{\GL(V)}(\rho_1)$ of $\rho_1$ in~$\GL(V)$
acts on $\tn{Transp}(\rho_1,\rho_2)$
by right multiplication,
and this action turns the $\Gal(\bbC/\bbR)$-set $\tn{Transp}(\rho_1,\rho_2)$ into a
$Z_{\GL(V)}(\rho_1)$-torsor
in the sense of \cite[Chapter~1, Section~5.2]{Serre02}.
Such torsors are classified by the cohomology set
$H^1(\Gal(\bbC/\bbR),Z_{\GL(V)}(\rho_1))$.
But $Z_{\GL(V)}(\rho_1)=Z(\GL(V)) \simeq \bbC^\times$ by Schur's Lemma, since the complex representation associated to $\rho_1$ is irreducible,
and $\Gal(\bbC/\bbR)$ acts on this group in the usual way,
by complex conjugation.
So the set $H^1(\Gal(\bbC/\bbR),Z_{\GL(V)}) = H^1(\Gal(\bbC/\bbR),\bC^\times)$
is trivial by Hilbert's Theorem~90 and thus the torsor is trivial,
implying that it has a $\Gal(\bbC/\bbR)$-fixed point.
\end{proof}

Next we turn to projective representations.
Let $R\in\{\bbR,\bbC,\bbH\}$ and let $W$ be a finite-dimensional right $R$-module.
Then
\[
Z(\GLR WR) = Z(R)^\times \simeq \begin{cases}
\bbC^\times & \tn{if $R=\bbC$,} \\
\bbR^\times & \tn{if $R\in\{\bbR,\bbH\}$,}
\end{cases}
\]
and we write $\PGLR WR\defeq\GLR WR/Z(R)^\times$.
We define a \mathdef{projective $R$-representation}%
\index[terminology]{projective $R$-representation}
of a finite group~$A$ to be a group homomorphism $\bar\rho\colon A\to\PGLR WR$,
and an \mathdef{$R$-linearization}\index[terminology]{$R$-linearization}
of~$\bar\rho$ to be a group homomorphism $\rho\colon A\to\GLR WR$ lifting~$\bar\rho$.
If $\rho$ and $\rho'$ are two $R$-linearizations of~$\bar\rho$,
then there is a character $\chi\colon A\to Z(R)^\times$
such that $\rho' = \chi\otimes\rho$.
Since $A$ is finite, the character $\chi$
takes values in the maximal compact subgroup
\[
Z(R)^\times_\tn{c} = \begin{cases}
\{ z \in \bC \, | \abs{z}=1\}     & \tn{if $R=\bbC$,} \\
\{\pm1\} & \tn{if $R\in\{\bbR,\bbH\}$.}
\end{cases}
\]
Consequently, linearizing
real or quaternionic projective representations involves less of a choice than 
linearizing complex projective representations. In the first case the linearization is unique
if and only if $A$ has no characters of order two,
but in the second case the linearization is unique
if and only if $A$ is perfect,
which is a much stronger condition.

\begin{lemma} \label{lemmaRlinearization}
Let $R\in\{\bbR,\bbH\}$,
let $A$ be a finite group, let $A_p\subseteq A$ be a Sylow $2$-subgroup,
and let $\bar\rho$ be a projective $R$-representation.
Then $\bar\rho$ has an $R$-linearization if and only if $\bar\rho|_{A_p}$ has an $R$-linearization.
\end{lemma}

This result is similar to \cite[Lemma~1.5]{Gerardin}, but much more general:
For our \crefname{lemmaRlinearization}\JFxxx{JF at DS: I want to just write ``For our lemma there is...'' but not sure how to do this with cref} there is no need to assume that $A$ or $\pi$ have a particular form.

\begin{proof}
The pullback of the short exact sequence 
\[
	\begin{tikzcd}
		1 \rar & \bR^\times \rar & \GLR VR \rar & \PGLR VR \rar & 1
	\end{tikzcd}
\]
via $\bar\rho\colon A\to\PGLR VR$ yields an extension of $A$ by $\bR^\times$ which is split if and only if $\bar\rho$ has an $R$-linearization. Let $c \in H^2(A,\bR^\times)$ be the cocycle class attached to this extension, which is trivial if and only if the extension splits. Note that since $A$ is a finite group, $H^2(A,\bR^\times)=H^2(A,\{\pm 1 \})\oplus H^2(A,\bR^\times_{>0})=H^2(A,\{\pm 1 \})$. Since the restriction map $H^2(A,\{\pm1\})\to H^2(A_p,\{\pm1\})$ is injective (\cite[Chapter~IX, Theorem~4]{Serre79}), we can detect the vanishing of~$c$ by restricting to~$A_p$, and hence if $\bar\rho|_{A_p}$ has an $R$-linearization, then so does~$\bar\rho$.
\end{proof}

\subsection{Heisenberg representations} \label{sec:Heisenberg}

In this subsection we recall the representation theory of Heisenberg $\bbF_p$-groups,
paying special attention to the Frobenius--Schur type (real, complex, or quaternionic) of the Heisenberg representation.
Let $P$ be a Heisenberg $\bbF_p$-group,
let $\psi\colon Z(P)\simeq\bbF_p\to\bbC^\times$ be a nontrivial character,
and let $\sfV_P\defeq P/Z(P)$,\index[notation]{VxP@$\sfV_P$} an $\bbF_p$-vector space.

\begin{lemma}[Stone--von Neumann theorem] \label{thm45}
There is (up to equivalence) a unique irreducible representation $\WeilRep_\psi$ of~$P$
whose restriction to~$Z(P)$ is $\psi$-isotypic.
Moreover, $\dim(\WeilRep_\psi) = \sqrt{|\sfV_P|}$.
\end{lemma}

\begin{proof}
This follows from \cite[Lemma~1.2]{Gerardin} and \Cref{thm51}.
\end{proof}

\begin{definition} \label{DefHeisenbergrep}
	We call the representation $\omega_\psi$ of \Cref{thm45} the \mathdef{Heisenberg representation} of $P$ corresponding to $\psi$.\index[terminology]{Heisenberg representation}\index[notation]{wx@$\WeilRep_\psi$}
\end{definition}

In order to relate the Heisenberg representations of $P$ to the Heisenberg representations of appropriate subgroups of $P$ that are themselves Heisenberg $\bF_p$-groups, and to compute their Frobenius--Schur type, we first introduce some additional notation and make a few observations.

By \Cref{thm51}, a direct calculation, and identifying $Z(P)$ with $\bF_p$, the formulas
\index[notation]{wP@$\omega_P$}
\index[notation]{QP@$Q_P$}
\begin{equation} \label{thm41}
\omega_P(xZ(P), yZ(P)) \defeq [x,z],
\qquad 
Q_P(x Z(P)) \defeq x^2\quad(p=2)
\end{equation}
define a symplectic form on~$\sfV_P$
and, when $p=2$, a nondegenerate quadratic form on~$\sfV_P$. 
The nondegeneracy of those forms follows from observing that if $P=\sfV^\sharp_B$, then under the identification of $\sfV_P$ with $\sfV$, we have $\omega_P=\omega_B$, and, if $p=2$, then $Q_P(v)=B(v,v)$ for $v \in \sfV_P=\sfV$ and $\omega_B=B_{Q_P}$ using the notation of \Cref{sec:forms}. 
Following \Cref{sec:forms}, 
we extend the notions of nondegenerate subspace,
isotropic subspace,\index[terminology]{subspace!isotropic}
polarization, and partial polarization\index[terminology]{partial polarization@(partial) polarization} to~$\sfV_P$,
taking these notions with respect to the nondegenerate alternating form~$\omega_P$
when $p\neq2$ and with respect to the nondegenerate quadratic form~$Q_P$ when $p=2$.

Let $\sfW$ be a subspace of~$\sfV_P$.
A \mathdef{splitting} \index[terminology]{splitting (of $\sfW$)}
of~$\sfW$ (in~$P$) is a subgroup
$\HeisLift{\sfW}$ of~$P$ for which the natural projection
$P\twoheadrightarrow P/Z(P)=\sfV_P$ induces an isomorphism $\HeisLift{\sfW}\isoarrow \sfW$.
The subspace $\sfW$ admits a splitting if and only if $\sfW$ is isotropic,
and all splittings are conjugate under the inner automorphism group $\sfV_P$ of~$P$.
At the opposite extreme, the preimage of~$\sfW$ in~$P$
is a Heisenberg $\bbF_p$-group if and only if $\sfW$ is a nondegenerate subspace.

Let $\sfV_P = \sfV^+\oplus \sfV_0\oplus \sfV^-$ be a partial polarization,
and let $P_0$ be the preimage of~$\sfV_0$ in~$P$.
Let $\WeilRep_\psi$ and $\WeilRep_{0,\psi}$
be the Heisenberg representations of~$P$ and~$P_0$, respectively.
Choose a splitting $\HeisLift{\sfV^+}$ of~$\sfV^+$ in~$P$
and let $\sfV^+\times P_0$ be the internal direct product \phantomsection\label{page:splitting:Heisenberg}
of $\HeisLift{\sfV^+}$ and~$P_0$ in~$P$.
Let $\tn{triv}\boxtimes\WeilRep_{0,\psi}$ denote the inflation
of $\WeilRep_{0,\psi}$ along the resulting projection map $\sfV^+\times P_0 \to P_0$.

\begin{lemma} \label{thm15}
Let $P$ be a Heisenberg $\bbF_p$-group.
With the notation of the paragraph above,
\begin{multicols}{2}
\begin{lemmaenum}
\item\label{thm15a}
$\WeilRep_\psi \simeq \Ind_{\sfV^+\times P_0}^P
(\tn{triv}\boxtimes\WeilRep_{0,\psi})$

\item\label{thm15b}
$(\WeilRep_\psi)^{\HeisLift{\sfV^+}} \simeq \WeilRep_{0,\psi}$.
\end{lemmaenum}
\end{multicols}
\end{lemma}

\begin{proof}
For the first part, by \Cref{thm45}, we know that $\dim(\WeilRep_\psi) = \sqrt{|\sfV_P|}$.
Since
\[
\dim\bigl(\Ind_{\sfV^+\times P_0}^P
(\tn{triv}\boxtimes\WeilRep_{0,\psi})\bigr)
= \sqrt{|\sfV_0|}\cdot|\sfV^-| = \sqrt{|\sfV_P|} = \dim(\WeilRep_\psi)
\]
and this induced representation has central character~$\psi$,
it must be the Heisenberg representation.
For the second part,
we use an identification of $P$ with $\sfV^\sharp_B$ for some $B$ as in \Cref{thm30} that sends $\HeisLift{\sfV^+}$ to $\{0\} \times \sfV^+$.
Then $\{0\} \times \sfV^-$, which we identify with $\sfV^-$ via $(0, v) \mapsto v$, forms a set of coset representatives for $P/(\sfV^+\times P_0)$, and by the first part we can describe $\WeilRep_\psi$
as the space of functions $f\colon \sfV^-\to \sfV_{\WeilRep_{0,\psi}}$ on which $\sfV^\sharp_B$ acts as follows
\[
((a,v^+ + v_0 + v^-)f)(x) = \WeilRep_{0,\psi}(a,v_0) \psi(\omega_B(v^+,x)) f(x + v^-)
\]
where $x\in \sfV^-$, $(a,v_0)\in \sfV_0^\sharp$, $v^+\in \sfV^+$, and $v^-\in \sfV^-$.
Such an $f$ is fixed by $\HeisLift{\sfV^+}$ if and only if $f(x) = 0$ for all $x\neq 0$.
The assignment $f\mapsto f(0)$ is the desired isomorphism.
\end{proof}

If the partial polarization is a polarization,
then \Cref{thm15a} gives a construction of the Heisenberg representation.
Indeed, in this case $\sfV_0=0$ and $P_0=\bbF_p$ if $p\neq 2$ or $P$ has positive type,
and $\dim(\sfV_0) = 2$ and $P_0=Q_8$ if $P$ has negative type,
meaning we can easily construct $\WeilRep_{0,\psi}$ by hand.
We will now use this observation to compute
the Frobenius--Schur type of the Heisenberg representation,
which will ultimately allow us to reduce the number of choices in the construction of supercuspidal representations (see \Cref{thm120}).

\begin{lemma}
The Heisenberg representation is complex if $p\neq 2$,
real if $p=2$ and $P$ is of positive type, and
quaternionic if $p=2$ and $P$ is of negative type.
\end{lemma}

\begin{proof}
Let $\omega_\psi$ denote the Heisenberg representation.
If $p\neq2$, then $\omega_\psi$ is complex because
$\omega_\psi^* \simeq \omega_{\psi^{-1}} \not\simeq \omega_\psi$.
Now suppose $p=2$.
In general, if a complex representation of a subgroup of a finite group is self-dual,
then its induced representation is self-dual of the same Frobenius--Schur type as the original representation.
Using \Cref{thm15a} for a polarization of $\sfV_P$,
we may therefore assume that $P=\bbF_2$ or $P=Q_8$.
Now in the positive type case the Heisenberg representation
$\bbF_2\hookrightarrow \{\pm1 \} \subset \bbC^\times$ is visibly real,
and in the negative type case the Heisenberg representation
can be identified with the tautological embedding
$Q_8\hookrightarrow\bbH^\times=\GL(\bH)$, which is visibly quaternionic.
\end{proof}

\begin{definition}
Let $R\in\{\bbR,\bbC,\bbH\}$ be the Frobenius--Schur type of the Heisenberg
representation $\WeilRep_\psi$ corresponding to~$\psi$.
The \mathdef{Heisenberg $R$-representation}%
\index[terminology]{Heisenberg zR-representation@Heisenberg $R$-representation}
corresponding to~$\psi$ is the irreducible $R$-representation of~$P$
whose associated complex representation is~$\WeilRep_\psi$.
\end{definition}

\subsection{Weil representations} \label{sec:weil}
We remain in the setting of the previous section, i.e., $P$ denotes a Heisenberg $\bF_p$-group, $\psi$ is a nontrivial character of $Z(P)$, and $\WeilRep_\psi$ the corresponding Heisenberg representation. To simplify notation we write $\sfV:=\sfV_P=P/Z(P)$. 

We write $\AutZfix(P)$\index[notation]{AutZ@$\AutZfix(\underline{\phantom{P}})$} for the group of automorphisms of $P$ that act trivially on the center $Z(P)$ of $P$.
(If $p=2$ and $P$ is of positive type, then $\AutZfix(P)$ is isomorphic to Weil's
pseudosymplectic group defined in \cite[Section~31]{Weil64}.)
Given a subgroup $A$ of~$\AutZfix(P)$,
since $\WeilRep_\psi$ is the unique irreducible representation of~$P$
with central character~$\psi$,
the action of $A$ preserves $\WeilRep_{\psi}$
up to isomorphism and thereby gives rise to
 a projective representation of~$A$ on the space underlying the Heisenberg representation,
which we call the \mathdef{projective Weil representation} of~$A$.
\index[terminology]{projective Weil representation}

Our construction of supercuspidal representations
requires us to linearize the projective Weil representation for certain subgroups~$A \subseteq \AutZfix(P)$.
If $\dim V >0$, then a linearization on all of $\AutZfix(P)$ is not possible (see \Cref{thm52} below), so it is crucial to restrict to appropriate subgroups $A$. However, without additional constraints, there is ambiguity in the choice of linearization when we restrict the projective Weil representation to $A$:
The character group of~$A$ acts transitively, by twisting, on the set of linearizations.
Since $A$ is often abelian in our applications, this ambiguity is quite dire. 

For $p>2$, the group $A$ in our application is contained in the image of a splitting of the quotient $\Sp(\sfV,\omega_P)$ of $\AutZfix(P)$,
and the projective Weil representation can be linearized on it by the theory of Weil representations. A choice for a preferred linearization has been made in
\cite[Theorem~2.4(a)]{Gerardin} (see also \Cref{rmk-traditional-Weil-rep}).

For $p=2$, the analogous short exact sequence for $\AutZfix(P)$ in \Cref{thm53} does not split if $\dim V>2$. Instead, to linearize the projective Weil representation and pin down a specific linearization,
we use a special feature of the Heisenberg representation present only when $p=2$:
the structure of an $R$-representation, for $R\in\{\bbR,\bbH\}$.
Linearization is accomplished using a criterion for $R$-linearizability,
\Cref{thm14}, which builds on the abstract criterion \Cref{lemmaRlinearization}.
To pin down a specific linearization, we use that
an $R$-linearization is unique up to a character of order two,
rather than an arbitrary complex character,
combined with the fact that in our application, \Cref{thm120},
the relevant subgroup of $\AutZfix(P)$ does not admit any character of order two. 

Note that the subgroup of inner automorphisms of $P$ fixes the center and is isomorphic to $\sfV=P/Z(P)$. Moreover, given an element of $\AutZfix(P)$, the induced automorphism of $\sfV$ preserves the symplectic form $\omega_P$ and, when $p=2$, also the nondegenerate quadratic form $Q_P$ defined in \eqref{thm41}.
 This leads to the following short exact sequences.

\begin{fact}[{\cite[Theorem~1]{Winter}, \cite[Theorem~1]{Griess73}, see also \cite[Section~1.3]{Blasco}}] \label{thm53}
	\begin{factenum}
	\item[]
		\item \label{thm53a}
		If $p\neq2$, then we have a (noncanonically) split exact sequence
		$$ 1 \ra \sfV \ra \AutZfix(P) \ra \Sp(\sfV, \omega_P) \ra 1.$$
		
		\item \label{thm53b}
		If $p=2$, then we have a short exact sequence  
		$$ 1 \ra \sfV \ra \AutZfix(P) \ra \O(\sfV, Q_P) \ra 1 $$
		that is split if and only if $\dim(\sfV)\leq 2$.
	\end{factenum}
\end{fact}

\begin{construction}\label{rmk-traditional-Weil-rep}
	If $p \neq 2$ and $P=V^\sharp_{\omega_P/2}$, then we have a splitting of $\Sp(\sfV, \omega_P)$ into $\AutZfix(P)$ given by $g \in \Sp(\sfV, \omega_P)$ acting via $(a,v) \mapsto (a, g(v))$, using the presentation described in \Cref{thm28}. The \mathdef{Weil representation} of $\Sp(\sfV, \omega_P)$ was defined in \cite[Theorem~2.4(a)]{Gerardin} as a linearization of the projective Weil representation of this splitting. This linearization is unique unless $\Sp(\sfV,\omega_P)=\Sp_2(\bF_3)$, in which case there are three linearizations and Gérardin singles out one of them. The resulting representation of $\Sp(\sfV, \omega_P) \ltimes V^\sharp_{\omega_P/2}$ that restricts to the Weil representation on the first factor and to the Heisenberg representation on the second factor is called the \mathdef{Heisenberg--Weil representation} of $\Sp(\sfV, \omega_P) \ltimes V^\sharp_{\omega_P/2}$.
\end{construction}

\begin{remark} \label{thm52} 
If $\dim \sfV>0$, then
the projective Weil representation is not linearizable on $\AutZfix(P)$.
If it were, then its restriction to the subgroup of inner automorphisms
 would be linearizable as well.
In other words, 
there would exist a representation
$\pi\colon\sfV\ltimes P \to\GL(V_{\WeilRep_\psi})$
extending the Heisenberg representation $P\to\GL(V_{\WeilRep_\psi})$.
Since $\WeilRep_\psi$ is irreducible,
for every $v\in\sfV$, there would then be a scalar
$c(v)\in\bbC^\times$ such that $\pi(v) = c(v)\WeilRep_\psi(0,v)$.
The fact that $\pi$ and $\WeilRep_\psi$ are homomorphisms
forces the function $c$ to satisfy a certain identity,
which we can use to show that $c(v)$ is contained in the $p$th roots of unity $\mu_p$
and that $c(v)$ gives rise to a splitting of the homomorphism $P\to\sfV$.
This is a contradiction as no splitting exists.
\end{remark}

Let $A$ be a subgroup of $\AutZfix(P)$
and let $R\in\{\bbR,\bbC,\bbH\}$ be the Frobenius--Schur type
of the Heisenberg representation~$\WeilRep_\psi$ of~$P$.
Then for every $a \in A$, by Lemma \ref{LemmaRrep}, the $a$-twist of the Heisenberg $R$-representation is isomorphic to the Heisenberg $R$-representation, and the intertwiner between these two representations is unique up to scaling by $Z(R)^\times$. Hence we obtain a projective $R$-representation of~$A$, which we call
the \mathdef{projective Weil $R$-representation}.%
\index[terminology]{projective Weil zR-representation@projective Weil $R$-representation}

In order to transfer the notion of a Weil representation (see \Cref{rmk-traditional-Weil-rep}) from the Heisenberg group $\sfV_{\omega_P/2}^\sharp$ to an abstract Heisenberg group $P$ in the case of $p \neq 2$, we drop $\omega_P$ from the notation, and use a special isomorphism.
\begin{definition}[{cf.~\cite[Section~10]{Yu}}]
We call an isomorphism $P \ra \sfV^\sharp$ a \mathdef{special isomorphism} if the induced morphism $P/Z(P) \ra \sfV^\sharp/\bF_p=\sfV$ is the identity.
\end{definition}

\begin{definition} \label{defnewHeisenbergWeil}
	Let $A$ be a subgroup of~$\AutZfix(P)$,
	let $R$ be the Frobenius--Schur type of the Heisenberg representation of~$P$,
	and let $\psi$ be a non-trivial character of $Z(P)$.
	\begin{enumerate}
		\item For $p \neq 2$, if there exists a special isomorphism $i: P \ra \sfV^\sharp$ that extends to a group homomorphism
		$$ A \ltimes P \ra \Sp(\sfV) \ltimes \sfV^\sharp $$
		whose restriction to the first factor is the projection $A \ltimes \{1\} \ra \Sp(\sfV) \ltimes \{1\}$ of \Cref{thm53a}, then 
		we call the composition of this morphism with the Heisenberg--Weil representation of $\Sp(\sfV) \ltimes \sfV^\sharp$ attached to the character $\psi \circ (i|_{Z(\sfV^\sharp)})^{-1}$  a \mathdef{Heisenberg--Weil $\bC$-representation} of $A \ltimes P$.
		\item For $p = 2$, we call an $R$-representation of $A \ltimes P$ whose restriction to $P$ is the Heisenberg representation $\WeilRep_\psi$ (if such a representation exists) a \mathdef{Heisenberg--Weil $R$-representation} of $A \ltimes P$.  
	\end{enumerate}
	We call the restriction to~$A$ of a Heisenberg--Weil $R$-representation of $A \ltimes P$ a \mathdef{Weil $R$-representation} of $A$, and we call the complex representation associated to a (Heisenberg--)Weil $R$-representation a \mathdef{(Heisenberg--)Weil representation}\index[terminology]{Heisenberg--Weil representation}\index[terminology]{Weil representation}.
\end{definition}
Note that a Weil-$R$ representation is an $R$-linearization of the projective Weil $R$-representation, and a Weil representation is a linearization of the projective Weil representation. Hence any two Heisenberg--Weil representations of a given group $A \ltimes P$ for a given character $\psi$ differ (up to isomorphism) at most by twisting by an $R$-character of $A$. The following results provide us with a setting in which the Heisenberg--Weil representation is unique.

\begin{lemma} \label{lemma-special-iso-uniqueness}
Suppose $p\neq2$ and $\dim\sfV>0$.  Let $A$ be a subgroup of~$\AutZfix(P)$ and
let $f_1, f_2 : A \ltimes P \ra \Sp(\sfV) \ltimes \sfV^\sharp $ be two group homomorphisms whose restriction to $A \ltimes \{1\}$ factors through $\Sp(\sfV) \ltimes \{1\}$ as the projection from \Cref{thm53a} and whose restriction to $\{1 \} \ltimes P$ provides a special isomorphism with $\{1\} \ltimes \sfV^\sharp$. Then there exists $h \in \{1 \} \ltimes P$ such that $f_2$ is the composition of conjugation by $h$ with $f_1$.
\end{lemma}
\begin{proof}
	The composition $f_1 \circ f_2^{-1}$ is an automorphism of $f_2(A) \ltimes \sfV^\sharp$ that induces the identity on the quotient $\sfV$ and hence fixes $\{1 \} \ltimes Z(\sfV^\sharp)=[\{1 \}  \ltimes \sfV^\sharp, \{1 \}  \ltimes \sfV^\sharp]$. Therefore there exists $w \in \sfV$ such that $(f_1 \circ f_2^{-1})(1,a,v)=(1,a+\omega_P(w,v), v)=(1,0,w)(1,a,v)(1,0,w)^{-1}$ for $(a, v) \in \sfV^\sharp$. Let $g \in \Sp(\sfV)$ be in the image $A$. Then for $(a, v) \in \sfV^\sharp$, we have
	\begin{eqnarray*}
		(1,a+\omega_P(w,gv),gv)&=& (f_1 \circ  f_2^{-1})(1, a, gv) =  (f_1 \circ  f_2^{-1})\left((g,1,1)(1, a, v)(g^{-1},1,1)\right) \\
		&=& (g,1,1)\bigl(( f_1 \circ  f_2^{-1})(1, a, v)\bigr)(g^{-1},1,1) \\
		&=&(g,1,1)(1,a+\omega_P(w,v),v)(g^{-1},1,1)=(1,a+\omega_P(w,v),gv).
\end{eqnarray*}
Hence $\omega_P(w,v)=\omega_P(w,gv)=\omega_P(g^{-1}w,v)$ for all $v$, which implies that $gw=w$. Therefore $f_1$ is the composition of $f_2$ with the conjugation by $(1,0,w) \in \Sp(\sfV) \ltimes \sfV$. Therefore $h=({f_2}|_{\{1\} \ltimes P})^{-1}\big((1,0,w)^{-1}\big) \in \{1\} \ltimes P$ satisfies the desired property.
\end{proof}

\begin{corollary}\label{cor-weil-uniqueness}
	Let $A$ be a subgroup of~$\AutZfix(P)$  and suppose that $A \ltimes P$ admits a Heisenberg--Weil representation $\WeilRep_\psi$ attached to a character $\psi$. If $p=2$, we also assume that $A$ has no characters of order two. Then $\WeilRep_\psi$ is unique.
\end{corollary}
\begin{proof}
	Since any two Heisenberg--Weil representations differ (up to isomorphism) only by twisting by an $R$-character of $A$, and since $R\in \{ \bR, \bH\}$ if $p=2$, the assumption implies that $\WeilRep_\psi$ is unique if $p=2$. If $p \neq 2$, then the result follows from combining \Cref{lemma-special-iso-uniqueness} and \Cref{defnewHeisenbergWeil}. 
\end{proof}

The assumption of the above \namecref{cor-weil-uniqueness}
 will be satisfied in our application to representations of $p$-adic groups (see \Cref{thm120}).

To finish this subsection,
we give a criterion for when the projective Weil $R$-representation admits an $R$-linearization.
We recall from \Cref{sec:heisenberg}, page~\pageref{page:splitting:Heisenberg},
that if $\sfV = \sfV^+\oplus \sfV_0 \oplus \sfV^-$ is a partial polarization,
the preimage $P_0$ of $\sfV_0$ in $P$ is a Heisenberg $\bF_p$-group. If $\HeisLift{\sfV^+}$ is a splitting of~$\sfV^+$ in $P$, then we write $\sfV^+\times P_0$  for the internal direct product of $\HeisLift{\sfV^+}$ and $P$, which is the preimage of $\sfV^+\oplus \sfV_0$ in $P$.

\begin{definition} \label{thm20}
Let $\sfV^+\subseteq \sfV$ be an isotropic subspace
and $\HeisLift{\sfV^+}$ a splitting of~$\sfV^+$ in $P$.
Define $\cal P(\HeisLift{\sfV^+})$\index[notation]{P@$\cal P$} to be the subgroup of $\AutZfix(P)$
consisting of the elements $g$ such that
\begin{enumerate}
\item
$g(\HeisLift{\sfV^+}) = \HeisLift{\sfV^+}$, and

\item
$g(x)\cdot x^{-1}\in\HeisLift{\sfV^+}$ for all
$x\in \sfV^+\times P_0$.
\end{enumerate}
\end{definition}

\begin{lemma} \label{thm14} 
Let $\psi\colon Z(P) \to\bbC^\times$ be a nontrivial character,
let $R\in\{\bbR,\bbC,\bbH\}$ be the Frobenius--Schur type of the Heisenberg representation of~$P$
corresponding to~$\psi$,
and let $A$ be a subgroup of $\AutZfix(P)$.
Suppose there is a Sylow $p$-subgroup $A_p$ of~$A$,
an isotropic subspace $\sfV^+$ of~$\sfV$,
and a splitting $\HeisLift{\sfV^+}$ of $\sfV^+$ such that 
\[
A_p\subseteq\cal P(\HeisLift{\sfV^+}).
\]
Then the Heisenberg $R$-representation 
of $P$ extends to an
$R$-representation of $A\ltimes P$. 
\end{lemma}

\begin{proof} 
By \cite[Lemma~1.5]{Gerardin} when $p\neq2$
 and by \Cref{lemmaRlinearization} when $p=2$, it suffices to show that the Heisenberg $R$-representation $\WeilRep^R_\psi$
 extends to $A_p \ltimes P$.
Write $\cal P \defeq \cal P(\HeisLift{\sfV^+})$.
Extend the Heisenberg $R$-representation
$\WeilRep^R_{0,\psi}$ of $P_0$ to the $R$-representation
$\pi$ of $\cal P\ltimes(\sfV^+\times P_0)$ defined by the formula
\[
\pi\colon (p,v,x)\mapsto \WeilRep_{0,\psi}(x),
\qquad p\in\cal P,\, v\in \HeisLift{\sfV^+},\, x\in P_0.
\]
By the definition of $\cal P$,
this formula defines a homomorphism:
$p(v,x)p^{-1} = (v+w,x)$ for some $w\in \HeisLift{\sfV_+}$,
and then $\pi(1,v+w,x) = \pi(1,v,x)$ because $\HeisLift{\sfV_+}\subseteq\ker(\pi)$.
Using \Cref{thm15a}, we deduce that the restriction to $P$ of the induced representation
$\Ind_{\cal P\ltimes(\sfV^+\times P_0)}^{\cal P\ltimes P}(\pi)$
is isomorphic to $\WeilRep^R_\psi$.
Restricting $\Ind_{\cal P\ltimes(\sfV^+\times P_0)}^{\cal P\ltimes P}(\pi)$ to $A_p \ltimes P$ yields therefore an extension of $\WeilRep^R_\psi$ to an $R$-representation.
\end{proof}

\section{Construction of supercuspidal representations} \label{sec:construction}

Let $F$ be a nonarchimedean local field. Let $G$ be a connected reductive group that splits over a tamely ramified extension of $F$.

\subsection{The input} \label{sec:input}
The input to our construction of supercuspidal representations is analogous to the input that Yu (\cite{Yu}) uses, but it allows $p=2$ and removes the second genericity assumption \GEE\ imposed by Yu in \cite[Section~8]{Yu} that we recall below. We follow the conventions in \cite{Fi-Yu-works}; see \cite[Remark~2.4]{Fi-Yu-works} for a comparison of conventions.

Throughout the paper we will use the following notion of generic elements,
which is slightly weaker than what has previously appeared in the literature
due to the absence of~\GEE.

\begin{definition}[{cf.~\cite[Definition~3.5.2]{Fintzen-IHES}}] \label{thm44}
\index[terminology]{generic@generic of depth $r$}
Let $G'$ be a connected reductive $F$-group
and let $H\subseteq G'$ be a twisted Levi subgroup
that splits over a tamely ramified field extension of $F$.
Let $x\in\sB(H,F)$, and let $r \in \bR_{>0}$.
\begin{enumerate}
\item
An element $X\in\Lie^*(H)^H(F)$ is \mathdef{$(G',H)$-generic of depth~$r$}
if it satisfies the following two conditions.
\begin{itemize}
\item[\textbf{(GE0)}]
$X \in \Lie^*(H)_{x,r} \smallsetminus \Lie^*(H)_{x,r+}$
for some (equivalently, every) $x \in \wt\sB(H, F)$.
\item[\textbf{(GE1)}]
$\val(X(H_\alpha))=r$
for some (equivalently, every) maximal torus $T$ of $H$
and every $\alpha \in \Phi(G', T) \smallsetminus \Phi(H,T)$.
\end{itemize}

\item
A character $\phi$ of~$H(F)$ is \mathdef{$(G',H)$-generic (relative to~$x$) of depth~$r$}
if $\phi$ is trivial on $H(F)_{x,r+}$
and the restriction of $\phi$ to $H(F)_{x,r}$
is realized by an element $X$ of $\Lie^*(H)^H(F)$
that is $(G',H)$-generic of depth~$-r$, i.e., $\phi|_{H_{x,r}}$ is given by composing 
$H_{x,r}/H_{x,r+} \simeq \fh_{x,r}/\fh_{x,r+}$ with $\Lambda \circ X$ for a fixed additive character $\Lambda: F \ra \bC^\times$ that is nontrivial on $\cO_F$, but trivial on the maximal ideal of $\cO_F$. 
\end{enumerate}
\end{definition}

Notably, these conditions do not require Yu's condition~\GEE.
However, since the fine details of \GEE\ will be relevant later,
we review the definition here.

In the setting of \Cref{thm44},
suppose $X$ is $(G',H)$-generic of depth~$r$
and let $T$ be a maximal torus of~$H$.
We denote by $X_\ft$ the restriction of $X$ to $\ft(F^\tn{sep})$
and choose an element $\varpi_r$ of valuation $r$ in $F^\tn{sep}$.
Then, under the identification of $\ft^*(F^\tn{sep})$
with $X^*(T) \otimes_\bZ F^\tn{sep}$,
the element $\frac{1}{\varpi_r} X_\ft$ is contained in
$X^*(T) \otimes_\bZ \cO_{F^\tn{sep}}$,
and we denote its image under the surjection
$X^*(T) \otimes_\bZ \cO_{F^\tn{sep}} \twoheadrightarrow
X^*(T) \otimes_\bZ \bar k_F$ by $\overline X_\ft$.
Now we can state \hypertarget{GE2}{condition}~\GEE: \phantomsection\label{GE2}
\begin{description}[leftmargin=\widthof{\textbf{(GE2)}}+\labelsep]
\item[\textbf{(GE2)}]
The centralizer of $\overline X_\ft$
in $W(G',T)(\bar k_F)$ is $W(H,T)(\bar k_F)$
for some (equivalently, every) maximal torus $T$ of $H$.
\end{description}
Moreover, if $\phi$ is a character of~$H(F)$ that is $(G',H)$-generic relative to~$x$ of depth~$r$, then we say that $\phi$ satisfies \GEE\ if the restriction of $\phi$ to $H(F)_{x,r}$
is realized by a $(G', H)$-generic element of $\Lie^*(H)^H(F)$ that satisfies~\GEE.

If $p$ is a torsion prime for the Langlands dual group~$\widehat {G'}$,
or in other words, for the dual root datum of~$G'$,
then $(G', H)$-generic characters always satisfy condition~\GEE\ (\cite[Lemma~8.1]{Yu}).
We refer the reader to Steinberg's article \cite{Steinberg-torsion}
for a full discussion of the notion of a torsion prime,
but for convenience, we recall the following properties of torsion primes.
Let $\pi_1(\bG)$ be the algebraic fundamental group of a reductive group~$\bG$
(see \cite[Section~11.3]{Kaletha-Prasad-BTbook}),
a finitely generated abelian group,
and let $\pi_1(\bG)_\tn{tors}$ be its torsion subgroup.

\begin{lemma} \label{thm90}
Let $p$ be prime, let $k$ be a field,
and let $\bG$ be a reductive $k$-group.
\begin{enumerate}
\item
$p$ is torsion for $\bG$
if and only if $p \mid |\pi_1(\bG)_\tn{tors}|$
or $p$ is torsion for the root system of~$\bG$.

\item
The torsion primes~$p$ of an irreducible root system~$\Phi$ are as follows:
\begin{center}
\begin{tabular}{|c|c|c|c|c|} \hline
$\Phi$	& $A_n$, $C_n$ & $B_n$, $D_n$, $G_2$ & $E_6$, $E_7$, $F_4$ & $E_8$ \\ \hline
$p$	& \tn{none}    & $2$                 & $2$, $3$,           & $2$, $3$, $5$ \\ \hline
\end{tabular}
\end{center}

\item
Let $\bG'$ be a Levi subgroup of~$\bG$.
If $p$ is torsion for $\bG'$, then $p$ is torsion for~$\bG$.
\end{enumerate}
\end{lemma}

\begin{proof}
These claims are proved in \cite{Steinberg-torsion}.
In more detail, the first two parts follow from 2.5 and 1.13 of \cite{Steinberg-torsion}, respectively.
For the third part, \cite[1.19]{Steinberg-torsion} shows that if $p$ is torsion for the root system of $\bG'$,
then it is torsion for the root system of~$\bG$,
and \cite[2.11]{Steinberg-torsion} shows that $\pi_1(\bG')_\tn{tors}$
is a subgroup of $\pi_1(\bG)_\tn{tors}$.
Now the third part follows from the first part.
\end{proof}

\begin{example}[Failure of \GEE]
Let $p=2$, let $G=\SL_2$,
and let $T$ be an unramified, elliptic maximal torus.
Then $T(F)$ is the norm-one elements
in an unramified quadratic field extension $E/F$.
A character $\phi\colon T(F)\to\bbC^\times$
is $G$-generic of depth~$r$ if and only if $\depth(\phi) = r$.
By direct calculation,
or an argument using methods from the proof of \Cref{thm89},
we see that $N_G(T)(F)/T(F)\simeq\bbZ/2\bbZ$.
Hence $N_T(G)(F)/T(F)$ acts on $T(F)$ by inversion.
So if $\phi$ has positive depth and order two,
then $\phi$ will be $(G, T)$-generic but not satisfy \GEE.
\end{example}

The input to our construction is a tuple\index[notation]{Y@$\Upsilon$} (cf.~\cite[Section~2.1]{Fi-Yu-works})
\[
\Upsilon = ((G_i)_{1\leq i\leq n+1},x,(r_i)_{1\leq i\leq n},\rho,(\phi_i)_{1\leq i\leq n})
\]
for some non-negative integer $n$, where  
\begin{enumerate}[(a)]
	\item $G=G_1 \supseteq G_2 \supsetneq G_3 \supsetneq \hdots \supsetneq G_{n+1}$ are twisted Levi subgroups of $G$ that split over a tamely ramified extension of $F$,
	\item  $x \in \sB(G_{n+1},F)\subset \sB(G,F)$, 
	\item $r_1 > r_2 > \hdots > r_n >0$ are real numbers,
	\item $\rho$ is an irreducible representation of $(G_{n+1})_{[x]}$ that is trivial on $(G_{n+1})_{x,0+}$, 
	\item $\phi_i$, for $1 \leq i \leq n$, is a character of $G_{i+1}(F)$ of depth $r_i$, 
\end{enumerate}
satisfying the following conditions 
\begin{enumerate}[(i)]
	\item  $G_{n+1}$ is elliptic in~$G$, i.e.,\ $Z(G_{n+1})/Z(G)$ is anisotropic,
	\item the image of the point $x$ in $\sB(G_{n+1}^{\der},F)$ is a vertex,
	\item $\rho|_{(G_{n+1})_{x,0}}$ is a cuspidal representation of $(G_{n+1})_{x,0}/(G_{n+1})_{x,0+}$,
	\item $\phi_i$ is $(G_i, G_{i+1})$-generic relative to $x$  of depth $r_i$ for all $1 \leq i \leq n$.
\end{enumerate}
For brevity, we will refer to such an object
$\Upsilon$ as a \mathdef{supercuspidal $G$-datum}, and we will fix such a datum from now on.
\index[terminology]{supercuspidalGdatum@supercuspidal $G$-datum}

\subsection{Overview of the construction} \label{sec:overview}
We will now construct several objects out of a supercuspidal $G$-datum~$\Upsilon$,
culminating in a supercuspidal representation.
The dependence on~$\Upsilon$ is implicit, not reflected in the notation.

Define the open, compact-mod-$Z(G)$ subgroups\index[notation]{K@$K$}\index[notation]{Kt@$\Kplus$}
\begin{align*}
\Kplus &\defeq G_1(F)_{x,r_1/2}\cdot G_2(F)_{x,r_2/2}\cdots
G_n(F)_{x,r_n/2}\cdot N_G(G_1,G_2,\dots,G_n, G_{n+1})(F)_{[x]} \, ,\\
K &\defeq G_1(F)_{x,r_1/2}\cdot G_2(F)_{x,r_2/2}\cdots
G_n(F)_{x,r_n/2}\cdot G_{n+1}(F)_{[x]} \, .
\end{align*}
Then $K$ is a normal, finite-index subgroup of~$\Kplus$.
Note that our $K$ is denoted by $\wt K$ in \cite[Section~2.5]{Fi-Yu-works}, but we want to avoid too many tildes and indices.

\begin{lemma}
Let $G'$ be a connected reductive $F$-group
and let $H\subseteq G'$ be a twisted Levi subgroup
that splits over a tamely ramified extension of $F$.
Let $x\in\sB(H,F)$ and
let $\phi$ be a character of~$H$ of depth~$r$.
Then there exists a unique character $\hat\phi_{(G',x)}$%
\index[notation]{pzzh@$\hat\phi_{(G',x)}$}
of $H(F)_{[x]}\cdot G'(F)_{x,r/2+}$
such that $\hat\phi_{(G',x)}$ and $\phi$ agree on
$H(F)_{[x]}$ and $(H,G')_{x,r+,r/2+}\subseteq\ker(\hat\phi_{(G',x)})$.
\end{lemma}

\begin{proof}
This follows from the argument at the beginning of \cite[Section~4]{Yu}.
\end{proof}

In particular, for each $1\leq i\leq n$
we have a character $\hat\phi_i\defeq(\hat\phi_i)_{(G,x)}$\index[notation]{pzzh@$\hat\phi$}
of $(G_{i+1})_{[x]}\cdot G_{x,r_i/2+}$ that extends
the restriction of~$\phi$ to $(G_{i+1})_{[x]}$.%
\footnote{Our character $\hat\phi$ is defined
on a slightly larger subgroup than Yu's character $\hat\phi$,
which was defined on the subgroup
$(G_{n+1})_{[x]}\cdot(G_{i+1})_{x,0}\cdot G_{x,r_i/2}$.
There is little risk of confusion, however,
because our character extends Yu's character.}

Our goal is to construct a representation $\sigma$
of a certain group $\wt K \defeq N_{\Kplus}(\rho\otimes\kappa)$
contained between $K$ and $\Kplus$.
The irreducible supercuspidal representation is then $\cind_{\wt K}^{G(F)}(\sigma)$.
If the characters in the input $\Upsilon$ satisfy Yu's additional condition \GEE,
then $\wt K=K$ (see \Cref{thm02b}).
The construction takes two steps,
which we briefly summarize before describing them in more detail.

First, we define a certain normal subgroup $K^\minus$ of $K$,
for which the quotient $K/K^\minus$
is an abelian $2$-group if $p=2$ and is trivial otherwise (see \eqref{thm67}).
Using the Heisenberg--Weil representation,
we construct an irreducible representation $\kappa^\minus$ of $K^\minus$
(see \Cref{lemma:kappaminus}).

Second, we make two choices:
an irreducible representation~$\kappa$ of $K$
whose restriction to $K^\minus$ contains $\kappa^\minus$,
and an irreducible representation $\sigma$ of $\widetilde K \defeq N_{K^+}(\rho\otimes\kappa)$
whose restriction to $K$ contains $\rho\otimes\kappa$.
These two choices can be studied using Clifford theory,
and we reflect on the choices in \Cref{sec:choices}.

\textbf{Step 1: Heisenberg--Weil representation.}
The first step uses the theory of Heisenberg--Weil representations
that features in Yu's work (\cite{Yu}) in the case $p \neq 2$, 
but which was not available before this paper in the case $p=2$.

For $\tilde r,\tilde r'\in\wt\bbR\setminus\{\infty\}$
with $\tilde r\geq \tilde r'\geq \tilde r/2 > 0$,
let $(G_i)_{x,\tilde r,\tilde r'}=(G_{i+1},G_i)(F)_{x,\tilde r,\tilde r'}$
as in \cite[Section~2.5]{Fi-Yu-works}.\index[notation]{Gir@$(G_i)_{x,\tilde r,\tilde r'} = (G_{i+1},G_i)(F)_{x,\tilde r,\tilde r'}$}
In \Cref{thm33} we will show that the group\index[notation]{Vi@$\sfV_i^\natural$}
\[
\sfV_i^\natural \defeq (G_i)_{x,r_i,r_i/2}/((G_i)_{x,r_i,r_i/2+}\cap\ker(\hat\phi_i))
\]
is a Heisenberg $\bbF_p$-group. 
Let $(\WeilRep_i,V_{\WeilRep_i})$
denote the Heisenberg representation of $\sfV_i^\natural$
with central character $\phi_i|_{(G_i)_{x,r_i,r_i/2+}}$
(\Cref{DefHeisenbergrep}).

Define the group\index[notation]{Km@$K^\minus$}
\begin{equation} \label{thm67}
K^\minus \defeq (G_1)_{x,r_1,r_1/2}\cdots(G_n)_{x,r_n,r_n/2}
G_{n+1}(F)_{[x]}^\minus
\end{equation}
where if $p\neq2$,
then $(G_{n+1})^\minus_{[x]} \defeq (G_{n+1})_{[x]}$\index[notation]{Gm@$(G_{n+1})^\minus_{[x]}$},
and if $p=2$, then $(G_{n+1})^\minus_{[x]}$ is defined to be the kernel
of the projection map from $(G_{n+1})_{[x]}$ to the finite abelian $2$-group
\begin{equation} \label{thm10}
((G_{n+1})_{[x]}/(Z(G(F))\cdot(G_{n+1})_{x,0}))
\otimes_{\bZ}\bbZ_{(2)}.
\end{equation}
The left factor is a finite abelian group by
\cite[Corollary~11.6.3]{Kaletha-Prasad-BTbook}.

The representation $\kappa^\minus$\index[notation]{kzzm@$\kappa^\minus$}
of $K^\minus$ will have underlying vector space
$V_{\kappa^\minus} \defeq \bigotimes_{i=1}^n V_{\WeilRep_i}$.
To define $\kappa^\minus$ we give an action of each factor of $K^\minus$
on each vector space $V_{\WeilRep_i}$,
form the tensor product of the actions to produce an action of each factor of $K^\minus$ on $V_{\kappa^\minus} $,
and then check that the actions of the different factors of $K^\minus$ are compatible,
so that they descend to a morphism $\kappa^\minus\colon K^\minus\to\GL(V_{\kappa^\minus})$ (see \Cref{lemma:kappaminus}).
More precisely, let $1\leq i, j\leq n$.
Then the factor $(G_i)_{x,r_i,r_i/2}$
acts on the space $V_{\WeilRep_j}$ when $i\neq j$
by the character $\hat\phi_j|_{(G_i)_{x,r_i,r_i/2}}$,
and when $i=j$ by the Heisenberg representation $\WeilRep_i$ of $\sfV_i^\natural$
with central character $\hat\phi_i|_{(G_i)_{x,r_i,r_i/2+}}$.
As for the factor $(G_{n+1})_{[x]}^\minus$,
we let $(G_{n+1})_{[x]}^\minus$ act on $V_{\WeilRep_i}$ 
via $\phi_i\otimes\WeilRep_i$, where $\WeilRep_i$ denotes the restriction of a (pull back of a) Weil--Heisenberg representation as in \Cref{notation:Weil} (see also \Cref{thm12}, \Cref{defnewHeisenbergWeil}, and \Cref{thm120}). While Weil representations in general are only uniquely defined up to twisting by an order-two character (cf.~\Cref{cor-weil-uniqueness}), the resulting representation $\WeilRep_i$ is uniquely defined when $q>2$ because it is inflated from a group with no characters of order two (see \Cref{thm120}).

\textbf{Step 2: Clifford theory.}
Recall from \Cref{sec:notation} that
if $A$ is a group, $B$ is a normal subgroup of~$A$,
and $\pi$ is a representation of~$B$,
then we write
$\Irr(A,B,\pi)$ for the set of $\sigma\in\Irr(A)$
whose restriction to~$B$ contains $\pi$.

First, let $\kappa\in\Irr(K,K^\minus,\kappa^\minus)$\index[notation]{kzz@$\kappa$}.
If $p\neq 2$, then $K^\minus = K$ and $\kappa=\kappa^\minus$.
By \Cref{thm72a}, the character group of $K/K^\minus$
acts transitively, by twisting, on the set of such~$\kappa$.
We may now inflate the representation $\rho$
from $G_{n+1}(F)_{[x]}$ to~$K$ by asking the inflation to be trivial on $G_1(F)_{x,r_1/2}\cdot G_2(F)_{x,r_2/2}\cdots
G_n(F)_{x,r_n/2}$. We denote this inflation by $\rho$ as well, 
and we form the tensor product $\rho\otimes\kappa$ of the inflation with~$\kappa$.

Second, let $\sigma\in\Irr(\wt K,K,\rho\otimes\kappa)$, where we recall that $\wt K \defeq N_{\Kplus}(\rho\otimes\kappa)$.%
\index[notation]{sigma@$\sigma$}
By \Cref{thm72b}, these $\sigma$ are in bijection with
the irreducible representations of the intertwining algebra 
\(
\End_{\wt K}\bigl(\Ind_K^{\wt K}(\rho\otimes\kappa)\bigr).
\)

In \Cref{thmseveralscC} we will show that the representation
$\cind_{\wt K}^{G(F)}(\sigma)$ is irreducible and supercuspidal if $q>3$.
The representations $\cind_{\wt K}^{G(F)}(\sigma)$ for varying $\sigma$ can also be recovered as the irreducible subrepresentations of $\cind_{K}^{G(F)}(\rho\otimes\kappa)$, by \Cref{thmseveralscA}.
If the characters $\phi_i$ in the input $\Upsilon$ of the construction satisfy Yu's condition \GEE, then $\cind_{K}^{G(F)}(\rho\otimes\kappa)$ itself is irreducible (see \Cref{thmseveralscB}).

\subsection{Choices to be made in the construction}
 \label{sec:choices}

Let $\Upsilon$ be a cuspidal $G$-datum and assume $q>3$.
Ideally, a construction of supercuspidal representations
would output a single irreducible representation from each input $\Upsilon$.
For several reasons, however, our construction is not so precise.
In this subsection we reflect on the choices in our construction,
in addition to~$\Upsilon$,
that one needs to make to produce a single supercuspidal representation.
We stress that these additional choices
are only necessary when either
$p=2$ and $G$ has complicated Bruhat--Tits theory at~$x$,
or when $p$ is a torsion prime
for the Langlands dual group~$\widehat G$.
Neither of these phenomena occur for the general linear group,
where additional choices are not necessary (see \Cref{thm79}).

The only choices the reader has to make in the construction outlined in \Cref{sec:overview} are the representations $\kappa\in \Irr(K,K^\minus,\kappa^\minus)$ and $\sigma \in \Irr(\wt K, K, \rho\otimes\kappa)$, which are described by the following \namecref{thm72}.

\begin{lemma} \label{thm72}
	\begin{lemmaenum}
		\item[]
		\item \label{thm72a}
		The character group of $K/K^\minus$ acts transitively, by twisting,
		on the set $\Irr(K,K^\minus,\kappa^\minus)$.
		
		\item \label{thm72b}
		There is a bijection
		$\Irr(\wt K, K, \rho\otimes\kappa) \simeq
		\Irr\bigl(\End_{\wt K}\bigl(\Ind_K^{\wt K}(\rho\otimes\kappa)\bigr)\bigr)$.
	\end{lemmaenum}
\end{lemma}

\begin{proof}
	The first part is a special case of \cite[Lemma~A.4.3]{Kaletha-non-singular} since $K/K^-$ is abelian,
	and the second is a special case of \Cref{thm71a}.
\end{proof}

\begin{remark}[On choosing $\kappa$ and $\sigma$] \label{thm73}
The choices of $\kappa$ and~$\sigma$ are of a different nature.

The choice of $\kappa$ can be accounted for by refactorization,
as in the work of Hakim and Murnaghan (\cite[Definition~4.19]{HakimMurnaghan}).
Specifically, suppose $\kappa$ and~$\kappa'$
are two choices of an element of $\Irr(K,K^\minus,\kappa^\minus)$.
Then by \Cref{thm72a} there is a character $\chi$
of $(G_{n+1})_{[x]}$ trivial on $(G_{n+1})_{[x]}^\minus$
such that $\kappa'=\chi\otimes\kappa$,
where we identify $\chi$ with its inflation to~$K$.
Let $\Upsilon'$ be the supercuspidal $G$-datum obtained from~$\Upsilon$
by replacing $\rho$ with~$\rho'\defeq\chi\otimes\rho$.
Since $\rho\otimes\kappa' = \rho'\otimes\kappa$,
and the supercuspidal representations we construct
depend only on this tensor product rather than its individual factors,
we would produce the same set of supercuspidal representations
by choosing $\kappa'$ for~$\Upsilon$ or~$\kappa$ for~$\Upsilon'$.
All in all, choosing a different $\kappa$ 
can be accounted for by instead modifying the depth-zero part of~$\Upsilon$.

The choice of $\sigma$ is in general of a nonabelian nature
and cannot be accounted for by refactorization.
In \Cref{sec:ge2example}, summarized in \Cref{thm66},
we give an example where $\dim(\sigma) > \dim(\rho\otimes\kappa)$,
showing that in general $\sigma$ might not extend~$\rho\otimes\kappa$. 

\end{remark}

\begin{remark}[Why we do not refine the input~$\Upsilon$] \label{thm70}
Our construction is formulated so that a single $G$-datum $\Upsilon$
gives rise to a finite set of supercuspidal representations rather than a single one.
For many reasons it would be advantageous to reformulate the construction
so that it produce an individual representation,
starting from a variant datum $\Sigma$ which would somehow record the choices of $\sigma$ and~$\kappa$.
However, such a reformulation would come at the price of making the input $\Sigma$
much more conceptually and notationally complicated than~$\Upsilon$.

In more detail, suppose that $\kappa$ extends to a representation
$\wt \kappa$ of $\wt K=N_{\Kplus}(\rho\otimes\kappa)$ (for any allowed choice of $\rho$).
Then, instead of the representation $\rho$ of $(G_{n+1})_{[x]}$ as the input in $\Upsilon$,
we could take as an input in $\Sigma$
a depth-zero representation $\wt\rho$ of a group $K_{\wt\rho}$,
with $K\subseteq K_{\wt\rho} \subseteq\Kplus$,
such that $N_{\Kplus}(\wt\rho|_K\otimes\kappa) = K_{\wt\rho}$
and $\wt\rho$ is cuspidal when restricted to $(G_{n+1})_{x,0}$.
In terms of our current language, $\sigma = \wt\rho\otimes\wt\kappa$.
In this new language, however, the input is unpleasant to describe
because one needs to already construct $K$ and $\kappa$ to even define where $\wt \rho$ lives.
\end{remark}

We finish by discussing the case of the general linear group.

\begin{lemma} \label{thm87}
Let $k$ be a field, let $G=\GL_n$ over~$k$,
and let $M$ be a twisted Levi subgroup of~$G$.
Then there are separable field extensions $\ell_1,\dots,\ell_r$
of~$k$ and integers $d_1,\dots,d_r$ with
$n = \sum_{i=1}^r d_i[\ell_i:k]$ such that
$M\simeq\prod_{i=1}^r \Res_{\ell_i/k}\GL_{d_i}$.
If, moreover, $M$ is elliptic, then $r=1$.
\end{lemma}

\begin{proof}
Let $k^\tn{sep}$ be a separable closure of~$k$,
let $M_0$ be a Levi subgroup of~$G$,
let $N_0=N_G(M_0)$, and let $W_0 = N_G(M_0)/M_0$.
There is a section $W_0\to N_0$
that preserves a fixed pinning of~$M_0$,
which we use to write $N_0 \simeq M_0 \rtimes W_0$.

We can identify $(G/N_0)(k)$ with the set of twisted Levi subgroups of~$G$
that are $G(k^\tn{sep})$-conjugate to~$M_0$.
If $M'$ is one such Levi subgroup,
then the equivalence class $z'\in H^1(k,\Aut(M_0))$
corresponding to~$M'$ is the image of $M'$ under the composite map
\[
(G/N_0)(k)\longrightarrow H^1(k,N_0)\longrightarrow H^1(k,\Aut(M_0)),
\]
where the first map is the connecting homomorphism
and the second is obtained from the map
$N_0\to\Aut(M_0)$ arising from the conjugation action of $N_0$ on $M_0$.
At the same time, the composition $W_0\to N_0\to\Aut(M_0)$
yields a map $H^1(k,W_0)\to H^1(k,\Aut(M_0))$,
and for any $w\in H^1(k,W_0)$,
the resulting twist $M_{0,w}$ of~$M_0$
is a product of Weil restrictions of general linear groups
as in the statement of the \namecref{thm81} by our choice of section $W_0 \ra N_0$.
Therefore, it suffices to prove the following claim:
The map $f\colon H^1(k,N_0)\to H^1(k,W_0)$ is a bijection.

Since $H^1(k,M_0)$ is trivial by Hilbert's Theorem~90,
the fiber of $f$ over the basepoint of $H^1(k,W_0)$ is a singleton.
To prove the same for the other fibers,
we use a twisting argument,
as in \cite[Chapter~I, Section~5.5, Corollary~2]{Serre02}.
Let $w\in H^1(k,W_0)$ with image $n\in H^1(k,N_0)$
under the section $W_0\to N_0$.
Then there is an exact sequence of sets
\[
H^1(k,M_{0,n}) \longrightarrow H^1(k,N_{0,n})
\longrightarrow H^1(k,W_{0,w}).
\]
As we observed above,
$M_{0,n}$ is a product of Weil restrictions of general linear groups,
and thus $H^1(k,M_{0,n}) = 1$ by a minor extension of Hilbert's Theorem~90,
for instance, \cite[Chapter~X, Section~1, Exercise~1]{Serre79}.
Hence $f^{-1}(w)$ is a singleton,
and varying $w$, we find that $f$ is a bijection.
This completes the proof of the first part.

Lastly, the claim about ellipticity follows since the center of $M$ contains $\GL_1^r$.
\end{proof}

\begin{remark}[No choices for $\GL_N$] \label{thm79}
When $G=\GL_N$, there is a unique choice for $\kappa$ and~$\sigma$. 

For $\sigma$, 
the root data of $\GL_N$ and its twisted Levi subgroups have no torsion primes.
Hence $\widetilde K=K$ by \Cref{thm02b}, and so $\sigma=\rho\otimes\kappa$.

For $\kappa$, we claim that $G_{n+1}(F)_{[x]}^\minus = G_{n+1}(F)_{[x]}$,
from which it follows that $K=K^\minus$.
Indeed, by \Cref{thm87}, all elliptic tame twisted Levi subgroups of~$\GL_N$,
including $G_{n+1}$, are of the form $\Res_{E/F}\GL_d$ for $E/F$
a tame extension such that $N=d\cdot[E:F]$.
Moreover, using for instance the lattice-chain model
of the Bruhat-Tits building of the general linear group
(see \cite[Remark~15.1.33]{Kaletha-Prasad-BTbook}),
we observe that $\GL_d(E)_{[x]} = E^\times\cdot\GL_d(E)_{x,0}$.
Now
\[
\frac{G_{n+1}(F)_{[x]}}{Z(G(F))\cdot G_{n+1}(F)_{x,0}}
= \frac{E^\times\cdot\GL_d(\cal O_E)}{F^\times\cdot\GL_d(\cal O_E)}
\simeq\frac{E^\times}{F^\times\cdot\cal O_E^\times}
\simeq\frac{\bbZ}{e(E/F)\bbZ}.
\]
where $e(E/F)$ is the ramification degree of~$E/F$.
Since $E/F$ is tame, if $p=2$, then $e(E/F)$ is odd.
So $K=K^\minus$.
\end{remark}

\subsection{Heisenberg $\bbF_p$-groups arising from $p$-adic groups} \label{sec:p-adicHeis}
We recall that $G$ is a reductive $F$-group, and we 
let $H$ be a twisted Levi subgroup of~$G$ that splits over a tamely ramified extension of $F$. 
Let $x\in\sB(H,F)\subseteq \sB(G,F)$, and let $\phi\colon H(F)\to\bbC^\times$
be a $(G,H)$-generic character of some positive depth~$r$,
as in \Cref{thm44}.

The following lemma is due to Yu (\cite[Proposition~11.4]{Yu}) if $p>2$ and extends to the case of $p=2$.
\begin{lemma}%
	 \label{thm33}
The group $\sfV^\natural := \dfrac{(H,G)(F)_{x,r,r/2}}%
{(H,G)(F)_{x,r,r/2+}\cap\ker(\hat\phi)}$
is a Heisenberg $\bbF_p$-group with center $\dfrac{(H,G)(F)_{x,r,r/2+}}%
{(H,G)(F)_{x,r,r/2+}\cap\ker(\hat\phi)}$. The conjugation action of $H(F)_{[x]}$ induces an action on $\sfV^\natural$ that fixes the center of $\sfV^\natural$.
\end{lemma}

\begin{proof}
For brevity, we write $\sfV:=\dfrac{(H,G)(F)_{x,r,r/2}}%
{(H,G)(F)_{x,r,r/2+}} $, which is an abelian group of exponent $1$ or~$p$.

When $p\neq2$, the proofs of Lemma~11.3 and Proposition~11.4 of \cite{Yu} still work as written also for our more general notion of $(G,H)$-generic characters so that $\sfV^\natural$ ($=J/N$ in Yu's notation)
is a Heisenberg $\bbF_p$-group and the action of $H(F)_{[x]}$ on $\sfV^\natural$ fixes the center of $\sfV^\natural$.

When $p=2$, %
Lemma~11.1 of \cite{Yu}
still hold with the same proof, i.e., the bi-additive pairing 
$ \sfV \times  \sfV	
\ra \{\pm 1\} \subset \bC^\times$ 
given by
$(aJ_+, bJ_+) \mapsto \hat\phi([a,b])$ where $J_+:=(H,G)(F)_{x,r,r/2+}$, is well defined and non-degenerate. Hence the center of $\sfV^\natural$ is 
$$ Z(\sfV^\natural)= \dfrac{(H,G)(F)_{x,r,r/2+}}%
{(H,G)(F)_{x,r,r/2+}\cap\ker(\hat\phi)} \simeq \{\pm 1\} ,$$
 and $\sfV^\natural/Z(\sfV^\natural)=\sfV$ is an abelian 2-group of exponent at most 2.
Hence the group $\sfV^\natural$ is a Heisenberg $\bbF_2$-group by \Cref{thm39} and the action of $H(F)_{[x]}$ on $\sfV$ fixes the center.
\end{proof}

\begin{example}[Positive- and negative-type Heisenberg $\bbF_2$-groups in $2$-adic groups]
\label{thm68}
In this example we show that both positive- and negative-type
Heisenberg $\bbF_2$-groups can arise in the construction of supercuspidal representations.

Suppose $p=2$ and $F$ has residue field $k_F$.
Let $E/F$ be a quadratic unramified extension of~$F$ and
let $\sigma$ be the nontrivial element of~$\Gal(E/F)$.
Let $G = \GL(E/F)$ be the group of linear automorphisms of the $F$ vector space $E$,
isomorphic (after choosing an ordered basis of~$E$) to~$\GL_2$.
Then $T = \Res_{E/F}\bbG_\tn{m}$ canonically embeds as
a maximal torus of~$G$ through the multiplication action of~$E^\times$ on~$E$.
Let $x$ be such that $G(F)_x = \GL(\cal O_E/\cal O_F)$.
Let $n\geq 1$ be an integer, and let $\phi\colon E^\times\to\bbC^\times$
be a $(G,T)$-generic character of depth~$2n$.
We claim that $(T,G)(F)_{x,2n,n}/((T,G)(F)_{x,2n,n+} \cap \ker(\hat\phi))$
is a negative-type Heisenberg $\bbF_2$-group.

The main problem is to describe the group $(T,G)(E)_{x,2n,n}/(T,G)(E)_{x,2n+,n+}$,
and especially the quotients of root groups, together with the action of $\Gal(E/F)$ on this group.
The root groups have the following description:
There are orthogonal idempotents $e_1\neq e_2$ in~$(E\otimes_F E)^\times$
interchanged by~$\Gal(E/F)$ such that one root subgroup of $G(E)$, call it $U_\alpha(E)$,
is represented in the ordered basis $(e_1,e_2)$
by matrices of the form $u(a) = \smat1a01$ with $a\in E$,
while the other, call it $U_{-\alpha}(E)$,
is represented in this ordered basis
by the matrices of the form $v(b) = \smat10b1$ with $b \in E$.
It follows that $\Gal(E/F)$ acts on $U_\alpha(E)_{x,n}/U_\alpha(E)_{x,n+}\oplus U_{-\alpha}(E)_{x,n}/U_{-\alpha}(E)_{x,n+}$ by
\[
\sigma(\bar u(a) + \bar v(b))
= \bar v(\sigma a) + \bar u(\sigma a),
\quad (\val(a), \val(b) \geq n),
\]
where $\bar u$ and $\bar v$ denote the images of $u$ and $v$ in the above quotient spaces.
Using the commutator relation for opposite root groups,
we see that %
\[
Q(\bar u(a) + \bar v(b))
\defeq \bigl(u(a)v(b)\bigr)^2 \bmod{(T,G)(E)_{x,2n+,n+}}
\equiv \alpha^\vee(1 + ab) \bmod{(T,G)(E)_{x,2n+,n+}}. %
\]
In other words, after identifying $U_\alpha(E)_{x,n}/U_\alpha(E)_{x,n+}\oplus U_{-\alpha}(E)_{x,n}/U_{-\alpha}(E)_{x,n+}$ with $k_E\oplus k_E$ and $T(E)_{x,2n}/T(E)_{x,2n+}$ with $k_E$,
the quadratic form $Q$ becomes the split form $Q(a,b) = ab$.
On the subspace of $\Gal(E/F)$-invariants in $k_E\oplus k_E$, %
which we can identify with $k_E$ by matching $a\in k_E$ with $(a,\sigma a)$,
the quadratic form restricts to the norm form $Q(a) = a\cdot\sigma a$.

We claim that the nondegenerate quadratic form $Q'\defeq \hat\phi\circ Q\colon k_E\to\{\pm1\}$ is non-split. 
Note that the vanishing set of $Q'$
has size less than half of~$k_E$ as follows by direct computation:
\[
|Q'^{-1}(1)| = (\tfrac12 q-1)(q+1) + 1
= \tfrac12 q^2 - \tfrac12 q < \tfrac12|k_E| .
\]
Hence $Q'$ cannot be split, and therefore 
$(T,G)(F)_{x,2n,n}/\bigl((T,G)(F)_{x,2n,n+} \cap \ker(\hat\phi)\bigr)$ is of negative type.

At the same time, since the central product $Q_8\circ Q_8$
is a positive-type Heisenberg $\bbF_2$-group (see \Cref{thm29}),
doubling the previous example (that is, replacing $G$ by~$G\times G$,
$T$ by $T\times T$, $\phi$ by $\phi \otimes \phi$, and so on) yields
an example where $(T,G)(F)_{x,2n,n}/((T,G)(F)_{x,2n,n+} \cap \ker(\hat\phi))$ is of positive type.
So both possibilities can arise.
\end{example}

\subsection{Weil representations of open, compact-mod-center subgroups} \label{sec:aux}
Let $\Upsilon$ be a supercuspidal $G$-datum.
In this subsection we first construct various auxiliary subgroups and vector spaces from~$\Upsilon$
which are needed both for the Heisenberg--Weil extension step
and for the proof of supercuspidality, and then we prove the remaining claims used in the construction of smooth representations from $\Upsilon$ outlined in Section \ref{sec:overview}. In Section \ref{sec:proof} we will then prove that these representations are supercuspidal.

Let $(\sG_{n+1})_{x,0}$ be the connected parahoric integral model of $G$ at~$x$ and let $\sfG_{n+1}$\index[notation]{Gn+1@$\sfG_{n+1}$} be the reductive quotient of its special fiber. Hence $\sfG_{n+1}$ is a reductive $k_F$-group with
$\sfG_{n+1}(k_F) = (G_{n+1})_{x,0}/(G_{n+1})_{x,0+}$.
By \Cref{thm33}, the group\index[notation]{Vin@$\sfV_i^\natural$}
\[
\sfV_i^\natural \defeq \frac{(G_{i+1},G_i)(F)_{x,r_i,r_i/2}}%
{(G_{i+1},G_i)(F)_{x,r_i,r_i/2+}\cap\ker(\hat\phi_i)}
\]
is a Heisenberg $\bbF_p$-group.

For any tamely ramified finite field extension $F'/F$, we define the following group and its $\bF_p$-vector space quotient
\[
\widetilde \sfV_{i,F'}^\natural \defeq \frac{(G_{i+1},G_i)(F')_{x,r_i,r_i/2}}%
{(G_{i+1},G_i)(F')_{x,r_i+,r_i/2+}}
\qquad \text{and} \qquad
\sfV_{i,F'} \defeq \frac{(G_{i+1},G_i)(F')_{x,r_i,r_i/2}}%
{(G_{i+1},G_i)(F')_{x,r_i,r_i/2+}} . 
\]

We may drop the subscript $F'$ if $F'=F$, that is, write $\sfV_i:=\sfV_{i,F}$\index[notation]{Vi@$\sfV_i$} and $\widetilde \sfV_i^\natural:=\widetilde \sfV_{i,F}^\natural$.

Note that $\sfV_i^\natural$ is an intermediate quotient between $\widetilde \sfV_i^\natural$ and $\sfV_i$. While we are eventually interested in the group $\sfV_i^\natural$, we will take advantage of the groups $\widetilde \sfV_{i,F'}^\natural$ that allow us to deduce results over $F$ by proving the analogous results after base change.

The remaining objects that we like to introduce depend on two additional choices: $T$ and~$\lambda$.
Let $T$ be a maximally split, tame maximal torus of $G_{n+1}$ with splitting field $E/F$ such that $x$ is in the apartment $\sA(T,F)$ of $T$.
Let $\lambda\in X_*(T)^{\Gal(E/F)}\otimes\bbR$. The following discussion will later be applied to several choices of $\lambda$, but we do not record $\lambda$ in the notation as this should be clear from the context.

Let $S$ be the maximal split subtorus of~$T$,
let $\sS$ be an integral model of~$S$ that is
a split maximal torus of $(\sG_{n+1})_{x,0}$,
and let $\sfS = \sS_{k_F}$,
a maximal split torus in $\sfG_{n+1}$.
Then 
\[
\lambda\in X_*(S)\otimes\bbR \simeq X_*(\sfS)\otimes\bbR.
\]
Let $\sfP$ be the parabolic subgroup of $\sfG_{n+1}$
containing~$\sfS$ such that%
\[
\Phi(\sfP,\sfS) = \{\alpha \in \Phi(\sfG_{n+1},\sfS) \mid \lambda(\alpha) \geq 0\}
\]
and let $\sfU$ be the unipotent radical of $\sfP$.
Define $\sfP_E$ and $\sfU_E$ analogously. 
Since $\lambda$ is $\Gal(E/F)$-stable,
there is a (parabolic) subgroup $P$ of~$G_{n+1}$ containing~$T$ such that
\[
\Phi(P_E,T_E) = \{\alpha\in\Phi(G_{n+1},T) \mid \lambda(\alpha) \geq 0\}.
\]
Let $U$ be the unipotent radical of~$P$ and write
\[
U(F)_{x,0} \defeq U(F)\cap G(F)_{x,0},
\qquad
P(F)_{x,0} \defeq P(F)\cap G(F)_{x,0}.
\]

We will now use $\lambda$
to also define a partial polarization of each~$\sfV_i$.

For the isotropic subspaces,
define the unipotent subgroups $U_{i,E}^\pm$ of $G_{i,E}$ by
\begin{align*}
U_{i,E}^+ &\defeq \bigl\langle U_{\alpha,E} \mid
\alpha \in \Phi(G_i,T)\smallsetminus\Phi(G_{i+1},T),\;
\lambda(\alpha) > 0\bigr\rangle \\
U_{i,E}^- &\defeq \bigl\langle U_{\alpha,E} \mid
\alpha \in \Phi(G_i,T)\smallsetminus\Phi(G_{i+1},T),\;
\lambda(\alpha) < 0\bigr\rangle.
\end{align*}
Given $\tilde r\in\widetilde\bbR$,
let $U^\pm_{i,E}(E)_{x,\tilde r}$ be the compact subgroup generated by the groups
$U_\alpha(E)_{x,\tilde r}$ with $\alpha\in\Phi(U_{i,E}^\pm,T_E)$.
Since $\Phi(U_{i,E}^\pm,T_E)$ is $\Gal(E/F)$-stable,
the group $U^\pm_{i,E}$ descends to a unipotent $F$-group $U_i$ contained in~$G_i$.
We define $U^\pm_i(F)_{x,\tilde r} \defeq G(F)\cap U^\pm_{i,E}(E)_{x,\tilde r}$\index[notation]{UiplusF@$U_i^+(F)_{x,r}$}
for $\tilde r\in\widetilde\bbR$.
Let
\[
\sfV_{i,E}^\pm \defeq
\frac{U^\pm_{i,E}(E)_{x,r_i/2}}{U^\pm_{i,E}(E)_{x,r_i/2+}}, \qquad
\sfV_i^\pm \defeq \frac{U^\pm_i(F)_{x,r_i/2}}{U^\pm_i(F)_{x,r_i/2+}}
= (\sfV_{i,E}^\pm)^{\Gal(E/F)},
\]
where the last equality follows from the same arguments used to prove \cite[Corollary~2.3]{Yu}.
Via the inclusions of $U^\pm_{i,E}(E)_{x,r_i/2}$ and $U^\pm_{i,E}(F)_{x,r_i/2}$ into the appropriate subgroups of $G(E)$ and $G(F)$, we can identify $\sfV_{i,E}^\pm$ 
with a subgroup of $\widetilde \sfV_{i,E}^\natural$,
and $\sfV_i^\pm$ with subgroups of $\widetilde \sfV_i^\natural$, of $\sfV_i^\natural$, and of $V_i$.

For the nondegenerate part of the partial polarization,
write $H_i = Z_{G_i}(\lambda)$ and $H_{i+1} = Z_{G_{i+1}}(\lambda)$
and define the following subgroups of
$\widetilde \sfV_i^\natural$ and $\widetilde \sfV_{i,E}^\natural$:
\[
\widetilde \sfV_{i,0}^\natural \defeq
\frac{(H_{i+1},H_i)(F)_{x,r_i,r_i/2}}%
{(H_{i+1},H_i)(F)_{x,r_i+,r_i/2+}}
\subseteq \widetilde \sfV_i^\natural \qquad\tn{and}\qquad
\widetilde \sfV_{i,0,E}^\natural \defeq
\frac{(H_{i+1},H_i)(E)_{x,r_i,r_i/2}}%
{(H_{i+1},H_i)(E)_{x,r_i+,r_i/2+}}
\subseteq \widetilde \sfV_{i,E}^\natural.
\]
By \cite[Proposition~2.2]{Yu},
\[
\widetilde \sfV_i^\natural = (\widetilde \sfV_{i,E}^\natural)^{\Gal(E/F)},\qquad
\widetilde \sfV_{i,0}^\natural = (\widetilde \sfV_{i,0,E}^\natural)^{\Gal(E/F)}.
\]
We also define the following Heisenberg $\bF_p$-group with its quotient $\bF_p$-vector space
\[
\sfV_{i,0}^\natural \defeq
\frac{(H_{i+1},H_i)(F)_{x,r_i,r_i/2}}%
{(H_{i+1},H_i)(F)_{x,r_i,r_i/2+}\cap\ker(\hat\phi_i)}
\twoheadleftarrow \widetilde\sfV_{i,0}^\natural
\qquad
\text{and}
\qquad
\sfV_{i,0} \defeq
\frac{(H_{i+1},H_i)(F)_{x,r_i,r_i/2}}%
{(H_{i+1},H_i)(F)_{x,r_i,r_i/2+}} .
\]

Note that using the notation from Section \ref{sec:Heisenberg}, we have $\sfV_i=\sfV_i^\natural/Z(\sfV_i^\natural)=\sfV_{\sfV_i^\natural}$.
\begin{lemma}
	$\sfV_i = \sfV_i^+ \oplus \sfV_{i,0} \oplus \sfV_i^-$ is a partial polarization in the sense of Section \ref{sec:Heisenberg}.
\end{lemma}

\begin{proof}
It suffices to show that $\sfV_{i,0}$ is a nondegenerate subspace
and that the subspaces $\sfV_i^+$ and $\sfV_i^-$ are isotropic and orthogonal to~$\sfV_{i,0}$.
The subspace $\sfV_{i,0}$ is nondegenerate by \Cref{thm33} applied to
$H_{i+1}\subseteq H_i$ and the character
$\phi_i|_{H_{i+1}}$, which is $(H_i,H_{i+1})$-generic of depth~$r_i$ because
\[
\Phi(H_i,T)\setminus\Phi(H_{i+1},T)
\subseteq \Phi(G_i,T)\setminus\Phi(G_{i+1},T),
\]
and because $\phi_i$ can be represented by an element in $\fg_i(F)^*$ of depth $-r_i$ that is trivial on the sum $\ft^\perp$ of the root subspaces of $\fg_i(F)$ with respect to $T$ hence its restriction to $\fh_i$ has also depth $-r_i$.
The subspaces $\sfV_i^+$ and $\sfV_i^-$
are isotropic because they embed as abelian subgroups of $\sfV_i^\natural$.
To see that $\sfV_i^+$ and $\sfV_i^-$ are orthogonal to $\sfV_{i,0}$,
it is enough to show that they are normalized by $\sfV_{i,0}^\natural$,
or equivalently, by $\widetilde\sfV_{i,0}^\natural$,
and this can be checked over~$E$ by Galois descent.
Using the commutator relations for root groups, see, e.g., \cite[Section~6]{Yu},
we see that for $\alpha\in\Phi(H_i,T) \smallsetminus \Phi(H_{i+1},T), \beta \in \Phi(H_{i+1},T)$,
the images of the root groups $U_\alpha(E)_{x,r_i/2}$ and $U_\beta(E)_{x,r_i}$ and of $T(E)_{x,r_i}$ normalize the groups $\sfV_{i,E}^\pm$.
Hence $\widetilde\sfV_{i,0,E}^\natural$ normalizes $\sfV_{i,E}^\pm$.
\end{proof}

Recall that we also view $\sfV_i^+$ as a subgroup of $\widetilde \sfV_i^\natural$.

\begin{lemma} \label{thm16}
	The action of $G_{n+1}(F)_{[x]}$ on $\widetilde \sfV_i^\natural$ induced by conjugation satisfies the following properties.
\begin{lemmaenum}
\item \label{thm16a}
	We have $g(\sfV_i^+) = \sfV_i^+$ and $g(\sfV_i^+ \times \widetilde \sfV_{i,0}^\natural) = \sfV_i^+ \times \widetilde \sfV_{i,0}^\natural$ for all $g \in \sfP(k_F)$.
\item \label{thm16b}
	We have $g(x)\cdot x^{-1}\in \sfV_i^+$ for all $g\in\sfU(k_F)$ and all
$x\in \sfV_i^+ \times \widetilde \sfV_{i,0}^\natural$.
\end{lemmaenum}
\end{lemma}

\begin{proof}
We first analyze the situation over~$E$, then pass to~$F$.
Since $T$ is split over~$E$,
these three claims reduce to a commutator calculation with root groups,
and follow from the following observations.
Let $\alpha\in\Phi(G_{n+1},T)$ with $\lambda(\alpha)\geq 0$ (a root of $\sfP_E$),
let $\beta\in\Phi(G_i,T)\smallsetminus\Phi(G_{i+1},T)$ with $\lambda(\beta)\geq0$
(a potential ``root'' of $\sfV_i^+\times\widetilde\sfV_{i,0}^\natural$),
and suppose $i\alpha+j\beta\in\Phi(G_i,T)$ with $i,j>0$
(a root whose root group might appear in the commutator of the previous two root groups).
The following three claims are proved by the subsequent observations about roots:
\begin{itemize}
\item
$\sfP_E(k_E)$ preserves $\sfV_{i,E}^+$:
If $\lambda(\beta) > 0$, then $\lambda(i\alpha + j\beta) > 0$.

\item
$\sfP_E(k_E)$ preserves $\sfV_{i,E}^+ \times \widetilde \sfV_{i,0,E}^\natural$:
If $\lambda(\beta) \geq 0$, then $\lambda(i\alpha + j\beta) \geq 0$.

\item
$g(x)\cdot x^{-1}\in \sfV_{i,E}^+$ for all  $g\in\sfU_E(k_E)$ and all
$x\in \sfV_{i,E}^+ \times \widetilde \sfV_{i,0,E}^\natural$:
If $\lambda(\alpha) > 0$ and $\lambda(\beta) \geq 0$, then $\lambda(i\alpha + j\beta) > 0$.
\end{itemize}
Now the result over $F$ follows from Galois descent,
using 
$\sfV_i^+=(\sfV_{i,E}^+)^{\Gal(E/F)}$, 
$\widetilde \sfV_{i,0}^\natural=(\widetilde \sfV_{i,0,E}^\natural)^{\Gal(E/F)}$,
$\sfU(k_F)\subseteq \sfU_E(k_E)^{\Gal(E/F)}$, 
and $\sfP(k_F)\subseteq \sfP_E(k_E)^{\Gal(E/F)}$. 
\end{proof}

Recall the normal subgroup $(G_{n+1})^\minus_{[x]}\subseteq(G_{n+1})_{[x]}$
from \Cref{sec:overview}, and that the conjugation action of $G_{n+1}(F)_{[x]}$ induces an action on~$\sfV_i^\natural$ that is trivial on the center of $\sfV_i^\natural$ by \Cref{thm33}. 

\begin{proposition} \label{thm12} 
Let $A$ be the image of $(G_{n+1})_{[x]}^\minus$
under the map $(G_{n+1})_{[x]}^\minus\ra \AutZfix(\sfV_i^\natural)$ induced by conjugation.
Then the Heisenberg representation $\WeilRep_i$ of $\sfV_i^\natural$
extends to a Heisenberg--Weil representation of $A\ltimes \sfV_i^\natural$.
\end{proposition}

\begin{proof}
If $p \neq 2$, then by \cite[Lemma~11.4]{Yu} (which we already observed in the proof of \Cref{thm33} to also work in our setting) we have a group homomorphism $(G_{n+1})_{[x]} \ltimes \sfV_i^\natural \ra \Sp(\sfV_i) \ltimes \sfV_i^\sharp$ whose restriction to $(G_{n+1})_{[x]} \ltimes \{1\}$ is the map $(G_{n+1})_{[x]} \ra \Sp(\sfV_i)$ induced by conjugation and whose restriction to $\{1\} \ltimes \sfV_i^\natural$ is a special isomorphism  $\{1\} \ltimes \sfV_i^\natural \ra  \{1\} \ltimes \sfV_i^\sharp$. Thus we obtain a desired extension to a Heisenberg--Weil representation of of $A\ltimes \sfV_i^\natural$ by \Cref{defnewHeisenbergWeil}.

So we assume $p = 2$ for the remainder of the proof.
By the definition of $(G_{n+1})_{[x]}^\minus$
and the fact that the image of $Z(G_{n+1})$ in $\AutZfix(\sfV_i^\natural)$ is trivial,
there is a Borel subgroup $\sfB$ of $\sfG_{n+1}$
with unipotent radical $\sfN$ such that
the image $A_p$ of $\sfN(k_F)$ in $\AutZfix(\sfV_i^\natural)$
is a Sylow $p$-subgroup of~$A$.
 We choose $\lambda \in X_*(T)^{\Gal(E/F)} \otimes \bR$ such that $\sfP=\sfB$ and $\sfU=\sfN$.
It follows from \Cref{thm16a} that $A_p$ is contained in $\cP(\sfV_i^+)$, where the anisotropic subspace $\sfV_i^+ \subseteq \sfV_i$ is viewed as a subgroup of $\sfV_i^\natural$ via the above described splitting and $\cP$ is as defined in \Cref{thm20}.
Hence the existence of the Heisenberg--Weil representation follows from
\Cref{defnewHeisenbergWeil} and \Cref{thm14}.
\end{proof}

\begin{notation} \label{notation:Weil}
	We denote the composition of $(G_{n+1})_{[x]}^\minus \twoheadrightarrow A$ with the Heisenberg--Weil representation of \Cref{thm12} also by $\WeilRep_i$. 
\end{notation}

There is a potential ambiguity in the construction of $\WeilRep_i$ when $p=2$, because
a priori the Heisenberg--Weil representation of a finite group
is only well-defined up to a character of this finite group
that has order one or two.
However, the next result will imply that this finite group
has no characters of order two if $q>2$,
implying that the extension $\WeilRep_i$ is uniquely defined (see \Cref{thm120}).

\begin{lemma} \label{thm82}
	If $p=2$ and $q>2$, then
	$(G_{n+1})_{[x]}^\minus/(Z(G)\cdot(G_{n+1})_{x,0+})$
	has no characters of order two.
\end{lemma}

\begin{proof}
	The group $(G_{n+1})_{[x]}^\minus/(Z(G)\cdot(G_{n+1})_{x,0+})$ fits into the short exact sequence
	\[
	\begin{tikzcd}[column sep=scriptsize]
	1 \rar & \dfrac{(G_{n+1})_{x,0}}{Z(G)_0\cdot(G_{n+1})_{x,0+}} \rar &
	\dfrac{(G_{n+1})_{[x]}^\minus}{Z(G(F))\cdot(G_{n+1})_{x,0+}} \rar &
	\dfrac{(G_{n+1})_{[x]}^\minus}{Z(G(F))\cdot(G_{n+1})_{x,0}} \rar & 1.
	\end{tikzcd}
	\]
	By definition, the rightmost quotient has no characters of order two.
	So it suffices to show that the leftmost kernel has no characters of order two.
	But already the quotient $(G_{n+1})_{x,0}/(G_{n+1})_{x,0+}$,
	the $\bbF_q$-points of a reductive $\bbF_q$-group,
	has no characters of order two by \Cref{thm81}.
\end{proof}

\begin{proposition} \label{thm120}
	Suppose $q>2$.
	Let $A$ be the image of $(G_{n+1})_{[x]}^\minus$
	under the map $(G_{n+1})_{[x]}^\minus \ra \AutZfix(\sfV_i^\natural)$ induced by conjugation.
	Then the Heisenberg representation $\WeilRep_i$ of $\sfV_i^\natural$
	extends uniquely
	to a Heisenberg--Weil representation of $A\ltimes \sfV_i^\natural$.
\end{proposition}

\begin{proof}
	By \Cref{thm12} it remains to prove uniqueness of the extension.
	Since $(G_{n+1})_{x,0+}$ and $Z(G)$ act trivially on~$\sfV_i^\natural$, the group $A$ is a quotient of the group appearing in \Cref{thm82} and hence has no characters of order two. Therefore we can apply \Cref{cor-weil-uniqueness}. 
\end{proof}

We can now combine the representations $\WeilRep_i$
to make a representation of~$K^\minus$.
This step is almost exactly as in \cite{Yu},
but we spell it out in detail for clarity.

\begin{construction}[Homomorphisms from iterated semidirect products] \label{thm76}
Suppose we are given groups $A_1,\dots, A_n$
together with an action $(b,a)\mapsto {}^ba$
of $A_j$ on $A_i$ for every $1\leq i<j\leq n$.
If the ``cocycle condition'' ${}^{{}^cb}({}^ca) = {}^{cb}a$
is satisfied for all $i<j<k$
and $a\in A_i$, $b\in A_j$, $c\in A_k$,
then we can form the iterated semidirect product
$A^{\rtimes}\defeq A_1\rtimes\cdots\rtimes A_n$.
It is straightforward to check that under these circumstances
the semidirect product is associative, like the direct product,
so that there is no need to worry about the order
of inserting parentheses in this iterated semidirect product.

Suppose we are given another group $A'$
and homomorphisms $f_i\colon A_i\to A'$.
Then the induced map $f^\rtimes\colon A^{\rtimes}\to A'$
defined by $(a_1,\dots,a_n)\mapsto f_1(a_1)\cdots f_n(a_n)$
is a homomorphism if and only if
$f_i({}^ba)= f_j(b)f_i(a)f_j(b)^{-1}$ for every $1\leq i<j\leq n$
and $a\in A_i$ and $b\in A_j$.

Such iterated semidirect products naturally arise from the following situation.
Suppose that $B$ is an ambient group containing the $A_i$
as subgroups and that $A_j$ normalizes $A_i$ for $1\leq i < j\leq n$.
Using the conjugation action,
we see that the cocycle condition is satisfied (because $cbc^{-1}\cdot c = cb$)
and thus we may form the iterated semidirect product $A^\rtimes$.
Multiplication induces a homomorphism $A^\rtimes\to B$,
and we write $A_1\cdots A_n$ for its image.
In the situation of the second paragraph of this \namecref{thm76},
the homomorphism $f^\rtimes\colon A^\rtimes \to A'$ descends to a homomorphism
$f\colon A_1\cdots A_n\to A'$ if and only if $f_1(a_1)\cdots f_n(a_n) = 1$
whenever $a_1\cdots a_n=1$ and $a_i\in A_i$.

\end{construction}

\begin{lemma} \label{lemma:kappaminus}
	There exists a unique representation $\kappa^-$ of $K^-$ with underlying vector space $V_{\kappa^\minus} \defeq \bigotimes_{i=1}^n V_{\WeilRep_i}$ such that 
	\begin{enumerate}
		\item the restriction of $\kappa^-$ to $(G_{n+1})_{[x]}^\minus$ is $\otimes_{i=1}^n (\phi_i|_{(G_{n+1})_{[x]}^\minus} \otimes \WeilRep_i)$, and
		\item the restriction of $\kappa^-$ to $(G_j)_{x,r_j,r_j/2}$  for $1 \leq j \leq n$  is 
		$$\bigotimes_{i=1}^{j-1} \hat \phi_i|_{(G_j)_{x,r_j,r_j/2}} \otimes \WeilRep_j\otimes  \bigotimes_{i=j+1}^{n} \hat \phi_i|_{(G_j)_{x,r_j,r_j/2}},$$
		where $\WeilRep_j$ denotes the composition of $(G_j)_{x,r_j,r_j/2} \twoheadrightarrow \sfV_j^\natural$ with the Heisenberg representation $\WeilRep_j$. 
	\end{enumerate}
\end{lemma}

\begin{proof}
		\addtocounter{equation}{-1}
	\begin{subequations} 
We will prove the existence of $\kappa^-$. Uniqueness will then follow immediately. 
To make the notation more uniform, we write
\[
K_j \defeq \begin{cases}
(G_j)_{x,r_j,r_j/2} & \tn{if $1\leq j<n+1$,} \\
(G_{n+1})_{[x]}^-     & \tn{if $j=n+1$.}
\end{cases}
\]
Note for future use in the proof that if $j\neq i$,
then $K_j$ is contained in the domain of~$\hat\phi_i$.

We apply the observations from \Cref{thm76}.
The group $K^\minus$ is a quotient of the iterated semidirect product
$K_1\rtimes\cdots\rtimes K_{n+1}$.
Fix $i$ with $1\leq i\leq n$.
For each $1\leq j\leq n+1$, define the homomorphism
$\kappa^\minus_{ij}\colon K_j\to\GL(V_{\WeilRep_i})$ by 
\[
\kappa^\minus_{ij} = \begin{cases}
\hat\phi_i|_{K_j} & \tn{if $j\neq i$ and $j \neq n+1$,} \\
\WeilRep_i & \tn{if $j = i \neq n+1$,} \\
\phi_i|_{K_{j}}\otimes\WeilRep_i & \tn{if $j = n+1$.}
\end{cases}
\]
Note that in the second case, $\WeilRep_i$ denotes a Heisenberg representation,
and in the third case, it denotes a Weil representation, both restrictions of a Heisenberg--Weil representation.
We claim that for each fixed~$i$,
the representations $(\kappa^\minus_{ij})_{1\leq j\leq n+1}$
induce a representation $\kappa^\minus_i$ of~$K^\minus$.
This claim suffices because one can then define
$\kappa^\minus\defeq\otimes_{i=1}^n\kappa^\minus_i$.

To prove the claim, we have to first show that
the map on the iterated semidirect product induced by
the representations $(\kappa^\minus_{ij})_{1\leq j\leq n+1}$
is a homomorphism.
For this we use the criterion of \Cref{thm76},
which requires us to check that
\begin{equation} \label{thm77}
\kappa^\minus_{ij}(aba^{-1})
= \kappa^\minus_{ik}(a)\kappa^\minus_{ij}(b)
\kappa^\minus_{ik}(a)^{-1}
\end{equation}
for all $j<k$, $a\in K_k$, $b\in K_j$. We distinguish four cases.
First, if $j=i$ and $k=n+1$, then \eqref{thm77} holds by the definition
of the Heisenberg--Weil representation.
In the remaining cases one of $\kappa^\minus_{ij}$
or $\kappa^\minus_{ik}$ is a character
and so \eqref{thm77} amounts to showing that
$[K_k,K_j]\subseteq\ker(\kappa^\minus_{ij})$.
Second, if $j=i$ and $k\neq n+1$
then \eqref{thm77} holds because
\[
[K_k,K_i] \subseteq [(G_i)_{x,0+},K_i]
\subseteq (G_i)_{x,r_i+,r_i/2+} \subseteq \ker(\WeilRep_i),
\]
by a root-group commutator calculation and Galois descent.
In the remaining cases $\kappa^\minus_{ij}=\hat\phi_i|_{K_j}$
because $j\neq i$.
Third, if $j<i$, then \eqref{thm77} holds because
$[K_k,K_j]\subseteq K_j \subseteq\ker(\hat\phi_i)$.
Fourth, if $i < j$, then \eqref{thm77} holds because
$[K_k,K_j]\subseteq [G_{i+1}(F),G_{i+1}(F)] \subseteq \ker(\phi_i)$,
as $\phi_i$ is a character of~$G_{i+1}(F)$,
and $\ker(\phi_i) \cap [K_k,K_j] \subseteq\ker(\hat\phi_i)$
since $\hat\phi_i$ extends~$\phi_i|_{[K_k,K_j]}$.

It remains to show that $(\kappa^\minus_i)^\rtimes$
descends to a representation $\kappa^\minus_i$ of~$K^\minus$,
for which we need to prove that $\kappa^-_{i1}(a_1)\cdots \kappa^-_{i,n+1}(a_{n+1}) = \Id$, the identity linear transformation,
whenever $a_1\cdots a_{n+1}=1$ with $a_j \in K_j$ for $1 \leq j \leq n+1$. So suppose $a_1\cdots a_{n+1}=1$ with $a_j \in K_j$ for $1 \leq j \leq n+1$.
Then $\kappa^\minus_{ii}(a_i) = \hat\phi_i(a_i)\Id$
because $a_i\in K_i\cap\prod_{1\leq j\leq n+1,j\neq i} K_j \subseteq (G_i)_{x,r_i, r_{i-1}} \subseteq (G_i)_{x,r_i, r_i/2+}$,
and $\kappa^\minus_{i,n+1}(a_{n+1}) = \hat\phi_i(a_{n+1})\Id$ 
because $a_{n+1}\in K_n\cap\prod_{j=1}^n K_j \subseteq (G_{n+1})_{x,0+}$,
a group on which $\WeilRep_i$ is trivial. Hence 
$\kappa^-_{i1}(a_1)\cdots \kappa^-_{i,n+1}(a_{n+1})=\hat\phi_i(a_1)\cdots \hat\phi_i(a_{n+1}) \Id=\hat\phi_i(a_1\cdots a_{n+1})\Id=\Id$.
\end{subequations}
\end{proof}

For later use, let us record the following intertwining property.
\begin{lemma}\label{lemma:Knormalizeskappa-}
	Suppose $q>2$. For every $k \in K$, we have ${^k\kappa^-} \simeq \kappa^-$. In particular, if $\kappa$ is an irreducible representation of $K$ that contains $\kappa^-$ when restricted to $K^-$, then $\kappa$ is $\kappa^-$-isotypic.
\end{lemma}
\begin{proof}
 Let $k \in K$. Since $K^-$ normalizes $\kappa^-$, and $K=K^- \cdot (G_{n+1})_{[x]}$, we may assume without loss of generality that $k \in (G_{n+1})_{[x]}$. 
 Then $k$ is contained in the group
  $K'=(G_1)_{x,r_1/2+}\cdots (G_n)_{x,r_n/2+}(G_{n+1})_{[x]}$, and hence normalizes the character
 $\hat \phi \defeq \otimes_{i=1}^n\hat \phi_i|_{K'}$.
 Since the restriction of $\kappa^-$ to the normal subgroup 
 $$K_+\defeq (G_1)_{x,r_1/2+}\cdots (G_n)_{x,r_n/2+}(G_{n+1})_{x,0+}$$ is $\hat \phi|_{K_+}$-isotypic, and since the restriction of $\kappa^-$ to 
 $$K_{0+} \defeq (G_1)_{x,r_1/2} \cdots (G_n)_{x,r_1/n}(G_{n+1})_{x,0+}$$ 
 is by the theory of Heisenberg representations the unique (up to isomorphism) irreducible representation that is  $\hat \phi|_{K_+}$-isotypic when restricted to $K_+$, we obtain ${^k\kappa^-|_{K_{0+}}} \simeq \kappa^-|_{K_{0+}}$. 
 Thus it remains to show that under this isomorphism ${^k\kappa^-}(g)$ and $\kappa^-(g)$ agree for $g \in (G_{n+1})_{[x]}^-$. Since $\otimes_{i=1}^{n} \phi_i|_{(G_{n+1})_{[x]}^-}$ is the restriction of the character $\otimes_{i=1}^{n} \phi_i|_{(G_{n+1})_{[x]}}$, by \Cref{lemma:kappaminus}, it suffices to show that Weil representations ${^k\omega_i^-}(g)$ and $\omega_i^-(g)$ agree for $1 \leq i \leq n$ under the isomorphism that matches the underlying Heisenberg representations. This follows from the uniqueness of the extension of the Heisenberg representation (see \Cref{thm120}).
\end{proof}

\subsection{Supercuspidal representations} \label{sec:proof}
We keep the notation from the previous subsections. In particular, \Cref{lemma:kappaminus} provides us with a representation $\kappa^\minus$ of $K^\minus$, and we denote by $\kappa$ an irreducible representation of $K$ that contains $\kappa^\minus$ when restricted to $K^\minus$, and by $\sigma$ an irreducible representation of $N_{\Kplus}(\rho\otimes\kappa)$ that contains $(\rho \otimes \kappa)$ when restricted to $K$.
Our objective is to prove that $\cind_{N_{\Kplus}(\rho\otimes\kappa)}^{G(F)}(\sigma)$
is irreducible supercuspidal if $q>3$ (see \Cref{thmseveralscC}).

Since our proof is similar to the proof of \cite[Theorem~3.1]{Fi-Yu-works},
we will mostly focus on the modifications necessary to accommodate the new Heisenberg--Weil representation theory approach for $p=2$ and deal with the more complicated intertwining set when \GEE\ fails.

For the first half of the proof, which follows \cite{Yu}, in particular Sections 8 and 9, we need to generalize some of Yu's results, which is done in \Cref{thm75,thm80}. Due to the more complicated structure of the intertwining set when \GEE\ fails, we also work with the image of the simply connected cover at times, on which the desired characters we work with vanish (see \Cref{cor:trivialchar}). We also introduce two general lemmas, \Cref{thm26,thm27}, that are used to avoid the need of \cite[Theorem~2.4]{Gerardin} in the second part of the proof of supercuspidality as G\'erardin's result does not apply to our Heisenberg--Weil representations construction for $p=2$.

We begin with a lemma that is not strictly necessary for the proof of supercuspidality, but that allows a better understanding of the structure of the group from which we induce when \GEE\ fails. We write $H_\alpha\defeq d\alpha^\vee(1)$.

\begin{lemma} \label{thm46}
Let $G'$ be a split reductive group over $F$ with split maximal torus~$T$.
Let $X\in\Lie^*(T)_r$ for some integer $r$, and
let $\overline X$ be the image of $X$ in $\Lie^*(T)_r/\Lie^*(T)_{r+}$.
Let $W = W(G',T)(F)$, and let
$W'$ be the subgroup of~$W$ generated by the reflections
$s_\alpha$ with $\val(X(H_\alpha)) > r$, for $\alpha \in \Phi(G',T)$.
Then the group $Z_W(\overline X)/W'$ is a $p$-group.
\end{lemma}

\begin{proof}
Using $\Lie^*(T)_{-r}/\Lie^*(T)_{(-r)+} \simeq X^*(T)\otimes k_F$,
the group $W'$ is generated by the reflections~$s_\alpha$
with $\overline X(H_\alpha) = 0$,
and we may apply \cite[Theorem~4.5]{Steinberg-torsion} to (in the notation of \cite{Steinberg-torsion}) $H=\<\overline X\>$ with $Z_w(H)^0=W'$ and $X/X^0$ being an $\bF_p$-vector space.
\end{proof}

The following is a generalization of \cite[Lemma~8.3]{Yu} and the second part of the proof is a variant of Yu's arguments.

\begin{lemma}[{cf.\ \cite[Lemma~8.3]{Yu}}] \label{thm75}
Let $H$ be a connected reductive $F$-group
and let $H' \subseteq H$ be a twisted Levi subgroup that splits over a tame extension of~$F$.
Let $\widetilde H$ be a possibly disconnected reductive $F$-group
and let $f\colon\widetilde H\to\Aut(H)$ be an algebraic action
such that the induced map $f^\circ: \widetilde H^\circ\to H^\tn{ad}$
is surjective with central kernel.

Let $X\in\Lie^*(H')^{H'}(F)$ be $(H,H')$-generic of depth~$-r$.
Then there is a subgroup $\widetilde H'$ of $N_{\widetilde H}(H')$
containing the identity component of $N_{\widetilde H}(H')$
such that:

\begin{enumerate}
\item\label{thm75a} If $h\in\widetilde H(F)$, and $Y_1,Y_2\in\Lie^*(H')_{x,-r}$
are regular semisimple that satisfy
\[
Y_1 \equiv Y_2 \equiv X \!\!\!\! \pmod{\Lie^*(H')_{x,(-r)+}} \quad \, \, \textnormal{and} \quad \, \, \Ad(h) Y_1 = Y_2,
\]
then $h\in\widetilde H'(F)$.

\item \label{thm75b}
There is a short exact sequence of groups
$1 \to A' \to \pi_0(\widetilde H')(F^\tn{sep}) \to A \to 1$
where $A\subseteq\pi_0(\widetilde H)(F^\tn{sep})$
and $A'$ is a $p$-group,
which is trivial if $X$ satisfies \GEE.
\end{enumerate}
\end{lemma}

\begin{proof}
To start with, we give the construction of~$\widetilde H'$.
Let $T$ be a maximal torus of $H'$ and let $E$ be a finite Galois field extension of $F$ over which $T$ is split. Note that we do not assume $E/F$ to be tamely ramified.
We may identify $\Lie^*(T)$ with the subspace
$\Lie^*(H')^T\subseteq\Lie^*(H')$
on which the adjoint action of $T$ is trivial,
and $X\in\Lie^*(T)$ under this identification
because $\Lie^*(T)\supseteq\Lie^*(H')^{H'}$.
Let $\widetilde N_T \defeq N_{\widetilde H}(H', T)$, let $N_T \defeq N_{\widetilde H^\circ}(H', T)$, and let $N'_T$ be the normalizer of $T$ in~${f^\circ}^{-1}(H'/Z(H))$.
So $N'_T\subseteq N_T\subseteq\widetilde N_T$, each group of finite index in the next one.
We define $\wt H'$ to be the $F$-subgroup of $N_{\wt H}(H')$ for which
\[
\widetilde H'_E = Z_{\widetilde N_T}(\overline X)_E\cdot {f^\circ}^{-1}(H'/Z(H))_E,
\]
where $\overline X$ is the image of $X$ modulo $\Lie^*(T)_{x,(-r)+}$.
Note that this construction is independent of the choice of $E$ as replacing $E$ by a larger field extension yields the same group $\wt H'$. 
Moreover, this construction is also independent of the choice of~$T$:
If $T'$ is any other maximal torus of $H'$,
then we may take $E$ to split both $T$ and~$T'$,
and the construction of $\widetilde H'_E$
produces the same group for both tori because
$T_E$ and~$T'_E$ are $H'(E)$-conjugate.

The exact sequence of \ref{thm75b} is constructed by observing that
$N'_T$ is a normal subgroup of $Z_{N_T}(\overline X)$
and $\pi_0(\widetilde H')$ fits into the short exact sequence
\[
\begin{tikzcd}
1 \rar & \dfrac{Z_{N_T}(\overline X)}{N'_T} \rar &
\pi_0(\widetilde H') \rar &
\dfrac{Z_{\widetilde N_T}(\overline X)}{Z_{N_T}(\overline X)} \rar & 1.
\end{tikzcd}
\]
By \Cref{thm46}, the first group in this sequence is a $p$-group,
and it follows immediately from the definition of \GEE\ (see \cpageref{GE2})
that this group is trivial if $X$ also satisfies \GEE.

It remains to prove \ref{thm75a}, so let $h$, $Y_1$ and $Y_2$ be as in \ref{thm75a}.
Set $T_j \defeq Z_{H'}(Y_j)^\circ$ for $j=1,2$, which is a maximal torus of $H'$.
Let $E/F$ be a finite Galois extension splitting $T_1$ and~$T_2$.
Take $h'\in {f^\circ}^{-1}(H'/Z(H))(E)$ such that $T_2 = f(h')(T_1)$.
Then $Y_0\defeq f(h')(Y_1)\in\Lie^*(T_2)_{-r}$ by \cite[Lemma~8.2]{Yu},
where we identify $\Lie^*(T_2)$ with the subspace
$\Lie^*(H')^{T_2}\subseteq\Lie^*(H')$
on which the adjoint action of $T_2$ is trivial.
We have $\Lie^*(T_2)\supseteq\Lie^*(H')^{H'}$,
so that $X\in\Lie^*(T_2)$ under this identification.
Using \cite[Lemma~8.2]{Yu} we obtain
\[
Y_0 \equiv f(h')(X) = X \equiv Y_2 \pmod{\Lie^*(T_2)_{(-r)+}}.
\]
The element $n\defeq hh'^{-1}\in\widetilde H(E)$ normalizes~$T_2$.
Moreover, 
\[
f(n) (X) \equiv f(n) (Y_0) = Y_2 \equiv X \pmod{\Lie^*(T_2)_{(-r)+}}.
\]
Since $X$ is $(H,H')$-generic of depth~$-r$,
if $\alpha\in\Phi(H,T_2)$, then 
\[
\val(X(H_\alpha)) = \begin{cases}
\infty & \tn{if $\alpha\in\Phi(H',T_2)$,} \\
-r     & \tn{if not.}
\end{cases}
\]
As $n\in\widetilde H(E)$ preserves the set $\Phi(H,T_2)$,
it follows that the element $n$ also preserves the subset $\Phi(H',T_2)$ and hence normalizes $H'$. Hence $n \in Z_{\wt N_{T_2}}(\overline X)(E)$, and therefore 
\[
h=nh' \in \wt H(F) \cap (Z_{\wt N_{T_2}}(\overline X)(E)\cdot{f^\circ}^{-1}(H'/Z(H))(E))=\wt H'(F). \qedhere
\]
\end{proof}

Replacing \cite[Lemma~8.3]{Yu} by \Cref{thm75} in Yu's work we obtain the analogue of (the first half of) \cite[Theorem~9.4]{Yu} in our more general setting. However, due to the more complicated intertwining set when \GEE\ fails, we will have to work with a more general statement that only considers the restriction to the image of the simply connected covers of appropriate groups. Recall that we denote the image of $G^\sc(F)$ in $G(F)$ by $G(F)^\natural$. %

\begin{lemma}[cf.\ {\cite[Theorem~9.4]{Yu}}] \label{thm80}
Let $\widetilde H$ be a possibly disconnected reductive $F$-group
with identity component~$H$,
let $H' \subseteq H$ be a twisted Levi subgroup that splits over a tame extension of~$F$, and
let $\phi\colon H'(F)\to\bbC^\times$
be a character that is $(H,H')$-generic relative to $x$ of depth~$r$. We write $c\colon H^\tn{sc}\to H$ for the simply-connected cover of $H$, and ${H^{\sc}}'\defeq H'\times_H H^\tn{sc}$.
Then there exists a subgroup $\widetilde H'$ of~$N_{\widetilde H}(H')$
with identity component~$H'$ satisfying the following properties: 
\begin{lemmaenum}
\item \label{thm80a}
If $h\in\widetilde H(F)$ intertwines
the restriction of $\hat\phi_{(H,x)}$ to
$(H',H)_{x,r,r/2+}\cap H(F)^\natural$,
then $h\in H(F)_{x,r/2}\cdot \widetilde H'(F)\cdot H(F)_{x,r/2}$.

\item \label{thm80b}
There is a short exact sequence of groups
$1 \to A' \to \pi_0(\widetilde H')(F^\tn{sep}) \to A \to 1$
where $A\subseteq\pi_0(\wt H)(F^\tn{sep})$
and $A'$ is a $p$-group, which is trivial
if $\phi\circ c|_{{H^{\sc}}'(F)}$ satisfies \GEE.

\end{lemmaenum}
In particular, if $\widetilde H$ is connected
and $p$ is not a torsion prime for $\widehat H^\tn{ad}$,
then $\widetilde H' = H'$.
\end{lemma}

\begin{proof}
Since $\Out(H^\tn{der})\subseteq\Out(H^\tn{sc})$,
where $\Out(-)$ denotes the algebraic group of outer automorphisms,
the conjugation action of $H(F)$ on $H^\tn{der}$
lifts uniquely to an action of $H(F)$ on $H^\tn{sc}$.
Passage to $F$-points gives an action of $H(F)$ on~$H^\tn{sc}(F)$
lifting the conjugation action on~$H^\tn{der}(F)$.

We denote by $\psi$ the composition of $c|_{({H^{\sc}}',H^{\sc})_{x,r,r/2+}}$ with $\hat\phi_{(H,x)}$. 
As $c({({H^{\sc}}',H^{\sc})_{x,r,r/2+}}) \subseteq (H',H)_{x,r,r/2+}\cap H(F)^\natural$, if an element $h \in \wt H(F)$ intertwines the restriction of $\hat\phi_{(H,x)}$ to $(H',H)_{x,r,r/2+}\cap H(F)^\natural$, then $h$ also intertwines\footnote{Here we use the above action of $H(F)$ on $H^{\sc}(F)$ and the following slight generalization
	of the usual notion of intertwining:
	given a group $A$, a subgroup $B$, 
	and a representation $\lambda$ of~$B$,
	an automorphism $\sigma$ of~$A$
	\mathdef{intertwines} $\lambda$
	if there is a nonzero $B\cap \sigma(B)$-equivariant
	homomorphism from $\lambda\circ\sigma^{-1}$ to $\lambda$.
} $\psi$.

 Since $\phi$ is $(H,H')$-generic relative to $x$ of depth~$r$, there exists an element $X \in \Lie^*(H')^{H'}(F)$ that is $(H, H')$-generic of depth $-r$ such that $\phi|_{H'(F)_{x,r}}$ is realized by $X$. By the construction of $\hat \phi_{(H,x)}$ its restriction to $(H',H)_{x,r,r/2+}$ is therefore also realized by $X \in (\fh')^*_{x,-r} \subseteq \Lie^*(H')(F) \subseteq \Lie^*(H)(F)$ (where the last inclusion is obtained by identifying $\Lie^*(H')(F)$ with $\Lie^*(H)^{Z(H')}(F)$),
 i.e., is given by composing 
 $$(H',H)_{x,r,r/2+}/H_{x,r+} \simeq (\fh',\fh)_{x,r,r/2+}/\fh_{x,r+}$$ with $\Lambda \circ X$, where $\Lambda: F \ra \bC^\times$ is nontrivial on $\cO_F$, but trivial on the maximal ideal of $\cO_F$.
 By precomposition with $\Lie(H^{\sc}) \rightarrow \Lie(H)$, we can view $X$ also as an element in $\Lie^*(H^{\sc})(F)$. Note that then $X \in \Lie^*({H^{\sc}}')_{x,-r} \cap \Lie^*({H^{\sc}}')^{{H^{\sc}}'}(F)$, and since the derivative of $c$ maps $H_\alpha$ in $\Lie(H^{\sc})(F)$ to $H_\alpha$ in $\Lie(H)(F)$, the element $X$ is $({H^{\sc}},{H^{\sc}}')$-generic of depth $-r$.
  Since the diagram 
  \[
  \begin{tikzcd}
   ({H^\tn{sc}}',H^\tn{sc})_{x,r,r/2+}/H^\tn{sc}_{x,r+} \rar \dar &
  (H',H)_{x,r,r/2+}/H_{x,r+} \dar \\
   ({\frak h^\tn{sc}}',\frak h^\tn{sc})_{x,r,r/2+}/\frak h^\tn{sc}_{x,r+} \rar &
  (\frak h',\frak h)_{x,r,r/2+}/\frak h_{x,r+}
  \end{tikzcd}
  \]
  commutes by the construction of the Moy--Prasad isomorphism
  given in the proof of \cite[Theorem~13.5.1]{Kaletha-Prasad-BTbook}
  together with the functoriality of the Moy--Prasad isomorphism for tori
  \cite[Proposition~B.6.9]{Kaletha-Prasad-BTbook},
  the character $\psi$ is represented by $X$. Hence, in particular, also the restriction $(\phi \circ c)|_{{H^{\sc}}'_{x,r}}$ is represented by $X$ and therefore $\phi \circ c|_{{H^{\sc}}'(F)}$ is $(H^{\sc},{H^{\sc}}')$-generic relative to $x$ of depth~$r$ (in our sense). Moreover, by the definition of \GEE\ (see \cpageref{GE2}), the character $\phi \circ c|_{{H^{\sc}}'(F)}$ satisfies \GEE\ if and only if $X \in \Lie^*({H^{\sc}}')^{{H^{\sc}}'}(F)$ satisfies \GEE.
  
  Let $\wt H'$ be the group obtained from \Cref{thm75} applied to the group $\widetilde H$ acting on $H^\tn{sc}$,
  the twisted Levi subgroup ${H^{\sc}}'\subseteq H^\tn{sc}$,
  and the Lie algebra element $X \in \Lie^*({H^{\sc}}')^{{H^{\sc}}'}(F)$.

Now we can generalize the arguments from the proof of \cite[Theorem~9.4]{Yu} as follows to apply to our more general setting of $h \in \wt H(F)$ intertwining $\psi$ using \Cref{thm75} in place of \cite[Lemma~8.3]{Yu}.

More precisely, suppose $h \in \wt H(F)$ intertwines the restriction of $\hat\phi_{(H,x)}$ to $(H',H)_{x,r,r/2+}\cap H(F)^\natural$. Then $h$ also intertwines $\psi$.
Now note that Proposition~1.6.7 of \cite{Adler} (which Yu recorded in his setting as \cite[Theorem~5.1]{Yu}) holds if, in the notation there,
$g$ is an algebraic $F$-automorphism of~$G$
rather than simply (conjugation by) an element of~$G$.
Thus there are regular semisimple elements $Y_1,Y_2$
in $X+({\fh^{\sc}}',\frak h^\tn{sc})^*_{x,(-r)+,-r/2}$
such that $\Ad(h)Y_1=Y_2$.
Using \cite[Lemma~8.6]{Yu},
we can find $k_1,k_2\in H^\tn{sc}(F)_{x,r/2}$ such that
$Z_i\defeq \Ad(k_i)Y_i \in X + {\fh^{\sc}}'^*_{x,(-r)+}$ for $i=1,2$.
But $Z_i = \Ad(c(k_i))Y_i$ as well,
and $c(k_i)\in H(F)_{x,r/2}$.
The element $h' \defeq c(k_2)\cdot h\cdot c(k_1)^{-1}$ satisfies
$\Ad(h')Z_1 = Z_2$.
Using \Cref{thm75}, we obtain that $h' \in \wt H'(F)$, as desired.

The last claim follows from $(H^{\sc},{H^{\sc}}')$-generic characters (in our sense) of ${H^{\sc}}'(F)$ automatically satisfying \GEE\ if $p$ is not a torsion prime for $\widehat{H^{\sc}}=\widehat H^{\ad}$ by \cite[Lemma~8.1]{Yu}.
\end{proof}

\begin{lemma} \label{lemma:trivialchar}
	Let $H$ be a reductive group over $F$ and let $y \in \sB(H,F)$.  Let $\varphi$ be a character of $H(F)$. Then the restriction of $\varphi$ to the intersection $H(F)_{y,0+}^\natural$ of the image of the simply connected cover $H(F)^\natural$ and $H(F)_{y,0+}$ is trivial.
\end{lemma}

\begin{proof}
	\addtocounter{equation}{-1}
	\begin{subequations} 
We start by reviewing a few facts about semisimple anisotropic groups.
Recall that a reductive group is \mathdef{isotropic}
if it contains a unipotent element,
or equivalently, if its derived subgroup contains a nontrivial split torus.
First, if $H$ is a simply-connected isotropic $F$-group,
then $H(F) = [H(F),H(F)]$ by \cite[6.15]{PrasadRaghunathan}.
Hence $H(F)$ has no nontrivial characters in this case.
Second, if $H$ is semisimple and anisotropic
then $H$ splits over an unramified extension, $H$ is of type~$A$,
and the quasi-split inner form of~$H$ is split
(see \cite[Remark~10.3.2]{Kaletha-Prasad-BTbook}).
Hence $H^\tn{sc}(F)$ is compact and a product of groups of the form
$\SL_1(D)$ where $D$ is a division algebra
over a finite separable extension of~$F$.
Third, the derived subgroup of $\SL_1(D)$
is the pro-unipotent radical $\SL_1(D)_{0+}$
(see \cite[page~504 and Corollary on page~521]{Riehm70}).

Let $c\colon H^\tn{sc}\to H$ be the simply-connected covering map.
Let $H_i$, $1\leq i\leq n$, be the almost-simple subgroups of~$H$
corresponding to the irreducible factors of the relative Dynkin diagram of~$H$.
Using $H^\tn{sc} = \prod_{i=1}^n H_i^\tn{sc}$,
we can factor any element $h\in H(F)^\natural_{y,0+}$
as a product $h_1h_2\cdots h_n$
with $h_i$ in the image of $H_i^\tn{sc}(F)$.
At the same time, $\varphi$ is trivial on the image of $H_i^\tn{sc}(F)$
whenever $H_i$ is isotropic.
Replacing $H$ by the subgroup generated by the anisotropic~$H_i$,
we are reduced to the case where $H$ is anisotropic and semisimple,
which we assume for the rest of the proof.

To finish the proof, it suffices to show that 
\begin{equation} \label{thm74}
c^{-1}(H(F)_{0+}) = \ker(c)\cdot H^\tn{sc}(F)_{0+},
\end{equation}
since the restriction of $\varphi$ to $H(F)^\natural_{0+}$
inflates to a character of $c^{-1}(H(F)_{0+})$
restricted from a character of $H^\tn{sc}(F)$.
To prove \eqref{thm74},
note that if $H\to H'$ is an isogeny
and \eqref{thm74} holds with $H$ replaced by~$H'$,
then \eqref{thm74} holds for~$H$.
Hence we may assume that $H$ is an adjoint group.
But then $H$ is a product of almost-simple groups,
and each factor is therefore of the form $\Res_{E/F}(H')$,
where $E/F$ is a finite separable extension by
\cite[6.21(ii)]{Borel-Tits65}.
The $E$-group $H'$ is anisotropic,
so that $H'(F) = \PGL_1(D)$ for some division algebra over~$E$.
Since $H'(F)_{0+} = \PGL_1(D)_{0+}$
(see \cite[Proposition~A.12]{Fintzen-MPWeilrestriction}),
after replacing $E$ by $F$,
we are reduced to the case where $H^\tn{sc} = \SL_1(D)$
and $H = \PGL_1(D)$ for $D$ a division algebra over~$F$.

Let $\dim_F(D) = n^2$.
Now $\PGL_1(D)_{0+} \simeq D^\times_{0+}/F^\times_{0+}$
by \cite[Lemma~3.3.2(1)]{Kaletha-regular},
so \eqref{thm74} amounts to the claim that
if $z\in\SL_1(D)$ satisfies
$z\in F^\times\cdot D^\times_{0+}$,
then $z\in\mu_n(F)\cdot\SL_1(D)_{0+}$, where $\mu_n(F)$ denotes $n$th roots of~$1$ in $F$.
This follows from the fact that if $a\in F^\times$
satisfies $a^n\in F^\times_{0+}$,
then $a\in\mu_n(F)\cdot F^\times_{0+}$.
\end{subequations}
\end{proof}

\begin{corollary} \label{cor:trivialchar}
	Let $1 \leq i \leq n$, let $y \in \sB(G_{n+1},F)$ and $g \in N_{G}(G_{i+1}, G_{n+1})(F)$. Then the restriction of ${^g\phi_i}$ to $(G_{i+1})_{y,0+}^\natural$ is trivial. In particular, the restriction of ${^g\phi_i}$ to $U_{j}^\pm(F)_{y,{r_j}/{2}}$ is trivial for all $i < j \leq n$.
\end{corollary}
\begin{proof}
Apply \Cref{lemma:trivialchar} to the group
$H=G_{i+1}$ and the character $\varphi={}^g\phi_i$.
\end{proof}

In order to generalize the second half of the proof of supercuspidality in \cite[Theorem~3.1]{Fi-Yu-works} we need two more lemmas that allow us to avoid \cite[Theorem~2.4]{Gerardin}, which is not available for the Heisenberg--Weil representations in characteristic 2.

\begin{lemma} \label{thm26}
Let $P$ and $H$ be finite groups and $U\unlhd P$ a normal subgroup.
Let $P$ act on~$H$ by automorphisms,
and let $(\pi, V)$ be a representation of $P \ltimes H$ such that $\pi|_H$ is irreducible.
Suppose that $U$ acts trivially on $H$ and $U\subseteq[P,U]$.
Then $\pi|_U$ is trivial.
\end{lemma}

\begin{proof}
Since $U$ and $H$ commute,
the elements of $\pi(U)$ act $H$-equivariantly on~$V$,
hence are scalars by Schur's Lemma.
So $\pi|_U$ is isotypic for a character $\phi$ of~$U$.
Moreover, since $P$ normalizes $U$, it normalizes $\pi|_U$, hence~$\phi$.
In other words, $\phi([p,u]) = 1$ for all $p\in P$ and $u\in U$.
Hence $\phi$ is trivial.
\end{proof}

\begin{lemma} \label{thm27}
Let $k$ be a field,
let $H$ be a quasi-split
reductive $k$-group,
let $P$ be a parabolic subgroup of~$H$,
and let $U$ be the unipotent radical of~$P$.
If $|k|>3$, then $U(k)\subseteq[P(k),U(k)]$.
\end{lemma}

\begin{proof}
Let $S$ be a maximal split torus of~$H$ contained in~$P$
and fix $\alpha\in\Phi(G,S)$.
We will show that $U_\alpha(k)\subseteq[P(k),U(k)]$.
Note that there is $s\in S(k)$ such that $\alpha(s)\neq 1$:
Since $|k|>3$, there is $t\in k^\times$ such that $t^2\neq1$,
and then we may take $s=\alpha^\vee(t)$
because $\alpha(\alpha^\vee(t)) = t^2 \neq 1$.

It suffices to show that $U_\alpha(k)\subseteq[S(k),U_\alpha(k)]$.
Take $s\in S(k)$ such that $\alpha(s)\neq1$.
If $2\alpha\notin\Phi(G,S)$, then $U_\alpha(k)$ is abelian and
$S(k)$-equivariantly isomorphic to its Lie algebra $\frak g_\alpha(k)$,
and the claim follows from the fact that the endomorphism
$\Ad(s)-1$ of $\frak g_\alpha$ is multiplication by $\alpha(s)-1$ and hence invertible.
If $2\alpha\in\Phi(G,S)$, then %
 the quotient $U_{\alpha}(k)/U_{2\alpha}(k)$
is $S(k)$-equivariantly isomorphic to the root space $\frak g_\alpha(k)$,
and the same argument as above shows that
for every $u\in U_{\alpha}(k)$ there is $u'\in U_{\alpha}(k)$
and $s\in S(k)$ such that $[s,u'] = u \pmod{U_{2\alpha}(k)}$.
Now we are done because by the previous case,
$U_{2\alpha}(k)\subseteq[S(k),U_{2\alpha}(k)]$.
\end{proof}

Now we are in a position to prove the key intertwining result,
\Cref{thm02}, following the proof of \cite[Theorem~3.1]{Fi-Yu-works}.
This result will immediately imply the main theorem, \Cref{thmseveralsc}.

\begin{theorem} \label{thm02}
Let $\wt K = N_{\Kplus}(\rho\otimes\kappa)$
and let $\sigma\in\Irr(\wt K,K,\rho\otimes\kappa)$.
Suppose $q>3$.

\begin{theoremenum}
\item \label{thm02a}
If $g\in G(F)$ intertwines $\sigma$, then $g\in \wt K$.

\item \label{thm02b}
The group $\widetilde K/K$ is a finite $p$-group
that is trivial if all $\phi_i$ satisfy \GEE,
for example, if $p$ is not a torsion prime for $\widehat G$.
\end{theoremenum}

\end{theorem}

\begin{proof}
	Suppose $g \in G(F)$ intertwines $\sigma$.
	Since $\sigma$ restricted to the normal subgroup $K \trianglelefteq \wtK=N_{\Kplus}(\rho\otimes\kappa)$ is $(\rho \otimes \kappa)$-isotypic, the element $g$ also intertwines $(K, \rho \otimes \kappa)$. Moreover, by \Cref{lemma:Knormalizeskappa-}, $\sigma$ further restricted to $K^- \trianglelefteq \wt K$ is a direct sum of copies of $(\rho \otimes \kappa^-)$, and hence $g$ also intertwines $(K^-, \rho \otimes \kappa^-)$. 

We first claim that  $g \in K \wtG_{n+1}(F) K$ for some subgroup $\wtG_{n+1}\subseteq N_G(G_1, \hdots, G_n, G_{n+1})$ whose identity component is $G_{n+1}$ and whose component group is a finite $p$-group that is trivial if all $\phi_i$ satisfy~\GEE. 
We show this by induction following the first part of the proof of \cite[Theorem~3.1]{Fi-Yu-works}, focusing on the differences arising from our more general setup.
For technical reasons, stemming from the fact that
the adjoint quotient of a reductive group can have more torsion primes than the original group,
we divide the argument into two cases: either all $\phi_i$ satisfy~\GEE, or some~$\phi_i$ does not.

In the first case, let $1\leq i\leq n$
and suppose the induction hypothesis that $g\in K G_i(F) K$.
We will show that $g\in K G_{i+1}(F) K$.
Since $K$ intertwines $\rho \otimes \kappa^-$ by \Cref{lemma:Knormalizeskappa-},
we may assume that $g \in G_i(F)$.
As in \cite[Theorem~3.1]{Fi-Yu-works},
the restriction of $\rho\otimes\kappa^\minus$
to $(G_i)_{x,r_i,(r_i/2)+}$ is the restriction of $\prod_{j=1}^i \hat \phi_j$.
Hence $g$ intertwines $(\prod_{j=1}^i \hat \phi_j)|_{(G_i)_{x,r_i, (r_i/2)+}}$.
Moreover, if $1\leq j<i$, then $(\hat\phi_j)|_{(G_i)_{x,r_i,(r_i/2)+}}$
extends to a character of~$G_i(F)$, namely $\phi_j|_{G_i(F)}$.
So $g$ also intertwines $(\hat\phi_i)|_{(G_i)_{x,r_i,(r_i/2)+}}$
and $g\in KG_{i+1}(F)K$ by \cite[Theorem~9.4]{Yu},
completing the induction step.

The second case is similar, but requires us to deal with the complication that the character $\phi_j$ of $G_{j+1}(F)$ is not necessarily intertwined by disconnected groups containing $G_{j+1}$.
Let $1 \leq i \leq n$ and suppose the induction hypothesis that $g \in K \wt G_i(F) K$ where $\wt G_i$ is  a subgroup of $N_G(G_1, \hdots, G_{i-1}, G_i)$ whose identity component is $G_i$ and whose component group is a finite $p$-group.
Let $\widetilde G_{i+1}$ be the group $\widetilde H'$ from \Cref{thm80}
applied to $\widetilde H = \widetilde G_i$, $H'=G_{i+1}$, and $\phi=\phi_i$.
Then by induction, $\widetilde G_{i+1}\subseteq N_{\widetilde G_i}(G_{i+1})
\subseteq N_G(G_1,\dots,G_i,G_{i+1})$
and $\pi_0(\widetilde G_{i+1})(F^\tn{sep})$ is a finite $p$-group.
	We will show that $g \in K \wt G_{i+1}(F) K$. Since $K$ intertwines $\rho \otimes \kappa^-$ by \Cref{lemma:Knormalizeskappa-}, we may assume that $g \in \wt G_i(F)$. As in \cite[Theorem~3.1]{Fi-Yu-works}, the restriction of $\rho \otimes \kappa^-$ to $(G_i)_{x,r_i, (r_i/2)+}$ is the restriction of $\prod_{j=1}^i \hat \phi_j$. Hence $g$ intertwines $(\prod_{j=1}^i \hat \phi_j)|_{(G_i)_{x,r_i, (r_i/2)+}}$. Moreover,  for $1 \leq j \leq i-1$, the restriction of $\hat \phi_j$ to the intersection $(G_i)^\natural_{x,r_i, (r_i/2)+}$ of $(G_i)_{x,r_i, (r_i/2)+}$ with the image of $G_i^\tn{sc}(F)$ agrees with the restriction of $\phi_j$ and is trivial by \Cref{cor:trivialchar}.
	Hence $g$ intertwines $\hat \phi_i|_{(G_i)^\natural_{x,r_i, (r_i/2)+}}$. Thus \Cref{thm80} implies that $g \in G_i(F)_{x,r_i/2} \wt G_{i+1}(F) G_i(F)_{x,r_i/2} \subseteq K\wt G_{i+1}(F)K$, which finishes the induction step.

	Therefore, by induction $g \in K \wtG_{n+1}(F) K$, and to finish the proof of Part (a) we may assume that $g \in \wt G_{n+1}(F)$. It suffices to show that $g\in \widetilde G_{n+1}(F)_{[x]}$, because then $g$ normalizes $K$ and hence, since $g$ intertwines $\rho \otimes \kappa$, we have $g\in N_{\Kplus}(\rho\otimes\kappa)=\wt K$, as desired.

	Suppose to the contrary that $g[x]\neq [x]$. Note that $gx \in \sB(G_{n+1}, F)$, so we can 
	choose a tame maximal, maximally split torus $T$ of $G_{n+1}$ whose associated apartment contains $x$ and $gx$. We let $\lambda\in X_*(T)^{\Gal(E/F)}\otimes\bbR$ be such that $gx = x+ \lambda $, where $E$ denotes a splitting field of $T$.
	We are now in the setting of \Cref{sec:aux} and use the notation defined there. Note in particular that $\sfU(k_F)$ is non-trivial, because $x$ is a minimal facet.

	Let $f$ be a nonzero element of
	$\Hom_{K^-\cap {}^gK^-}({}^g(\rho\otimes\kappa^-),(\rho\otimes\kappa^-))$
	and let $V_f$ be the image of~$f$. We adjust the arguments of the bottom of page~2739 and the top of page~2740 of \cite{Fi-Yu-works} to our setting, using the image of the simply connected cover at various places instead of the groups considered in \cite{Fi-Yu-works}, and using %
	 \Cref{cor:trivialchar}, to show that 
	\[
	V_f \subseteq V_\rho\otimes_\bbC\bigotimes_{i=1}^n V_{\WeilRep_i}^{U_i^+(F)_{x,{r_i}/{2}}},
	\]
	and that the action of
	$$U \defeq ((G_{n+1})_{x,0} \cap (G_{n+1})^\natural_{gx,0+})(G_{n+1})^\natural_{x,0+}$$
	on $V_f$ via $\rho \otimes \kappa^-$ is trivial. Noting that the image of $U$ in $\sfG_{n+1}(k_F)$ is $\sfU(k_F)$, it will then suffice to show that $U$ acts also trivially on $V_{\WeilRep_i}^{U_i^+(F)_{x,{r_i}/{2}}}$ for $1 \leq i \leq n$, because this will contradict that $\rho$ is cuspidal.

	For the convenience of the reader, we spell out a few more details. The restriction of $\rho \otimes \kappa^-$ to  $(G_{n+1})^\natural_{x,0+}$ is the restriction of the character $\prod_{i = 1}^n \phi_i|_{G_{n+1}(F)}$ (times the identity), and hence is the identity by \Cref{cor:trivialchar}. Recall that $gx \in \sB(G_{n+1}, F)$ and ${^gG_{n+1}}=G_{n+1}$. Hence the group $(G_{n+1})_{x,0} \cap (G_{n+1})^\natural_{gx,0+}$ acts on $V_f$ via the restriction of the character $\prod_{i = 1}^n {^g\phi_i}|_{G_{n+1}(F)}$, whose restriction to $(G_{n+1})^\natural_{gx,0+}$ is also trivial by \Cref{cor:trivialchar}. Thus the action of $U$
	on $V_f$ via $\rho \otimes \kappa^-$ is trivial, as desired. 
	Moreover, by definition of $\lambda$ and $U_i^+(F)_{x,r_i/2}$, we have $U_i^+(F)_{x,r_i/2} \subseteq (G_i)_{gx,r_i, r_i/2+}$, and hence $U_i^+(F)_{x,r_i/2}$ acts on $V_f$ via the restriction of the character $\prod_{j=1}^{i-1}{^g\phi_j}|_{G_{n+1}(F)}$ to $U_i^+(F)_{x,r_i/2}$, hence by \Cref{cor:trivialchar}, the action of $U_i^+(F)_{x,r_i/2}$ on $V_f$ is trivial. On the other hand, also using \Cref{cor:trivialchar}, $U_i^+(F)_{x,r_i/2}$ acts also trivially on $V_{\omega_j}$ for $1 \leq j \leq n$ with $i \neq j$. Hence we conclude that 
	\(
	V_f \subseteq V_\rho\otimes_\bbC\bigotimes_{i=1}^n V_{\WeilRep_i}^{U_i^+(F)_{x,{r_i}/{2}}},
	\) 
	as claimed.

	To finish the proof of Part (a), it suffices to show that the action of $U$ on $V_{\WeilRep_i}^{U_i^+(F)_{x,{r_i}/{2}}}$ is trivial. Since $U \subseteq (G_{n+1})_{gx,0+}^\natural \cap (G_{n+1})_{x,0+}^\natural$, the restriction of $\phi_i$ to $U$ is trivial by \Cref{cor:trivialchar}, so it suffices to show that the restriction of the Heisenberg--Weil
	representation to $\sfU(k_F)$ acting on $V_{\WeilRep_i}^{U_i^+(F)_{x,{r_i}/{2}}}=V_{\WeilRep_i}^{\sfV_i^+}$ is trivial.
	Since the image $\sfV_i^+$ of ${U_i^+(F)_{x,{r_i}/{2}}}$ in
	the Heisenberg $\bbF_p$-group $\sfV_i^\natural$
	is a splitting of an isotropic subspace,
	we can identify $V_{\WeilRep_i}^{\sfV_i^+}$ as a representation of $\sfV_{i,0}^\natural$
	with the irreducible Heisenberg representation for $\sfV_{i,0}^\natural$ (with same central character) by \Cref{thm15b}.
	At the same time, by \Cref{thm16a},
	the group $\sfP(k_F)$ acts by conjugation on $\sfV_i^+$ and on the quotient 
	\[
	\sfV_{i,0}^\natural \simeq \sfV_{i,0}^\natural \sfV_i^+
	/\sfV_i^+,
	\]
	and the action of the subgroup $\sfU(k_F)$ on the quotient is trivial. Hence the action of $\sfP(k_F)$ on $V_{\WeilRep_i}$ preserves $V_{\WeilRep_i}^{\sfV_i^+}$.
	Since $q>3$, we may apply \Cref{thm27} to show that $\sfU(k_F)\subseteq[\sfP(k_F),\sfU(k_F)]$,
	and then apply \Cref{thm26} to conclude that $\sfU(k_F)$ acts trivially on $V_{\WeilRep_i}^{\sfV_i^+}=V_{\WeilRep_i}^{U_i^+(F)_{x,{r_i}/{2}}}$.
This completes the proof of (a).

For the second part, we observed earlier in the proof,
during the inductive argument,
that $\pi_0(\widetilde G_{n+1})(F^\tn{sep})$ is a finite $p$-group
that is trivial if all $\phi_i$ satisfy~\GEE.
Moreover, we have seen that if $g$ intertwines $\sigma$, then $g \in K\wtG_{n+1}(F)_{[x]}K=K\cdot \wtG_{n+1}(F)_{[x]}$, and since all elements of $\wt K$ intertwine $\sigma$, we have  
$K\trianglelefteq\widetilde K\trianglelefteq K\cdot\widetilde G_{n+1}(F)_{[x]}$. Hence
there is a chain of inclusions
\[
\frac{\widetilde K}{K} \subseteq
\frac{\widetilde G_{n+1}(F)_{[x]}}{G_{n+1}(F)_{[x]}} \subseteq
\frac{\widetilde G_{n+1}(F)}{G_{n+1}(F)}\subseteq \pi_0(\widetilde G_{n+1})(F),
\]
and so $\widetilde K/K$  is also a finite $p$-group
that is trivial if all $\phi_i$ satisfy~\GEE.
\end{proof}

\begin{theorem}\label{thmseveralsc}
Let $\wt K = N_{\Kplus}(\rho\otimes\kappa)$
and let $\sigma\in\Irr(\wt K,K,\rho\otimes\kappa)$.
Suppose $q>3$.
\begin{theoremenum}
\item \label{thmseveralscC}
Let $\sigma\in\Irr(\widetilde K,K,\rho\otimes\sigma)$.
Then $\cind_{\wt K}^{G(F)}(\sigma)$ is irreducible supercuspidal.

\item \label{thmseveralscA}
$\displaystyle
\cind_K^{G(F)}(\rho\otimes\kappa)
\simeq \bigoplus_{\sigma \in \Irr(\wt K,K,\rho\otimes\kappa)}
\big(\cind_{\widetilde K}^{G(F)}(\sigma)\big)^{\oplus m_\sigma}$
with $m_\sigma\in \bN_{\geq 1}$.

\item \label{thmseveralscB}
If all $\phi_i$ satisfy~\GEE,
for example, if $p$ is a torsion prime for $\widehat G$,
then $\widetilde K = K$ and
the representation $\cind_K^G(\rho\otimes\kappa)$ is irreducible supercuspidal.
\end{theoremenum}

\end{theorem}

\begin{proof}
	By \Cref{thm02a}, if $g \in G(F)$ intertwines $\sigma$, then $g\in \wt K=N_{\Kplus}(\rho\otimes\kappa)$. Moreover, $K$ is a normal, finite-index subgroup of $\Kplus$ that contains $Z(G(F))$ and $K/Z(G(F))$ is compact. Now the result follows by applying \Cref{thm71},
	using the well-known fact that for such a compactly-induced representation, irreducibility implies supercuspidality (see \cite[Lemma~3.2.1]{Fintzen-IHES}).
\end{proof}

\numberwithin{equation}{section}
\appendix
\section{Alternating, symmetric, and quadratic forms} \label{sec:forms}
In this section we review the notions of symmetric forms,
alternating forms, and quadratic forms,
paying close attention to the features of these objects in characteristic~$2$
(see \cite[pp.~xvii--xxi]{Involutions}).
Fix a base field $k$ and a finite-dimensional $k$-vector space~$V$.

Let $B$ be a bilinear form on~$V$.
Recall that $B$ is
\begin{itemize}
\item
\mathdef{alternating} if $B(v,v) = 0$ for all $v\in V$,
\index[terminology]{form!alternating}

\item
\mathdef{symmetric} if $B(v,w) = B(w,v)$ for all $v,w\in V$, and
\index[terminology]{form!symmetric}

\item
\mathdef{skew-symmetric} if $B(v,w) = -B(w,v)$ for all $v,w\in V$.
\index[terminology]{form!skew-symmetric}

\end{itemize}
If $\tn{char}(k)\neq2$, then alternating is equivalent to skew-symmetric.
If $\tn{char}(k)=2$, then skew-symmetric is equivalent to symmetric. %
Moreover, if $\tn{char}(k)=2$, then 
alternating implies skew-symmetric but not conversely,
as we see by considering the simplest nontrivial bilinear pairing,
$(a,b)\mapsto a\cdot b$ on the one-dimensional vector space~$k$.

A bilinear form $B$ is called \mathdef{nondegenerate}\index[terminology]{form!nondegenerate} if for every nonzero $v \in V$ there exists $w \in W$ such that $B(v,w)\neq 0$, and a nondegenerate alternating bilinear form is called a \mathdef{symplectic form}.\index[terminology]{form!symplectic}

A \mathdef{quadratic form}\index[terminology]{form!quadratic}
$Q$ on~$V$ is an element of $\Sym^2(V^*)$,
that is, a homogeneous polynomial function on~$V$ of degree~$2$.
Any quadratic form~$Q$ defines a symmetric bilinear form
$B_Q\in\Sym^2(V)^*$ by the formula\index[notation]{BQ@$B_Q$}
\[
B_Q\colon (v,w) \mapsto Q(v+w) - Q(v) - Q(w).
\]
A quadratic form $Q$ is defined to be \mathdef{nondegenerate} if $B_Q$ is nondegenerate.

The assignment $Q\mapsto B_Q$ defines a map 
\[
\Sym^2(V^*) \to \Sym^2(V)^*
\]
whose behavior depends on $\tn{char}(k)$.
If $\tn{char}(k)\neq 2$, then the map is an isomorphism
with inverse $B \mapsto(v\mapsto\tfrac12 B(v,v))$,
giving a bijection between quadratic forms
and symmetric bilinear forms.
If $\tn{char}(k)=2$, then the map $Q\mapsto B_Q$ is not an isomorphism.
Instead, its kernel is the space $(V^*)^{(2)}$
of diagonal quadratic forms and therefore, by a dimension count,
its image is the space $\Alt^2(V)^*$ of alternating bilinear forms.

Assume for simplicity in the remainder of this section that $\dim(V)=2n$ is even.

Let $\omega$ be a nondegenerate alternating form.
A subspace $W$ of~$V$ is \mathdef{isotropic} if $\omega(w,w') = 0$ for all $w,w'\in W$.
A \mathdef{partial polarization} of~$V$ is a decomposition
$V = V^+\oplus V_0\oplus V^-$ in which $V^+$ and $V^-$ are isotropic,
$V_0$ is orthogonal to $V^+\oplus V^-$,
and the restriction of $\omega$ to $V_0$ is nondegenerate.
A \mathdef{polarization} is a partial polarization in which $V_0=0$.

Similarly, let $Q$ be a quadratic form.
A subspace $W$ of~$V$ is \mathdef{isotropic}\index[terminology]{subspace!isotropic}
if every $w\in W$ is isotropic,
meaning that $Q(w)=0$.%
\footnote{Some authors call a subspace ``isotropic'' if it contains some isotropic vector
and ``totally isotropic'' if every vector is isotropic.} 
A subspace $W$ of~$V$ is called \mathdef{anisotropic}\index[terminology]{subspace!anisotropic}
if  $Q(w)\neq 0$ for every $w\in W-\{0\}$.
A \mathdef{partial polarization} of~$V$ is a decomposition
$V = V^+\oplus V_0\oplus V^-$ in which $V^+$ and $V^-$ are isotropic,
$V_0$ is orthogonal to $V^+\oplus V^-$ (with respect to~$B_Q$),
and the restriction of $Q$ to $V_0$ is nondegenerate.
A \mathdef{polarization}\index[terminology]{partial polarization@(partial) polarization}
is a partial polarization in which $V_0$ is anisotropic.

The \mathdef{Witt index} %
of~$Q$ is the dimension of one
(equivalently, by a theorem of Witt, any \cite[Section~I.4]{Lam05})
maximal isotropic subspace of~$V$.
We say $Q$ is \mathdef{split} if $Q$ has Witt index~$n$,
in which case $(V,Q)$ is isomorphic to
the $k$-vector space $k^{2n}$ equipped with the quadratic form
\begin{equation} \label{thm32}
Q\colon \sum_{i=1}^n (x_i e_i + x_{-i}e_{-i})
\longmapsto \sum_{i=1}^n x_i\cdot x_{-i},
\end{equation}
where $\{e_i : i\in\{\pm1,\dots,\pm n\}\}$ is a basis of $V$.
If the Witt index of~$Q$ is $n-1$, then
there is a separable quadratic extension $\ell/k$ such that
$(V,Q)$ is isomorphic to the $k$-vector space $k^{2n-2}\oplus\ell$
equipped with the quadratic form
\begin{equation} \label{thm17}
Q\colon \sum_{i=1}^{n-1} (x_ie_i + x_{-i}e_{-i}) + y e_0
\longmapsto \sum_{i=1}^{n-1} x_i\cdot x_{-i} + \Nm_{\ell/k}(y),
\qquad x_i\in k,\; y\in\ell, 
\end{equation}
where $\{e_i : i\in\{\pm1,\dots,\pm(n-1)\}\}$ is a basis of
$k^{2n-2}$ and $e_0$ is a non-zero element of $\ell$.

Given $Q$ nondegenerate, we can form the orthogonal group
$\O(V) = \O(Q)$ of $g\in\GL(V)$ that preserve~$Q$,
meaning that $Q(gv) = Q(v)$ for all $v\in V$.
Let $\SO(V)$ be the index-two subgroup of~$\O(V)$
defined as the kernel of a map $\O(V)\to\bbZ/2\bbZ$
which is the determinant if $\tn{char}(k)\neq2$
and the Dickson invariant if $\tn{char}(k)=2$.
See \cite[Appendix~C.2]{ConradSGA3} for more discussion
of the definition of the special orthogonal group
in characteristic~$2$.
The group $\SO(V)$ is reductive and of type~$D_n$
over the algebraic closure of~$k$.
Moreover, $\SO(V)$ is split if and only if $Q$ is split.
As \cite[Table~II]{TitsBoulder} explains,
the group $\SO(V)$ is quasi-split but not split if and only if $Q$ has Witt index~$n-1$.

\section{Basic Clifford theory and intertwining}
\label{sec:clifford}
Let $B$ be a group, let $C$ be a finite-index normal subgroup of~$B$,
and let $\rho$ be an irreducible representation of~$C$.
Clifford theory concerns two closely related problems:
decomposing the induced representation $\Ind_C^B(\rho)$
and describing the set $\Irr(B,C,\rho)$ of irreducible representations of~$B$
whose restriction to~$C$ contains $\rho$.
In this \namecref{sec:clifford} we collect
some results from basic Clifford theory
and combine them with the classical intertwining criterion
for irreducibility of a compactly-induced representation.

\begin{lemma} \label{thm71}
	Let $C \trianglelefteq B \leq A$ be groups with $C$ normal and finite-index in~$B$,
	and let $\rho$ be a finite-dimensional irreducible representation of~$C$.
	\begin{lemmaenum}
		\item \label{thm71a}
		Sending $\sigma$ to the $\sigma$-isotypic component of $\Ind_C^B(\rho)$
		defines a bijection
		\[
			\Irr(B,C,\rho) \longleftrightarrow
			\Irr\bigl(\End_B(\Ind_C^B(\rho))\bigr).
		\]
	\end{lemmaenum}

	Suppose in addition that $A$ is locally profinite and $\rho$ is smooth.
	Then
	\begin{lemmaenum}[resume]
		\item \label{thm71b}
		$\displaystyle\cind_C^A(\rho) \simeq \bigoplus_\sigma \cind_{N_B(\rho)}^A(\sigma)\otimes V_\sigma$,
		where the sum ranges over $\sigma\in\Irr(N_B(\rho),C,\rho)$ 
		and each $V_\sigma$ is a finite-dimensional vector space with trivial $A$-action.
	\end{lemmaenum}

	Finally, suppose in addition that $C$ is open and has compact image in~$A/Z(A)$.
	If every element of~$A$ intertwining $\rho$ lies in~$B$,
	then the following holds.
	\begin{lemmaenum}[resume]
		\item \label{thm71c}
		For every $\sigma\in\Irr(N_B(\rho),C,\rho)$,
		the representation $\cind_{N_B(\rho)}^A(\sigma)$ is irreducible.
	\end{lemmaenum}
\end{lemma}

\begin{proof}
	For the first part, since $\Ind_C^B(\rho)$ is semisimple
	(see, e.g., \cite[Fact~A.3.2]{Kaletha-non-singular}),
	it decomposes as a finite direct sum
	\[
	\Ind_C^B(\rho) \simeq \bigoplus_\sigma \sigma\otimes V_\sigma
	\]
	where $\sigma\in\Irr(B)$ and $V_\sigma$
	is a vector space with trivial $B$-action recording the (finite) multiplicity of~$\sigma$ in $\Ind_C^B(\rho)$.
	By Frobenius reciprocity,
	$\sigma$ contributes to this direct sum if and only if
	$\sigma|_C$ contains $\rho$.
	Using Schur's Lemma (see, e.g., \cite[Section~B.II]{Renard10}), we obtain that
	\[
	\End_B(\Ind_C^B(\rho)) \simeq \bigoplus_\sigma \End_B(V_\sigma).
	\]
	The claim now follows from the fact that a finite-dimensional matrix algebra has,
	up to isomorphism, a unique irreducible representation.
	
	The second part follows from transitivity of compact induction
	together with the proof of the first part where $B$ is replaced by $N_B(\rho)$, which shows that
	\[
	\cind_C^{N_B(\rho)}(\rho) \simeq \bigoplus_\sigma \sigma\otimes V_\sigma.
	\]
	
	For the third part, we first claim that if $a\in A$ intertwines~$\sigma$, then $a \in N_B(\rho)$.
	If $a \in A$ intertwines~$\sigma$, then $a$ intertwines $\sigma|_C$ and thus~$\rho$ because $\sigma|_C$ is $\rho$-isotypic.
	Hence $a\in B$ by assumption, and therefore $a\in N_B(\rho)$ because $B$ normalizes~$C$.
	The proof is now completed using the standard intertwining criterion
	for irreducibility of a compactly-induced representation.
	This criterion is stated when $A=G(F)$ in \cite[Lemma~3.2.3]{Fintzen-IHES},
	and the proof adapts to our setting by using
	the Mackey decomposition for locally profinite groups
	(see \cite{Kutzko77,Yamamoto22}) 
	and noting that $Z(A)\subseteq B$ because $Z(A)$
	intertwines every representation of a subgroup of~$A$.
\end{proof}

\section{Commutators in simply-connected quasi-split groups}
Let $k$ be a field and let $H$ be a reductive $k$-group.
In this appendix we review a classical result of Tits
showing that if $H$ is simply-connected and quasi-split,
then $H(k)$ usually equals its own commutator subgroup,
except for some degenerate cases when $k=\bbF_2$ or $\bbF_3$
which we list in \eqref{thm83} (and the proof of \Cref{thm91}).
This result has been well-known for many years,
but we were unable to find a source that states it concisely.
In \Cref{thm81}, we take $k=\bbF_q$
and use Tits's result to study the order of
the abelianization of~$H(\bbF_q)$.

Let $H(k)^+$ be the subgroup of $H(k)$ generated by the subgroups $U(k)$
where $U$ is the unipotent radical of some parabolic subgroup of~$H$.

\begin{theorem} \label{thm84}
Suppose that $H$ is simply-connected and quasi-split and that either
\begin{equation} \label{thm83}
\begin{aligned}
|k| &\geq 4, \text{ or} \\
k &\simeq \bbF_3 \text{ and $H$ has no factor isomorphic to $\SL_2$, or} \\
k &\simeq \bbF_2 \text{ and $H$ has no factor isomorphic to $\SL_2$, $\Sp_4$,
$G_2$, or $\SU_3$.}
\end{aligned}
\end{equation}
Then $[H(k)^+,H(k)^+] = H(k)^+$.
\end{theorem}

\begin{proof}
This is explained in \cite[Section~3.4]{Tits64}.
\end{proof}

\begin{proposition} \label{thm91}
Suppose that $H$ is simply-connected and quasi-split.
\begin{enumerate}
\item If $H$ satisfies \eqref{thm83}, then $H(k) = H(k)^+$.
\item $H$ satisfies \eqref{thm83} if and only if $[H(k),H(k)] = H(k)$.
\end{enumerate}
\end{proposition}

\begin{proof}
The first part is claimed without proof in \cite[Section~1.1.2]{Tits78},
using the standard notation for the Whitehead group $W(H,k) \defeq H(k)/H(k)^+$;
see \cite{MO486102} for a proof.

For the second part, the forward implication follows from
\Cref{thm84} and the first part.
The reverse implication is proved by direct computation:
Clearly $\SL_2(\bbF_2) \simeq S_3$,
and the remaining groups are worked out,
for example, in \cite{Wilson09},
specifically Section 3.3.1 ($\SL_2(\bbF_3)$),
Section~3.5.2 ($\Sp_4(\bbF_2)\simeq S_6$),
Section~4.4.4 ($G_2(\bbF_2)$), and
Exercise~3.24 ($\SU_3(\bbF_2)$).
\end{proof}

\begin{lemma} \label{thm81}
	Let $H$ be a reductive $\bbF_q$-group.
	If $H^\tn{sc}(\bbF_q)$ has trivial abelianization, for instance, if $q>3$,
	then the abelianization of $H(\bbF_q)$ has order prime to~$q$.
\end{lemma}

\begin{proof}
	Let $\widetilde H\to H$ be a $z$-extension
	(see \cite[Section~11.4]{Kaletha-Prasad-BTbook}
	for a discussion of this notion):
	a surjective map of reductive $\bbF_q$-groups
	whose kernel is an induced torus
	and for which $\widetilde H^\tn{der}=\widetilde H^\tn{sc}$.
	Then the map $\widetilde H(\bbF_q)\to H(\bbF_q)$
	is surjective, and thus induces a surjection on abelianizations.
	So without loss of generality, after replacing $H$ by~$\widetilde H$,
	we may assume that $H^\tn{der}=H^\tn{sc}$, and hence that $H^\tn{der}(\bbF_q)$ has trivial abelianization.
	We have a short exact sequence
	\[
	\begin{tikzcd}
	1 \rar & H^\tn{der}(\bbF_q) \rar &  H(\bbF_q) \rar &
	(H/H^\tn{der})(\bbF_q) \rar & 1
	\end{tikzcd}
	\]
	in which $H/H^\tn{der}$ is a torus.
	Since $H^\tn{der}(\bbF_q)$ has trivial abelianization,
	the abelianizations of $H(\bbF_q)$ and $(H/H^\tn{der})(\bbF_q)$
	are isomorphic.
	So we are reduced to the case where $H=T$ is a torus,
	where the claim follows from the fact that
	every element of $T(\bbF_q)$ is semisimple,
	and thus has order prime to~$q$.
	That $H^\tn{sc}(\bbF_q)$ has trivial abelianization
	when $q>3$ follows from \Cref{thm91}.
\end{proof}

\section{An example in the spin group} \label{sec:ge2example}
In this section we give an example of the failure of \GEE\ which illustrates the need for Clifford theory in our construction of supercuspidal representations.
Our example shows that the dimension of $\sigma$
can be strictly larger than the dimension of $\rho \otimes \kappa$,
as \Cref{thm66} spells out.
The example is an extension of \cite[2.20~Example]{Steinberg-torsion}.

Let $G$ be the split group $\Spin_8$ over $F$,
let $\sG$ be the split reductive $\cal O_F$-group with generic fiber~$G$,
and let $\sfG$ be the special fiber of~$\sG$.
We use Bourbaki's model for the root system $\Phi(D_4)$ and its Weyl group $W(D_4)$ as in
\cite[Plate~IV]{Bourbaki-4-6}:
\begin{equation} \label{roots-D4}
\Phi(D_4) = \{\pm e_i \pm e_j \colon 1\leq i,j\leq 4,\; i\neq j\}
\end{equation}
with basis $\Delta=\{e_1 - e_2, e_2 - e_3, e_3 - e_4, e_3 + e_4\}$.
So $W(D_4) \simeq (\bbZ/2\bbZ)^3\rtimes S_4$, of order $2^6\cdot 3$.

Recall (from \cite[Proposition~3.3.3]{Carter85}, for example)
the standard bijection between maximal tori of~$\sfG$
and Frobenius-conjugacy classes in the Weyl group.
Let $\sfT$ be a maximal torus of $\sfG$
corresponding to the conjugacy class of~$-1$.
By \cite[Lemma~2.3.1]{DeBacker06} and the discussion preceding it,
there exists a maximal torus $\sT$ of~$\sG$
with special fiber~$\sfT$.
Let $T$ be the generic fiber of~$\sT$,
an elliptic, unramified maximal torus of~$G$.

\begin{lemma} \label{thm89}
$N_G(T)(F)/T(F) = W(G,T)(F) \simeq W(D_4)$.
\end{lemma}

\begin{proof}
By Lang's Theorem, $N_{\sfT}(\sfG)(k_F)/\sfT(k_F) \simeq W(\sfG,\sfT)(k_F)$,
and this latter group is the centralizer in $W(\sfG,\sfT)(\bar k_F)$
of $-1$ by \cite[Proposition~3.3.6]{Carter85},
which is the full Weyl group because $-1$ is central.
Hence $W(\sfG,\sfT)(k_F) = W(\sfG,\sfT)(\bar k_F) \simeq W(D_4)$.
Moreover, since $\sT$ is a maximal torus of the reductive $\cal O_F$-group $\sG$,
the normalizer $N_\sG(\sT)$ is smooth,
and so by Hensel's Lemma (\cite[Exposé~XI, Corollaire~1.11]{SGA3II})
the map $N_\sG(\sT)(\cal O_F)\to N_\sfG(\sfT)(k_F)$ is surjective.
Now we are finished because $N_\sG(\sT)(\cal O_F)\subseteq N_G(T)(F)$.
\end{proof}

Let $\sB(T,F) = \{x\}$.
Our example starts from a $G$-datum of the form
$\Upsilon = ((G,T),x,r,\tn{triv},\phi)$,
where $\tn{triv}$ is the trivial one-dimensional representation of $T(F)$
and where $\phi$ is a character of~$T(F)$ that we carefully define as follows.

First, observe that $T(F) \simeq U^1_{E/F}(F)\otimes_{\bZ} X_*(T)$
where $E$ is an unramified quadratic extension of $F$.
Since $\Spin_8$ is simply-connected, $X^*(T)$ is the weight lattice of type~$D_4$.
Number the fundamental weights $\varpi_i$, $1\leq i\leq 4$,
as in \cite[Plate~IV, p.~272]{Bourbaki-4-6}, so that $\varpi_2$
corresponds to the central vertex of the Dynkin diagram.

\begin{lemma} \label{thm49}
Let $q \geq 4$ and let $r$ be an odd integer with $1\leq r < \val(2)$.
\begin{enumerate}
\item
There exist two order-two characters
$\phi_i\colon U^1_{E/F}(F)\to\bbC^\times$, $i=1,2$
such that
\[
\depth(\phi_1) = \depth(\phi_2) = \depth(\phi_1\phi_2) = r.
\]

\item
For all such $\phi_i$,
the following character $\phi\colon T(F)\to\bbC^\times$
is $(G,T)$-generic of depth~$r$:
\[
\phi(t) = \phi_1(t^{\varpi_2})\cdot\phi_2(t^{\varpi_1 + \varpi_3 + \varpi_4}).
\]
\end{enumerate}
\end{lemma}

Here we use the exponential notation $t^\alpha \defeq \alpha(t)$.

\begin{proof}
	\addtocounter{equation}{-1}
	\begin{subequations} 
For the first part, consider the maximal $2$-torsion quotient
\[
U^1_{E/F}(F)/U^1_{E/F}(F)^{\times 2},
\]
where $(-)^{\times 2}$ forms the subgroup of squares.
If $a\in E^\times_0$ and $\val(a - 1) < \val(2)$, then 
\[
\val(a^2 - 1) = \val((a-1)^2 - 2(a-1)) = 2\val(a-1).
\]
Moreover, if $\val(a - 1) = 0$,
then $\val((a-1)^2) = \val(a-1)$.
Since $r$ is odd, the map
\[
U^1_{E/F}(F)_{r:r+}
\to U^1_{E/F}(F)/\bigl(U^1_{E/F}(F)^{\times 2}\cdot U^1_{E/F}(F)_{r+}\bigr)
\]
is injective,
and hence we can freely extend any character of the group~$U^1_{E/F}(F)_{r:r+}$
to an order-two character of $U^1_{E/F}(F)$.
Since the group $U^1_{E/F}(F)_{r:r+}\simeq k_F$ has cardinality $q$ and $q\geq4$,
this group has two linearly independent characters,
completing the proof.

For the second part, recall that we coordinatize $X^*(T)$ as in Bourbaki, thus
$X^*(T) = \bbZ^4 + \bbZ\varpi_4$, where $\bZ^4$ has basis $e_1, e_2, e_3, e_4$ and
where $\varpi_4 = \tfrac12(e_1 + e_2 + e_3 + e_4)$.
Since $\varpi_1 + \varpi_3 + \varpi_4 = 2e_1 + e_2 + e_3$,
we can rewrite $\phi$ as
\begin{equation}\label{equation-phi}
\phi(t) = \phi_1(t^{e_1+e_2})\cdot\phi_2(t^{e_2+e_3}).
\end{equation}
Hence the restriction of $\phi$ to depth $r$ is represented by the following element in the dual Lie algebra, using the identification of $\ft^*(F^{\sep})$ with $X^*(T) \otimes_{\bZ} F^{\sep}$:
\[
X = (e_1 + e_2) \otimes a_1 + (e_2 + e_3) \otimes a_2
\]
for some $a_i$ with $a_1, a_2$ and $a_1+a_2$ of valuation $-r$. If we denote by $\{e_i^*\}$ the basis of $\ft(F^{\sep})$ dual to $\{e_i\}$, then the set of $H_\alpha$ for $\alpha \in \Phi(G,T)$ is $\{\pm e_i^* \pm e_j^* \colon 1\leq i,j\leq 4,\; i\neq j\}$. Hence $\val(X(H_\alpha))=-r$ for all $\alpha \in \Phi(G,T)$, and since $X$ also satisfies (GE0), the character $\phi$ is $(G,T)$-generic of depth $r$.
\end{subequations}
\end{proof}

Recall that $W(D_4)\simeq(\bbZ/2\bbZ)^3\rtimes S_4$, where the subgroup $(\bbZ/2\bbZ)^3$ preserves each of the sets $\{\pm e_1\}, \{\pm e_2\}, \{\pm e_3\}, \{\pm e_4\}$.
The group $A_4$ contains a unique Sylow $2$-subgroup,
the Klein four-group, which is normal in $S_4$.
Let $P$ be the subgroup of $W(D_4)$
generated by the normal subgroup $(\bbZ/2\bbZ)^3$
and this Klein four-group.
Then $P\simeq(\bbZ/2\bbZ)^3\rtimes(\bbZ/2\bbZ)^2$
is a nonabelian group of order~$32$, normal in~$W(D_4)$. We also view $P$ as a subgroup of $N_G(T)(F)/T(F)$ via \Cref{thm89}.

\begin{lemma} \label{lemma-centralizer-expl}
The centralizer in $N_G(T)(F)/T(F)$ 
of the character $\phi$ from \Cref{thm49} is~$P$.
\end{lemma}

\begin{proof}
The normal subgroup $(\bbZ/2\bbZ)^3$ centralizes~$\phi$
because the $\phi_i$ have order two.
To see that the Klein four-group centralizes~$\phi$,
use \eqref{equation-phi} and that $\phi_1(t^{e_1+e_2})= \phi_1(t^{-e_3-e_4}) = \phi_1(t^{e_3+e_4})$ and
$\phi_2(t^{e_2+e_3}) = \phi_2(t^{e_1+e_4})$
because $e_1 + e_2 + e_3 + e_4 \in 2X^*(T)$.
\end{proof}

\begin{corollary} \label{thm65}
Let $\phi\colon T(F)\to\bbC^\times$ be as in \Cref{thm49}
and let $N_G(T)(F)_P$ be the preimage of $P$ under the projection to $W(G,T)$.
Then
\[
Z_{G(F)_{[x]}}(\hat\phi_{(G,x)}) = G(F)_{x,r/2} \cdot N_G(T)(F)_P.
\]
\end{corollary}
\begin{proof} Since $\{x\}=\sB(T,F)$, we have $N_G(T)(F) \subset G(F)_{[x]}$. Now the result follows from combining \Cref{thm80} and \Cref{lemma-centralizer-expl}.
\end{proof}

Now we are in position to provide an example in which $\dim(\sigma)>\dim(\rho \otimes \kappa)$.

\begin{example} \label{thm66}
Consider the supercuspidal $G$-datum $((G,T),x,r,1,\phi)$ where $\phi$ and $r$ are as in \Cref{thm49},
so that $K=G(F)_{x,r/2}T(F)$, $\Kplus=G(F)_{x,r/2}N_G(T)(F)$,  $\rho\otimes\kappa=\hat\phi_{(G,x)}$ and
$N_{\Kplus}(\hat\phi_{(G,x)}) = G(F)_{x,r/2} \cdot N_G(T)(F)_P$
by \Cref{thm65}.
Let $\sigma$ be a representation of~$N_{\Kplus}(\hat\phi_{(G,x)})$
that is $\hat\phi_{(G,x)}$-isotypic on~$K$.
Then $\sigma$ is determined by its restriction $\sigma_0$
to~$N_G(T)(F)_P$. 
Conversely, any representation~$\sigma_0$ of $N_T(G)(F)_P$
that is $\phi$-isotypic on~$T(F)$
determines such a~$\sigma$ by the formula
$\sigma(kn) = \hat\phi_{(G,x)}(k)\sigma_0(n)$
for $k\in G(F)_{x,r}$ and $n\in N_G(T)(F)_P$.

Since $\sigma_0$ and~$\phi$ are trivial on~$\ker(\phi)$,
we can interpret them as representations of the finite group
$\widetilde P \defeq N_G(T)(F)_P/\ker(\phi)$.
Moreover, as $\phi$ has order 2, we have $T(F)/\ker(\phi)\simeq\bbZ/2\bbZ$, and
the group~$\widetilde P$ fits into a short exact sequence
\[
1 \to \bbZ/2\bbZ \to \widetilde P \to P \to 1.
\]
Choosing $\sigma_0$ amounts to choosing
a representation of~$\widetilde P$
whose restriction to~$\bbZ/2\bbZ$ is nontrivial. 
Hence there certainly exists a $\sigma_0$
for which $\dim(\sigma_0) > \dim(\phi) = 1$,
since $P$ is nonabelian, thus $\dim(\sigma) > \dim(\rho \otimes \kappa)$.
Moreover, if we choose two $\sigma_0$ and $\sigma_0'$
with $\dim(\sigma_0)\neq\dim(\sigma_0')$,
then the resulting supercuspidal representations
$\pi$ and~$\pi'$ are not isomorphic
because their formal degrees are different:
\[
\fdeg(\pi) = \frac{\dim(\sigma_0)}{\vol(N_{\Kplus}(\hat\phi_{(G,x)}))}
\neq \frac{\dim(\sigma_0')}{\vol(N_{\Kplus}(\hat\phi_{(G,x)}))} = \fdeg(\pi').
\]
See, e.g., Section~3.2 of \cite{schwein24} for a review of the formal degree,
including, in Lemma~3.4, the formula for the formal degree
of a compactly induced representation used above.
\end{example}

\clearpage
\printindex[notation]
\printindex[terminology]
\clearpage

\addcontentsline{toc}{section}{References}
\bibliography{bib.bib}

\end{document}